\documentclass[11pt,a4paper]{article}

\usepackage{amsmath,amssymb,amsthm}
\usepackage{bbm,color,dcolumn,enumerate}
\usepackage{fancybox,fancyhdr,graphicx,multirow}
\usepackage{psfrag,pgfplots,relsize,subcaption,times}
\usepackage{tikz,tikz-3dplot}
\usepackage[affil-it]{authblk}
\usepackage[sc]{mathpazo}  
\usepackage{setspace}  
\usepackage[sort,numbers]{natbib}

\usepackage{array}  
\newcolumntype{L}[1]{>{\raggedright\let\newline\\\arraybackslash\hspace{0pt}}m{#1}}
\newcolumntype{C}[1]{>{\centering\let\newline\\\arraybackslash\hspace{0pt}}m{#1}}
\newcolumntype{R}[1]{>{\raggedleft\let\newline\\\arraybackslash\hspace{0pt}}m{#1}}

\usepackage{environ}
\NewEnviron{eq}{%
\begin{equation}\begin{split}
 \BODY
\end{split}\end{equation}
}

\usepackage[linktocpage=true]{hyperref}
\hypersetup{
 colorlinks,
 linkcolor=blue,          
 citecolor=red,       
 filecolor=black,   
 urlcolor=black, 
 pdftitle={},
 pdfauthor={},
 pdfcreator={},
 pdfsubject={},
 pdfkeywords={}
}

\pgfplotsset{compat=1.14} 
\tikzstyle{doublearr} = [latex-latex,red, line width=0.5pt]
\tikzstyle{doublearr2} = [latex-latex,green!80!black, line width=0.5pt]
\tikzstyle{mybox} = [draw=black, thick, minimum height=.6cm, rectangle, text centered]
\tikzstyle{fancytitle} = [fill=blue, text=white]

\def\ee{{\rm e}}

\def\E{{\mathbb{E}}}
\def\P{{\mathbb{P}}}
\def\R{{\mathbb{R}}}

\newcommand{\dd}{\boldsymbol{\delta}}

\newcommand{\N}{\mathbbm{N}}

\newcommand{\GG}{\mathbf{G}}

\newcommand{\qq}{\mathbf{q}}
\newcommand{\QQ}{\mathbf{Q}}

\newcommand{\cS}{\mathcal{S}}

\newcommand{\CJSQ}{\mathrm{ CJSQ}}

\newcommand{\ERRG}{\mathrm{ERRG}}
\newcommand{\full}{\mathrm{full}}
\newcommand{\idle}{\mathrm{idle}}
\newcommand{\JSQ}{ \mathrm{ JSQ}}
\newcommand{\MJSQ}{\mathrm{ MJSQ}}

\newcommand{\e}{{\rm e}}
\newcommand{\oo}{{\rm o}}
\newcommand{\OO}{{\rm O}}

\newcommand{\jsq}{\rm JSQ}

\newcommand{\bQ}{\bar{Q}}

\newcommand{\hQ}{\hat{Q}}

\newcommand{\dis}{\text{\textsc{dis}}}
\newcommand{\com}{\text{\textsc{com}}}

\newcommand{\dif}{\ensuremath{\mbox{d}}}  

\newcommand{\pto}{\ensuremath{\xrightarrow{\mathbbm{P}}}}  

\newcommand{\expect}[1]{{\mathbb E}[#1]}
\newcommand\ind[1]{\ensuremath{\mathbbm{1}_{\left[#1\right]}}} 

\newcommand\Pro[1]{\mathbbm{P}\left(#1\right)} 

\providecommand{\href}[2]{#2}

\newtheorem{theorem}{Theorem}[section]

\newtheorem{corollary}[theorem]{Corollary}
\newtheorem{definition}[theorem]{Definition}

\newtheorem{lemma}[theorem]{Lemma}

\newtheorem{proposition}[theorem]{Proposition}
\newtheorem{remark}[theorem]{Remark}

\theoremstyle{definition}

\definecolor{mygray}{rgb}{0.7,0.7,0.7}

\definecolor{col1}{rgb}{0.8, 0.01568, 0.}
\definecolor{col2}{rgb}{0.9372549019607843, 0.62745, 0.1}
\definecolor{col3}{rgb}{0.9686274509803922, 0.87, 0.}
\definecolor{col4}{rgb}{0.3176470588235294, 0.58, 0.0784313}
\definecolor{col5}{rgb}{0.22745098, 0.2392156862,0.43}
\definecolor{col6}{rgb}{0.5019607843137255, 0.16, 0.4470}

\makeatletter
\renewcommand{\fnum@figure}{\small\textbf{\figurename~\thefigure}}
\makeatother

\makeatletter
\renewcommand{\fnum@table}{\small\textbf{\tablename~\thetable}}
\makeatother

\numberwithin{equation}{section}
\numberwithin{theorem}{section}

\graphicspath{{./},{./images/}}

\begin{document}

\title{Scalable load balancing in networked systems: \\
A survey of recent advances\footnote{This survey extends the short review presented at ICM 2018~\cite{BBLM18}, and Section 6 provides a synopsis of a SIGMETRICS 2018 conference paper published in the Proceedings of the ACM on Measurement and Analysis of Computing Systems~\cite{MBL17}}
\footnote{The work is partially supported by the National Science Foundation under Grant No.~2113027 and The Netherlands Organization for Scientific Research (NWO) [Gravitation Grant NETWORKS 024.002.003].}}
\author[1]{Mark van der Boor}
\author[1]{Sem C.~Borst}
\author[1,2]{\\ Johan S.H.~van Leeuwaarden,}
\author[3]{Debankur Mukherjee}
\affil[1]{Eindhoven University of Technology, The Netherlands}
\affil[2]{Tilburg University, The Netherlands}
\affil[3]{Georgia Institute of Technology, Atlanta GA, USA}

\date{\today}

\maketitle

\begin{abstract}{
In this survey we provide an overview of recent advances on scalable load balancing schemes which provide favorable delay performance and yet require minimal implementation overhead. The basic load balancing scenario involves a single dispatcher where tasks arrive that must immediately be forwarded to one of $N$~single-server queues. The Join-the-Shortest-Queue (JSQ) policy yields vanishing delays as $N$ grows large, as in a centralized queueing arrangement, but involves a prohibitive communication burden. In contrast, JSQ($d$) schemes that assign an incoming task to a server with the shortest queue among $d$~servers selected uniformly at random require little communication, but lead to constant delays. In order to examine this fundamental trade-off between delay performance and implementation overhead, we discuss a body of recent research on JSQ($d(N)$) schemes where the diversity parameter $d(N)$ depends on~$N$ and investigate the growth rate of $d(N)$ required to match the optimal JSQ performance on fluid and diffusion scale.
 
Stochastic coupling techniques and scaling limits play an instrumental role in establishing this asymptotic optimality. We demonstrate how this methodology carries over to infinite-server settings, finite buffers, multiple dispatchers, servers arranged on graph topologies, and token-based load balancing schemes such as Join-the-Idle-Queue (JIQ), thus providing a broad overview of the main trends in the field.}
\end{abstract}

\newpage

\setcounter{tocdepth}{2} 
{\small \tableofcontents}

\newpage

\section{Introduction}
\label{sec:intro}

In this survey we review scalable load balancing algorithms (LBAs)
which achieve excellent delay performance in large-scale systems
and yet have a low implementation overhead.
LBAs play a critical role in distributing service requests or tasks
(e.g.~computing jobs, data base look-ups, file transfers) among servers
or distributed resources in parallel-processing systems.
The analysis and design of LBAs has attracted significant attention in recent
years, mainly spurred by crucial scalability challenges arising
in cloud networks and data centers with massive numbers of servers.
LBAs can be broadly categorized as static, dynamic, or some intermediate
blend, depending on the amount of feedback or state information
(e.g.~congestion levels) that is used in allocating tasks.
The use of state information naturally allows dynamic policies
to achieve better delay performance, but also involves higher
implementation complexity and a substantial communication burden.
The latter issue is particularly pertinent in cloud networks and data centers
with immense numbers of servers handling a huge influx of service requests.
In order to understand the large-system characteristics, we examine
scalability properties through the prism of asymptotic scalings
where the system size grows large, and identify LBAs which strike
a balance between delay performance and implementation overhead.

The most basic load balancing scenario consists of $N$~identical
parallel servers and a dispatcher where tasks arrive sequentially.
Arriving tasks must immediately be forwarded to one of the servers.
Tasks are assumed to have unit-mean exponentially distributed service
requirements, and the service discipline at each server is supposed
to be oblivious to the actual service requirements.
These assumptions, in conjunction with a Poisson arrival process,
permit a Markovian state description for the evolution
of the queue length process.
Moreover, the symmetry arising from the homogeneity of tasks
and exchangeability of the servers provides a particularly convenient
basis for stochastic coupling arguments and scaling limits.
In the early parts of this survey we will focus on this basic setup
which has been prevalent in the literature, but in later sections
of the paper we will also discuss graph-based versions
where the servers are no longer statistically identical.
In addition, we will touch on scenarios with heterogeneous tasks,
extensions to general service requirement distributions and situations
where advance knowledge of the service requirements is available.

In the above-described basic setup, the celebrated Join-the-Shortest-Queue
(JSQ) policy has several important stochastic optimality properties.
In particular, the JSQ policy achieves the minimum mean overall
delay among all non-anticipating policies that do not have any
advance knowledge of the service requirements \cite{EVW80,Winston77}.
In order to implement the JSQ policy however, a dispatcher requires
instantaneous knowledge of all the queue lengths, which may involve
a prohibitive communication burden with a large number of servers~$N$.
This poor scalability has motivated consideration of JSQ($d$) policies,
where an incoming task is assigned to a server with the shortest queue
among $d \geq 2$ servers selected uniformly at random.
Note that this involves an exchange of $2 d$ messages per task,
irrespective of the number of servers~$N$.
Seminal results in~\cite{Mitzenmacher01,VDK96} imply that even sampling
as few as $d = 2$ servers yields significant performance enhancements
over purely random assignment ($d = 1$) as $N$ grows large, which is commonly
referred to as the \emph{power-of-two} or \emph{power-of-choice} effect.
In particular, when tasks arrive at rate $\lambda N$, the queue length
distribution at each individual server exhibits super-exponential
decay for any fixed $\lambda < 1$ as $N$ grows large, a considerable
improvement compared to exponential decay for purely random assignment.

The diversity parameter~$d$ thus induces a fundamental trade-off
between the amount of communication overhead and the delay performance.
Specifically, a random assignment policy does not entail any
communication burden, but the mean waiting time remains \emph{constant}
as $N$ grows large for any fixed $\lambda > 0$.
In contrast, a nominal implementation of the JSQ policy (without
maintaining state information at the dispatcher) involves $2 N$
messages per task, but the mean waiting time \emph{vanishes}
as $N$ grows large for any fixed $\lambda < 1$.
Although JSQ($d$) policies with $d \geq 2$ yield major performance
improvements over purely random assignment while reducing the
communication burden by a factor O($N$) compared to the JSQ policy,
the mean waiting time \emph{does not vanish} in the limit.
Hence, no fixed value of~$d$ will provide asymptotically optimal delay
performance.
This is evidenced by powerful results~\cite{GTZ16,GTZ18,GTZ20}
indicating that in the absence of any memory at the dispatcher
the communication overhead per task \emph{must increase} with~$N$
in order for any scheme to achieve a zero mean waiting time in the limit.

In the context of JSQ($d$) policies, scalability specifically
pertains to the intrinsic trade-off between delay performance
and communication overhead as governed by the diversity parameter~$d$,
in conjunction with the relative load~$\lambda$.
In this survey we review scaling results which offer detailed
insight in the latter trade-off in a regime where not only
the overall task arrival rate is assumed to grow with~$N$,
but also the diversity parameter is allowed to depend on~$N$.
We write $\lambda(N)$ and $d(N)$ to explicitly reflect that,
and provide a sketch of the analysis in~\cite{MBLW16-3}
which identifies the growth rate of $d(N)$ required in order
to achieve a zero mean waiting time in the limit,
depending on the scaling of $\lambda(N)$.
This involves both fluid-scaled and diffusion-scaled versions
of the queue length process in regimes
where $\lambda(N) / N \to \lambda < 1$
and $(N - \lambda(N)) / \sqrt{N} \to \beta > 0$ as $N \to \infty$,
respectively, see Section~\ref{defi} for definitions of these objects.
As we will be discussed in detail there, the limiting processes
are insensitive to the exact growth rate of $d(N)$,
as long as the latter is sufficiently fast,
and in particular coincide with the limiting processes for the JSQ policy.
This demonstrates that the optimality of the JSQ policy can asymptotically
be preserved while dramatically lowering the communication overhead.

As mentioned above, we will also consider network scenarios
where the $N$~servers are assumed to be inter-connected
by some underlying graph topology $G_N$.
Tasks arrive at the various servers as independent Poisson processes
of rate~$\lambda$, and each incoming task is assigned to whichever server
has the shortest queue amongst the one where it appears and its neighbors
in $G_N$.
Such network scenarios are not only of theoretical interest,
but also of major practical relevance since they emerge in modeling
so-called affinity relations and compatibility constraints
between tasks and servers.
Such features are increasingly common in data centers and cloud networks
due to heterogeneity and data locality issues,
and also relate to scalability considerations, see Section~\ref{networks}
for a further discussion and specific literature pointers.
In case $G_N$ is a clique (fully connected graph), each incoming task
is assigned to the server with the shortest queue across the entire system,
and the behavior is equivalent to that under the JSQ policy.
The stochastic optimality properties of the JSQ policy thus imply that
the queue length process in a clique will be `better'
than in an arbitrary graph~$G_N$.
We will present sufficient conditions formulated in~\cite{MBL17}
for the fluid-scaled and diffusion-scaled versions
of the queue length process in an arbitrary graph to be equivalent
to the limiting processes in a clique as $N \to \infty$.
The conditions demonstrate that the optimality of a clique can
be asymptotically preserved while dramatically reducing the number
of connections, provided the graph $G_N$ is `suitably random',
see Section~\ref{networks} for a more formal statement.

While a zero waiting time can be achieved in the limit by sampling
only $d(N) \ll N$ servers, the amount of communication overhead
in terms of $d(N)$ must still grow with~$N$.
This may be explained from the fact that a large number of servers
need to be sampled for each incoming task to ensure that at least one
of them is found idle with high probability.
This can be avoided by introducing memory at the dispatcher,
in particular maintaining a record of vacant servers,
and assigning tasks to idle servers, if there are any.
This so-called Join-the-Idle-Queue (JIQ) scheme \cite{BB08,LXKGLG11}
has gained huge popularity recently, and can be implemented through
a simple token-based mechanism generating at most one message per task.
The JIQ scheme is thus quite appealing from a scalability perspective,
which raises the question what the corresponding delay performance is
in large-scale systems.
We will therefore also review results implying that not only the
fluid-scaled queue length process under the JIQ scheme asymptotically
coincides with that under the JSQ policy as shown in~\cite{Stolyar15},
but that this equivalence property extends to the diffusion-scaled
queue length process as established in~\cite{MBLW16-1}.
Thus, the use of memory allows the JIQ scheme to achieve asymptotically
optimal delay performance with minimal communication overhead
(at least in the idealized setting with statistically identical servers
and homogeneous tasks).
In particular, ensuring that tasks are assigned to idle servers
whenever available is sufficient to achieve asymptotic optimality,
and using any additional queue length information yields no meaningful
performance benefits on the fluid or diffusion levels.
It is worth pointing out though that the JIQ scheme is \emph{not} optimal
in certain asymptotic regimes such as the non-degenerate slow-down (NDS) regime,
see Section~\ref{asym} for a formal definition.
In~\cite{GW19} it was shown that a minor modification of the JIQ scheme,
called Idle-One-First, which besides idle servers also keeps track
of queues of length one is asymptotically optimal,
see Section~\ref{nondegenerate} for a detailed discussion.

On a methodological note, it is worth observing that a direct derivation
of the fluid limits and diffusion limits in the above scenarios
is quite challenging.
Instead, the asymptotic equivalence results in \cite{MBLW16-1,MBLW16-3,MBL17}
are derived by relating the relevant system occupancy processes to the
corresponding processes under a JSQ policy, and showing that the deviation
between these processes is asymptotically negligible on either fluid scale
or diffusion scale under suitable assumptions on $d(N)$ or $G_N$.
The known fluid and diffusion limits for the JSQ policy thus yield
the corresponding limit process for the JSQ($d(N)$) policy,
a load balancing graph $G_N$ and the JIQ scheme as by-products.

In this survey we highlight the stochastic coupling techniques that
played an instrumental role in proving the asymptotic equivalence results
in \cite{MBLW16-1,MBLW16-3,MBL17}.
Although the stochastic coupling concepts provide an effective
and overarching approach, they defy a systematic recipe and involve
some degree of ingenuity and customization.
Indeed, the specific coupling arguments that were developed
in \cite{MBLW16-1,MBLW16-3,MBL17} are different from those that were originally
used in establishing the stochastic optimality properties of the JSQ policy.
Moreover, the specific coupling approaches differ in sometimes subtle
but critical ways between a JSQ($d(N)$) policy, a load balancing graph $G_N$
and the JIQ scheme, which all require the arguments to be suitably tailored.
We also review further stochastic coupling constructions that were devised
in~\cite{MBLW16-4} for scenarios with infinite-server dynamics.

While the results for load balancing graphs illustrate that the
stochastic coupling techniques can be applied in `asymmetric' situations,
it is fair to say that this approach is at its strongest in scenarios
where all the servers are exchangeable, and the evolution of the system
occupancy can be represented in terms of a density-dependent Markov process.
In these cases, the approach is particularly powerful in analyzing the
system occupancy process on fluid or diffusion scale, where for many
policies the behavior can be shown to asymptotically coincide with that
of JSQ, for which fairly explicit characterizations are known.

Stochastic coupling does not seem to provide a directly useful approach
for other functionals of the system occupancy process,
such as the the maximum queue length, where asymptotic equivalence
with JSQ on fluid or diffusion scale does not provide any information,
and in fact even asymptotically the behavior for many schemes is different.
Applying stochastic coupling techniques in highly heterogeneous settings
is also difficult since the lack of symmetry tends to break its underpinnings,
and establishing scaling results for such scenarios remains
as a particularly challenging subject for further research,
as further discussed in Section~\ref{sec:ext}.

A final caveat is in order.
Load balancing is a broad subject which has been actively pursued
for decades and has been investigated from a variety of perspectives
in several communities (algorithm design, applied probability,
complexity theory, performance evaluation).
While this survey aims to touch on many of these aspects, reflect
historical developments and connect various threads, it is impossible
to exhaustively cover the load balancing literature in full detail.
Rather than provide an encyclopedic overview, we therefore focus
on scalability in terms of delay performance and implementation overhead
in large-scale systems as the overarching theme, and highlight the
combined power of stochastic coupling methods and scaling limits.

The survey is organized as follows.
In Section~\ref{spec} we discuss various LBAs and evaluate their
scalability properties.
In Section~\ref{sec:powerofd} we introduce some useful preliminary
concepts, and then review fluid and diffusion limits for the JSQ
policy as well as JSQ($d$) policies with a fixed value of~$d$.
In Section~\ref{univ} we discuss the trade-off between delay
performance and communication overhead as a function of the diversity
parameter~$d$, in conjunction with the relative load.
In particular, we formulate asymptotic universality properties
for JSQ($d$) policies, which are extended to systems with server pools
and network scenarios in Sections~\ref{bloc} and~\ref{networks},
respectively.
Section~\ref{token} is devoted to asymptotic optimality properties
for the JIQ scheme.
We discuss somewhat related redundancy policies and alternative scaling
regimes and performance metrics in Section~\ref{miscellaneous}.
The survey is concluded in Section~\ref{sec:ext} with a discussion
of yet further extensions and several open problems and emerging
research directions.

\section{Scalability spectrum}
\label{spec}

In this section we review a wide spectrum of LBAs and examine
scalability properties in terms of their delay performance vis-a-vis
associated implementation overhead in large-scale systems.

\subsection{Basic model}

Throughout this section and most of the paper, we focus on a basic
scenario with $N$~parallel single-server infinite-buffer queues
and a single dispatcher where tasks arrive as a Poisson process
of rate~$\lambda(N)$, as depicted in Figure~\ref{figJSQ}.
Arriving tasks cannot be queued at the dispatcher,
and must immediately be forwarded to one of the servers.
This canonical setup is commonly dubbed the \emph{supermarket model},
in loose analogy with the daily-life situation of choosing
between parallel check-out lanes in supermarkets.
Tasks are assumed to have unit-mean exponentially distributed service
requirements, and the service discipline at each server is supposed
to be oblivious to the actual service requirements.

When tasks do not get served and never depart but simply accumulate,
the above setup corresponds to a so-called balls-and-bins model,
and we will further elaborate on the connections and differences
with work in that domain in Section~\ref{ballsbins}.

\begin{figure}
\begin{minipage}{\textwidth}
\begin{center}
\hfill\parbox{0.48\textwidth}{\centering
\begin{tikzpicture}[scale=0.5]
\node[above] at (3.750,5) {\small$\lambda(N)$};
\draw[thick,black,->] (2.75,5)--(4.75,5);
\draw[thick,black,fill=black](5,5) circle [radius=0.1];
\draw[thick,black,->,dash pattern=on 2 off 1] (5,5)--(7,5);
\draw[thick,black,->,dash pattern=on 2 off 1] (5,5)--(7,6.5);
\draw[thick,black,->,dash pattern=on 2 off 1] (5,5)--(7,8);
\draw[thick,black,->,dash pattern=on 2 off 1] (5,5)--(7,3.5);
\draw[thick,black,->,dash pattern=on 2 off 1] (5,5)--(7,2);
\foreach \i in {1,...,4}
{
\draw[thick,black,fill=mygray](13-\i,8) circle [radius=0.4];
}
\foreach \i in {1,...,2}
{
\draw[thick,black,fill=mygray](13-\i,6.5) circle [radius=0.4];
}
\foreach \i in {1,...,3}
{
\draw[thick,black,fill=mygray](13-\i,5) circle [radius=0.4];
}
\foreach \i in {1,...,5}
{
\draw[thick,black,fill=mygray](13-\i,2) circle [radius=0.4];
}
\draw[thick,black](13.2,8) circle [radius=0.5];
\draw[thick,black](13.2,6.5) circle [radius=0.5];
\draw[thick,black](13.2,5) circle [radius=0.5];
\draw[thick,black](13.2,2) circle [radius=0.5];
\draw[thick,black,->] (13.7,8)--(14.5,8);
\draw[thick,black,->] (13.7,6.5)--(14.5,6.5);
\draw[thick,black,->] (13.7,5)--(14.5,5);
\draw[thick,black,->] (13.7,2)--(14.5,2);
\node at (13.2,8) {\small 1};
\node at (13.2,6.5) {\small 2};
\node at (13.2,5) {\small 3};
\node at (13.2,3.5) {$\vdots$};
\node at (13.2,2) {\small $N$};
\end{tikzpicture}
}
\parbox{0.48\textwidth}{\centering
\begin{tikzpicture}[scale=.60]
\foreach \x in {10, 9,...,1}
    \draw (\x,6)--(\x,0)--(\x+.7,0)--(\x+.7,6);
\foreach \x in {10, 9,...,1}
    \draw (\x+.35,-.45) node[circle,inner sep=0pt, minimum size=10pt,draw,thick] {{{\tiny $\mathsmaller{\x}$}}} ;
\foreach \y in {0, .5}
    \draw[fill=mygray,mygray] (1.15,.1+\y) rectangle (1.55,.5+\y);
\foreach \y in {0, .5, 1, 1.5}
    \draw[fill=mygray,mygray] (2.15,.1+\y) rectangle (2.55,.5+\y);
\foreach \y in {0, .5, 1, 1.5}
    \draw[fill=mygray,mygray] (3.15,.1+\y) rectangle (3.55,.5+\y);
\foreach \y in {0, .5, 1, 1.5, 2, 2.5}
    \draw[fill=mygray,mygray] (4.15,.1+\y) rectangle (4.55,.5+\y);
\foreach \y in {0, .5, 1, 1.5, 2, 2.5, 3}
    \draw[fill=mygray,mygray] (5.15,.1+\y) rectangle (5.55,.5+\y);
\foreach \y in {0, .5, 1, 1.5, 2, 2.5, 3, 3.5, 4}
    \draw[fill=mygray,mygray] (6.15,.1+\y) rectangle (6.55,.5+\y);
\foreach \y in {0, .5, 1, 1.5, 2, 2.5, 3, 3.5, 4}
    \draw[fill=mygray,mygray] (7.15,.1+\y) rectangle (7.55,.5+\y);
\foreach \y in {0, .5, 1, 1.5, 2, 2.5, 3, 3.5, 4, 4.5, 5}
    \draw[fill=mygray,mygray] (8.15,.1+\y) rectangle (8.55,.5+\y);
\foreach \y in {0, .5, 1, 1.5, 2, 2.5, 3, 3.5, 4, 4.5, 5}
    \draw[fill=mygray,mygray] (9.15,.1+\y) rectangle (9.55,.5+\y);
\foreach \y in {0, .5, 1, 1.5, 2, 2.5, 3, 3.5, 4, 4.5, 5}
    \draw[fill=mygray,mygray] (10.15,.1+\y) rectangle (10.55,.5+\y);

\draw[thick] (.9,0) rectangle (10.8,.5);
\draw[thick] (.9,.5) rectangle (10.8,1);
\draw[thick] (3.9,2.5) rectangle (10.8,3);

\draw  (12, .2) node {{\scriptsize $\leftarrow Q_1=10$}};
\draw  (12, .9) node {{\scriptsize $\leftarrow Q_2=10$}};

\draw  (11.85, 1.3) node {{\tiny $\cdot$}};
\draw  (11.85, 1.8) node {{\tiny $\cdot$}};
\draw  (11.85, 2.3) node {{\tiny $\cdot$}};
\draw  (12, 2.7) node {{\scriptsize $\leftarrow Q_6=7$}};
\draw  (11.85, 3.3) node {{\tiny $\cdot$}};
\draw  (11.85, 3.8) node {{\tiny $\cdot$}};
\draw  (11.85, 4.3) node {{\tiny $\cdot$}};
\end{tikzpicture}
}\\[1ex]
\parbox{0.48\textwidth}{\centering
\captionof{figure}{Tasks arrive at the dispatcher as a Poisson process
of rate $\lambda(N)$, and are forwarded to one of the $N$ servers
according to some specific load balancing algorithm.
\label{figJSQ}}
}
\hfill\parbox{0.48\textwidth}{\centering
\captionof{figure}{The value of $Q_i$ represents the width of the
$i$-th row, when the servers are arranged in non-descending order
of their queue lengths.
\label{figB}}
}
\end{center}
\end{minipage}
\end{figure}

\subsection{Asymptotic scaling regimes}
\label{asym}

An exact analysis of the delay performance is quite involved,
if not intractable, for all but the simplest LBAs.
Numerical evaluation or simulation are not straightforward either,
especially for high load levels and large system sizes.
A common approach is therefore to consider various limit regimes,
which not only provide mathematical tractability and illuminate the
fundamental properties, but are also natural in view of the typical
conditions in which cloud networks and data centers operate.
One can distinguish several asymptotic scalings that have been used
for these purposes:

\begin{enumerate}[{\normalfont (i)}]
\item In the classical heavy-traffic regime, $\lambda(N) = \lambda N$
with a fixed number of servers~$N$ and a relative load~$\lambda$
that tends to one in the limit.
\item In the conventional large-capacity or many-server regime,
the relative load $\lambda(N) / N$ approaches a constant $\lambda < 1$
as the number of servers~$N$ grows large.
\item The popular Halfin-Whitt regime, named after the authors of the
seminal paper~\cite{HW81} where this was introduced and first analyzed,
combines heavy traffic with a large capacity, with
\begin{equation}
\label{eq:HW}
\frac{N - \lambda(N)}{\sqrt{N}} \to \beta > 0 \mbox{ as } N \to \infty,
\end{equation}
so the relative capacity slack behaves as $\beta / \sqrt{N}$
as the number of servers~$N$ grows large.
\item The so-called non-degenerate slow-down regime~\cite{Atar12,GW19}
involves $N - \lambda(N) \to \gamma > 0$, so the relative capacity
slack shrinks as $\gamma / N$ as the number of servers~$N$ grows large.
\end{enumerate}

The term non-degenerate slow-down refers to the fact that
in the context of a centralized multi-server queue (where load balancing
between servers occurs implicitly), the mean waiting time
in regime (iv) tends to a strictly positive constant as $N \to \infty$,
and is thus of similar magnitude as the mean service requirement.
In contrast, in regimes (ii) and (iii), the mean waiting time
in a multi-server queue decays exponentially fast in~$N$
or is of the order~$1 / \sqrt{N}$, respectively as $N \to \infty$,
while in regime~(i) the mean waiting time grows arbitrarily large
relative to the mean service requirement.

In the context of a centralized M/M/$N$ queue, scalings (ii), (iii)
and (iv) are commonly referred to as Quality-Driven (QD),
Quality-and-Efficiency-Driven (QED) and Efficiency-Driven (ED) regimes.
These terms reflect that (ii) offers excellent service quality
(vanishing waiting time), (iv) provides high resource efficiency
(utilization approaching one), and (iii) achieves a combination
of these two, providing the best of both worlds.

In the remainder of the paper we will focus on scalings (ii) and (iii),
and refer to these as fluid and diffusion scalings, since it is natural
to analyze the relevant system occupancy processes on fluid scale ($1 / N$)
and diffusion scale ($1 / \sqrt{N}$) in these regimes, respectively.
In line with the central theme of this survey, we will not provide
a detailed account of scalings~(i) and (iv), which do not capture the
large-scale perspective and do not allow for low delays, respectively.
However, we will briefly mention some results for these regimes
in Sections~\ref{classical} and~\ref{nondegenerate}.

\subsection{Basic load balancing algorithms}

\subsubsection{Random assignment: $N$ independent M/M/1 queues}
\label{random}

One of the most basic LBAs is to assign each arriving task to a server
selected uniformly at random.
In that case, the various queues collectively behave as
$N$~independent M/M/1 queues, each with arrival rate $\lambda(N) / N$
and unit service rate.
In particular, at each of the queues, the total number of tasks
in stationarity has a geometric distribution with parameter
$\lambda(N) / N$.
By virtue of the PASTA property, the probability that an arriving task
incurs a non-zero waiting time is $\lambda(N) / N$.
The mean number of waiting tasks (excluding the possible task in service)
at each of the queues is $\frac{\lambda(N)^2}{N (N - \lambda(N))}$,
so the total mean number of waiting tasks is
$\frac{\lambda(N)^2}{N - \lambda(N)}$, which by Little's law implies
that the mean waiting time is $\frac{\lambda(N)}{N - \lambda(N)}$.
In particular, when $\lambda(N) = N \lambda$, the probability that
a task incurs a non-zero waiting time is $\lambda$,
and the mean waiting time of a task is $\frac{\lambda}{1 - \lambda}$,
independent of~$N$, reflecting the independence of the various queues.

As we will see later, a broad range of queue-aware LBAs can deliver
a probability of a non-zero waiting time and a mean waiting time that
vanish asymptotically.
While a random assignment policy is evidently not competitive
with such queue-aware LBAs, it still plays a relevant role due to the
strong degree of mathematical tractability.
For example, the queue process under purely random assignment can be
shown to provide an upper bound (in a stochastic majorization sense)
for various more involved queue-aware LBAs for which even stability
may be difficult to establish directly, yielding conservative
performance bounds and stability guarantees.

A slightly better LBA is to assign tasks to the servers
in a Round-Robin manner, dispatching every $N$-th task to the same server.
In the fluid regime~(ii), the inter-arrival time of tasks at each given
queue will then converge to a constant $1 / \lambda$ as $N \to \infty$.
Thus each of the queues will behave as a D/M/1 queue in the limit,
and the probability of a non-zero waiting time and the mean waiting
time will be somewhat lower than under purely random assignment.
However, both the probability of a non-zero waiting time and the mean
waiting time will still tend to strictly positive values and not vanish
as $N \to \infty$.

\subsubsection{Join-the-Shortest Queue (JSQ)}
\label{ssec:jsq}

Under the Join-the-Shortest-Queue (JSQ) policy, each arriving task is
assigned to the server with the currently shortest queue
(ties are broken arbitrarily).
In the basic model described above, the JSQ policy has several
stochastic optimality properties, and yields the `most balanced
and smallest' queue process among all non-anticipating policies
that do not have any advance knowledge of the service requirements
\cite{EVW80,Winston77}.

\subsubsection{Join-the-Smallest-Workload (JSW): centralized M/M/N queue}
\label{ssec:jsw}

Under the Join-the-Smallest-Workload (JSW) policy, each arriving task
is assigned to the server with the currently smallest workload.
Note that this is an anticipating policy, since it requires advance
knowledge of the service requirements of all the tasks in the system.
Further observe that this policy (myopically) minimizes the waiting
time for each incoming task, and mimicks the operation of a centralized
$N$-server queue with a FCFS discipline.
The equivalence with a centralized $N$-server queue with a FCFS
discipline yields an additional optimality property of the JSW policy:
The vector of joint workloads at the various servers observed
by each incoming task is smaller in the Schur convex sense
than under any alternative admissible policy~\cite{FC01}.

It is worth observing that the above optimality properties in fact
do not rely on Poisson arrival processes or exponential service
requirement distributions.
At the same time, these optimality properties do not imply that
the JSW policy minimizes the mean stationary waiting time.
In our setting with Poisson arrivals and exponential service
requirements, however, it can be shown through direct means that the
total number of tasks under the JSW policy is stochastically smaller
than under the JSQ policy.
Indeed, in view of the equivalence with a centralized M/M/$N$ queue,
the total service completion rate under the JSW policy is given
by $\min\{L, N\}$ when there are $L$~tasks in total in the system,
while under the JSQ policy the total service completion rate is at most
equal to $\min\{L, N\}$, and may be lower than that when some servers
are idle while tasks are queued up at other servers.
Even though the JSW policy requires a similar excessive communication
overhead, aside from its anticipating nature, the above-mentioned
equivalence in fact means that the total number of tasks behaves
as a birth-death process, which renders it far more tractable
than the JSQ policy.
Specifically, it follows from textbook results for the centralized
M/M/$N$ queue that, given that all the servers are busy, the total number
of waiting tasks is geometrically distributed with parameter
$\lambda(N) / N$.
The total mean number of waiting tasks is then
$\Pi_W(N, \lambda(N)) \frac{\lambda(N)}{N - \lambda(N)}$,
and the mean waiting time is
$\Pi_W(N, \lambda(N)) \frac{1}{N - \lambda(N)}$,
with $\Pi_W(N, \lambda(N))$ denoting the probability of the total occupancy
in an M/M/$N$ queue being $N$ or larger, i.e., the probability
of all servers being occupied and a task incurring a non-zero waiting time.
The probability $\Pi_W(N, \lambda(N))$ can be obtained from the stationary
distribution of the birth-death process representing the system occupancy,
and is described by the so-called Erlang-C formula as function
of the load and number of servers.
The latter function can be expressed in semi-explicit well approximated
`closed form’ in terms of a normalizing constant
which is the sum of an explicit infinite series.
Standard results for the M/M/1 queue imply that the mean waiting time
is $\frac{\lambda(N)}{N - \lambda(N)}$ for the random assignment policy
considered in Section~\ref{random}.
Thus it can immediately be concluded that the mean waiting time
under the JSW policy is smaller by at least a factor $\lambda(N)$.

In the fluid regime $\lambda(N) = N \lambda$, it can be shown
that the probability $\Pi_W(N, \lambda(N))$ of a non-zero waiting time
decays exponentially fast in~$N$, see for instance~\cite{HW81},
and hence so does the mean waiting time.
The pivotal results in~\cite{HW81} further demonstrate that in the
diffusion regime~\eqref{eq:HW}, the probability $\Pi_W(N, \lambda(N))$
of a non-zero waiting time converges to a finite constant $\Pi_W^*(\beta)$.
This implies that the mean waiting time of is of the order $1 / \sqrt{N}$,
and hence vanishes as $N \to \infty$.

\subsubsection{Power-of-$d$ load balancing (JSQ($d$))}
\label{ssec:powerd}

We have seen that the Achilles heel of the JSQ policy is its excessive
communication overhead in large-scale systems.
This poor scalability has motivated consideration of so-called
JSQ($d$) policies, where an incoming task is assigned to a server
with the shortest queue among $d$~servers selected uniformly at random.
The seminal results in~\cite{Mitzenmacher01,VDK96} demonstrate that
in the fluid regime~(ii), the stationary probability that there are
$i$ or more tasks at a given queue is proportional
to $\lambda^{(d^i - 1)/(d - 1)}$ as $N \to \infty$,
and thus exhibits super-exponential decay as opposed to exponential decay
for the random assignment policy considered in Section~\ref{random}.

As alluded to in Section~\ref{sec:intro}, the diversity parameter~$d$
thus induces a fundamental trade-off between the amount of communication
overhead and the performance in terms of queue lengths and delays.
A rudimentary implementation of the JSQ policy ($d = N$, without
replacement) involves $\OO(N)$ communication overhead per task,
but it can be shown that the probability of a non-zero waiting time
and the mean waiting \emph{vanish} as $N \to \infty$
in both the fluid and the diffusion regime,
see Sections~\ref{ssec:jsqfluid} and \ref{ssec:diffjsq}.
Although JSQ($d$) policies with a fixed parameter $d \geq 2$ yield
major performance improvements over purely random assignment
as implied by the results in~\cite{Mitzenmacher01,VDK96},
these results at the same time show that even in the fluid regime,
the probability of a non-zero waiting time and the mean waiting time
\emph{do not vanish} as $N \to \infty$.

\subsubsection{Token-based mechanisms: Join-the-Idle-Queue (JIQ)}
\label{ssec:jiq}

While a zero waiting time can be achieved in the limit by sampling
only $d(N) \ll N$ servers, the amount of communication overhead
in terms of $d(N)$ must still grow with~$N$.
This can be countered by introducing memory at the dispatcher,
in particular maintaining a record of vacant servers,
and assigning tasks to idle servers as long as there are any,
or to a uniformly at random selected server otherwise.
This so-called Join-the-Idle-Queue (JIQ) scheme \cite{BB08,LXKGLG11}
has received keen interest recently, and can be implemented through
a simple token-based mechanism.
Specifically, idle servers send tokens to the dispatcher to advertise
their availability, and when a task arrives and the dispatcher has
tokens available, it assigns the task to one of the corresponding servers
(and disposes of the token).
Note that a server only issues a token when a task completion leaves
its queue empty, thus generating at most one message per task.
Surprisingly, the mean waiting time and the probability of a non-zero
waiting time vanish under the JIQ scheme in both the fluid and the
diffusion regime, as we will further discuss in Section~\ref{token}.
Thus, the use of memory allows the JIQ scheme to achieve asymptotically
optimal delay performance with minimal communication overhead.

\subsection{Performance comparison}
\label{ssec:perfcomp}

We now present some simulation results to compare the above-described
LBAs in terms of delay performance.
Specifically, we evaluate the mean waiting time and the probability
of a non-zero waiting time in both a fluid regime ($\lambda(N) = 0.9 N$)
and a diffusion regime ($\lambda(N) = N - \sqrt{N}$).
The simulations are conducted for $N=10,20,\hdots,200$ servers,
and run for 10000 time units.
Each simulation starts with an empty system, but only jobs that leave
after 2500 time units are counted in order to avoid transient effects.
The probability of waiting and mean waiting time are computed using
the empirical averages over all jobs that leave after 2500 time units.
This procedure is repeated 20 times, and the results
in Figure~\ref{differentschemes} show the mean waiting time
and probability of waiting averaged across these 20 runs.
An overview of the asymptotic delay performance and overhead
associated with various LBAs is provided in Table~\ref{table}.

\begin{figure}
\includegraphics[width=\linewidth]{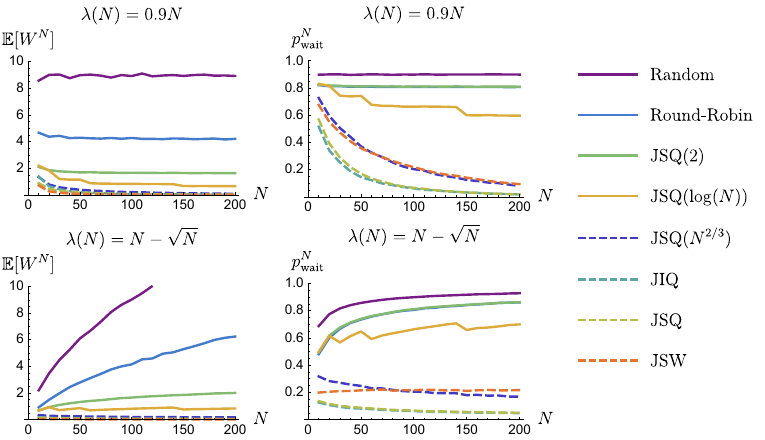}
\caption{Simulation results for mean waiting time $\mathbb{E}[W^N]$
and probability of a non-zero waiting time $p_{\textup{wait}}^N$,
for both a fluid and a diffusion regime.}
\label{differentschemes}
\end{figure}

We are specifically interested in distinguishing two classes of LBAs --
the ones delivering a mean waiting time and probability of a non-zero
waiting time that vanish asymptotically, and the ones that fail to do so
-- and relating that dichotomy to the associated communication
overhead and memory requirement at the dispatcher.
We give these classifications for both the fluid and the diffusion regime.

\paragraph{JSQ, JIQ, and JSW.}
As mentioned earlier, JSQ, JIQ and JSW have vanishing mean waiting times
in both the fluid and the diffusion regime, and this is supported by the figures,
which further reflect the optimality of JSW in terms of mean waiting time.
We can also observe a crucial difference, however, between JSW and JSQ/JIQ.
Somewhat surprisingly, the probability of a positive wait does not vanish
for JSW in the diffusion regime, while it does vanish for JSQ/JIQ.
Since the mean waiting time for JSW is smaller than for JSQ/JIQ,
this implies that the mean of all non-zero waiting times
(i.e., the mean waiting time conditional on having to wait)
is an order-of-magnitude larger in JSQ/JIQ compared to JSW.
This difference can be explained from the fact that JSW uses knowledge
of the service requirements, whereas JSQ/JIQ do not.
Indeed, when a task is placed in a queue under JSQ/JIQ, it will need to wait
for a `normal' residual service time, whereas JSW exploits knowledge
of that residual service time being relatively short among all $N$~queues.
Or taking the equivalent view of JSW as a centralized M/M/N queue,
a task that needs to wait may find several tasks ahead of it in the queue,
but this queue is served by $N$~servers combined,
whereas in JSQ/JIQ each queue is handled by just a single server.
Conversely, when there are $N$ or more tasks in the system in total,
an arriving task will need to wait under JSW,
while in JSQ/JIQ some of the servers may have several tasks in queue,
and the arriving task may still find an idle server with high probability.
We will revisit the comparison between JSQ and a centralized M/M/$N$ queue
in Section~\ref{ssec:diffjsq}.

\paragraph{Random and Round-Robin.}
The mean waiting time does not vanish for Random and Round-Robin
in the fluid regime, as already mentioned in Section~\ref{random}.
Moreover, the mean waiting time grows without bound in the diffusion regime
for these two schemes.
This is because the system can still be decomposed into single-server queues,
and the loads of the individual M/M/1 and D/M/1 queues tend to~1.

\paragraph{JSQ($d$) policies.}
Three versions of JSQ($d$) are included in Figure~\ref{differentschemes};
$d(N) = 2$, $d(N) = \lfloor \log(N) \rfloor \to \infty$ and $d(N) = N^{2/3}$
for which $\frac{d(N)}{\sqrt{N} \log(N)} \to \infty$.
Note that the graph for JSQ($\log(N)$), where the diversity parameter
grows logarithmically in~$N$, shows kneepoints due to the slow growth rate
of $\log(N)$ and the fact that the actual integer value
$d(N) = \lfloor \log(N) \rfloor$ occasionally jumps by~$1$.
As can be seen in Figure~\ref{differentschemes}, the choices
for which $d(N) \to \infty$ have vanishing wait in the fluid regime,
while $d = 2$ has not.
Overall, we see that JSQ($d$) policies clearly outperform Random
and Round-Robin dispatching, while JSQ, JIQ, and JSW are better
in terms of mean wait.

\begin{table}\centering
\def\arraystretch{1.5}
\begin{tabular}{|C{3cm}|C{3cm}|C{2.5cm}|C{3cm}|C{1.75cm}|}
\hline
Scheme & Queue length & Waiting time (fixed $\lambda < 1$) &
Waiting time ($1 - \lambda \sim 1 / \sqrt{N}$) & Overhead per task \\
\hline\hline
Random & $q_i^\star = \lambda^i$ &  $\frac{\lambda}{1 - \lambda}$ &
$\Theta(\sqrt{N})$ & 0 \\
\hline
JSQ($d$) & $q_i^\star = \lambda^{(d^i - 1)/(d - 1)}$ &
$\Theta$(1) & $\Omega(\log{N})$ & $2 d$ \\
\hline
\mbox{$d(N)$} \mbox{$\to \infty$} & same as JSQ & same as JSQ & ?? & $2d(N)$ \\
\hline
$\frac{d(N)}{\sqrt{N} \log(N)}\to \infty$ &
same as JSQ & same as JSQ & same as JSQ & $2d(N)$ \\
\hline
JSQ & $q_1^\star = \lambda$, $q_2^\star =$ o(1) & o(1) &
$\Theta(1 / \sqrt{N})$ & $2 N$ \\
\hline\hline
JIQ & same as JSQ & same as JSQ & same as JSQ & $\leq 1$ \\
\hline
\end{tabular}
\caption{Queue length distribution, waiting times, and communication
overhead for various~LBAs.}
\label{table}
\end{table}

\section{Preliminaries, JSQ policy, and power-of-$d$ algorithms}
\label{sec:powerofd}

In this section we first introduce some notation and preliminary concepts,
and then review fluid and diffusion limits for the JSQ policy
as well as JSQ($d$) policies with a fixed value of~$d$.

\subsection{Definitions, limit sequences and convergence issues}
\label{defi}

We continue to focus on the basic scenario where all the servers are
homogeneous, the service requirements are exponentially distributed,
and the service discipline at each server is oblivious of the actual
service requirements.
Moreover, most of the LBAs under consideration do not distinguish
between servers with equal queue lengths.
Consequently, the queue-length process is Markov on an enlarged filtration,
allowing for random draws to resolve ties.
In order to obtain a Markovian state description, it therefore
suffices to only track the number of tasks, and in fact we do not need
to keep record of the number of tasks at each individual server,
but only count the number of servers with a given number of tasks.
Specifically, we represent the state of the system by a vector
\[
\QQ(t) := \left(Q_1(t), Q_2(t), \dots\right)
\]
with $Q_i(t)$ denoting the number of servers with $i$ or more tasks
at time~$t$, including the possible task in service, $i = 1, 2 \dots$.
Note that if we represent the queues at the various servers as (vertical)
stacks, and arrange these from left to right in ascending order,
then the value of $Q_i$ corresponds to the width of the $i$-th (horizontal)
row, as depicted in the schematic diagram in Figure~\ref{figB}.

In order to examine the fluid and diffusion limits in regimes
where the number of servers~$N$ grows large, we consider a sequence
of systems indexed by~$N$, and attach a superscript~$N$ to the
associated state variables.
The fluid-scaled occupancy state is denoted by
$\qq^N(t) := (q_1^N(t), q_2^N(t), \dots)$, with $q_i^N(t) = Q_i^N(t) / N$
representing the fraction of servers in the $N$-th system with $i$
or more tasks as time~$t$, $i = 1, 2, \dots$.
Let $$\cS = \big\{\qq \in [0, 1]^\infty: q_i \leq q_{i-1}\
\forall i = 2, 3, \dots, \text{ and }\sum_{i=1}^\infty q_i <\infty\big\}$$
be the set of all possible fluid-scaled states equipped with the $\ell_1$ topology.
Any (weak) limit $\qq(\cdot)$ of the sequence of processes
$\{\qq^N(t)\}_{t \geq 0}$ in the conventional large capacity regime~(ii)
as $N \to \infty$ (in a suitable topology on the space of functions
on $[0, T]$ taking values in~$\cS$) is called a \textit{fluid limit}.
In some frameworks in the literature this is also commonly referred to
a \textit{mean-field limit} when the occupancy process is viewed
as the (density-dependent) state evolution of a population of randomly
interacting nodes or particles \cite{BLB08,DN08,Kurtz78,LBMDM07}.
Whenever we consider fluid limits, we assume the sequence of initial
states is such that $\qq^N(0) \to \qq^\infty \in \cS$ as $N \to \infty$.

The diffusion-scaled occupancy state is defined as
$\bar{\QQ}^N(t) = (\bar{Q}_1^N(t), \bar{Q}_2^N(t), \dots)$, with
\begin{equation}
\label{eq:diffscale}
\bar{Q}_1^N(t) = - \frac{N - Q_1^N(t)}{\sqrt{{N}}}, \qquad
\bar{Q}_i^N(t) = \frac{Q_i^N(t)}{\sqrt{{N}}}, \quad i = 2,3, \dots,
\end{equation}
where we include a minus sign in the definition of $\bar{Q}_1^N(t)$
so as to adhere to the notation adopted in~\cite{EG15} which is the basis
for the results that will be presented in Section~\ref{ssec:diffjsq}.
Any (weak) limit $\QQ(\cdot)$ of the sequence of processes
$\{\QQ^N(t)\}_{t \geq 0}$ in the Halfin-Whitt heavy-traffic regime
(iii) as $N \to \infty$, once again in a suitable topology,
is called a \textit{diffusion limit}.
Note that $-\bar{Q}_1^N(t)$ corresponds to the number of vacant servers,
normalized by $\sqrt{N}$.
The reason why $Q_1^N(t)$ is centered around~$N$ while $Q_i^N(t)$,
$i = 2,3, \dots$, are not, is that for the scalable LBAs we consider
the fraction of servers with exactly one task tends to one, whereas the
fraction of servers with two or more tasks tends to zero as $N \to \infty$.
For convenience, we will assume that each server has an infinite-capacity
buffer, but all the results extend to the finite-buffer case,
see for instance \cite{EG15,MBLW16-1,MBLW16-3,MBLW16-4,MBLW2016-5}. \\

We conclude this subsection with a discussion of two important
convergence issues associated with the above-defined scaling limits. 

\paragraph{Accuracy of asymptotic approximations.}
A critical issue in the context of scaling limits is the rate
of convergence and the accuracy for finite-size systems.
Some interesting results for the accuracy of mean-field approximations
for interacting-particle systems including load balancing models may be
found in~\cite{Gast17,Ying16,Ying17}.
These results can be leveraged to develop refined approximations
and improve the accuracy by adding expansion terms as demonstrated
in~\cite{GH18,GBT19,GLM18}.

\paragraph{Global asymptotic stability, stationary distributions,
and interchange of limits.}
A further crucial issue in the context of scaling limits is whether
limit processes that arise as $N \to \infty$ itself have (unique)
subsequential limits or limiting distributions as $t \to \infty$,
and if so, how the stationary distributions of the pre-limit processes
(assuming that exists) relate to those limits.
For fluid limits, which are usually described in terms of a system
of differential equations, the first question translates to the
existence of a unique invariant point (fixed point) of these equations.
While in most cases of practical interest such a unique invariant point
tends to exist, this may be non-trivial to prove, and the existence
of multiple invariant points can not a priori be ruled out in general.
In fact, existence of multiple invariant points has been shown
in specific scenarios, and is an indication of oscillatory behavior
and so-called bi-stability issues in the original stochastic process
for large~$N$ \cite{GHK90,MR18}.
Even when it can be established that a unique invariant point exists, the
next question pertains to global attraction or global asymptotic stability.
Specifically, the invariant point is said to be a global attractor,
or globally asymptotically stable, if the fluid limit process
converges to this point for any initial condition.
Global asymptotic stability has been established for various
particular model instances, including the supermarket model with JSQ($d$)
load balancing strategies \cite{VH19,Mitzenmacher01,VDK96}.
Common proof methodologies involve Lyapunov constructions \cite{BDFR15,FG16,Gast15},
monotonicity properties \cite{Mitzenmacher96,TX11,VH19,VDK96}
and reversibility concepts \cite{LeBoudec13}, but there is no
systematic recipe available, and the specific proof arguments tend to
be highly tailored to the particular system under consideration.
If global asymptotic stability of the invariant point can be established,
then along with tightness this ensures that the sequence of stationary
distributions of the pre-limit process (assuming these exist) converge
to this point, see for instance~\cite{BLB11}, with some of the key ideas
and results dating back to much earlier work \cite{Whitt85,Hunt95}.
This provides justification for an interchange of the large-scale
($N \to \infty$) and stationary ($t \to \infty$) limits,
indicating that the invariant point provides a suitable approximation
for the stationary distribution of the original stochastic process
for sufficiently large values of~$N$.
In addition, the interchange of limits tends to furnish asymptotic
independence among any finite subset of the queues~\cite{Graham00}.
Related results, convergence rates and error probabilities are established
in~\cite{LM07,LN13}.
Somewhat similar issues and observations apply for diffusion limits
\cite{GZ06,KR12}.

\subsection{Fluid limit for JSQ(d) policies}

We first consider the fluid limit for JSQ($d$) policies
with an arbitrary but fixed value of~$d$ as characterized
by the seminal results in \cite{Mitzenmacher96,VDK96}.
The result below is paraphrased from \cite{Mitzenmacher96, VDK96}.

\paragraph{Fluid limit for JSQ($d$).}
{\em The sequence of processes $\{\qq^N(t)\}_{t \geq 0}$ has a weak
limit $\{\qq(t)\}_{t \geq 0}$ that satisfies the system
of differential equations}
\begin{equation}
\label{fluid:standard}
\frac{\dif q_i(t)}{\dif t} =
\lambda (q_{i-1}^d(t) - q_i^d(t)) - (q_i(t) - q_{i+1}(t)),
\quad i = 1, 2, \dots,
\end{equation}
with $q_0(t) \equiv 1$ for all $t \geq 0$.
The fluid-limit equations may be interpreted as follows.
The first term represents the rate of increase in the fraction
of servers with $i$ or more tasks due to arriving tasks that are
assigned to a server with exactly $i - 1$ tasks.
Note that the latter occurs in fluid state $\qq \in \cS$ with probability
$q_{i-1}^d - q_i^d$, i.e., the probability that all $d$~sampled servers
have $i - 1$ or more tasks, but not all of them have $i$ or more tasks.
The second term corresponds to the rate of decrease in the fraction
of servers with $i$ or more tasks due to service completions from servers
with exactly $i$ tasks, and the latter rate is given by $q_i - q_{i+1}$.
The system in~\eqref{fluid:standard} characterizes the functional law
of large numbers (FLLN) behavior of systems in regime (ii) under the JSQ($d$) scheme.
Weak convergence of the diffusion-scaled variation around the fluid-limit path
to a certain Ornstein-Ulenbeck process in the same load regime
(both the transient behavior and in steady state) was shown in~\cite{Graham05},
establishing a functional central limit theorem (FCLT) result.
Strong approximations for systems under the JSQ($d$) scheme on any finite
time interval by the deterministic system in~\eqref{fluid:standard},
a certain infinite-dimensional jump process, and a diffusion
approximation were established in~\cite{LN05}.

Now, assume $\lambda\in (0,1)$ for ergodicity of the queue-length process.
When the derivatives in~\eqref{fluid:standard} are set equal to zero
for all~$i$, the unique fixed point for any $d \geq 2$ is obtained as
\cite{Mitzenmacher96,VDK96}
\begin{equation}
\label{eq:fixedpoint1}
q_i^* = \lambda^{\frac{d^i-1}{d-1}}.
\quad i = 1, 2, \dots.
\end{equation}
It can be shown that the fixed point is asymptotically stable in the
sense that $\qq(t) \to \qq^*$ as $t \to \infty$ for any initial fluid
state $\qq^\infty$ with $\sum_{i = 1}^{\infty} q_i^\infty < \infty$.
As mentioned earlier, the fixed point reveals that the stationary
queue length distribution at each individual server exhibits
\emph{super-exponential} decay as $N \to \infty$,
as opposed to exponential decay for a random assignment policy.
As described above, this involves an interchange of the many-server
($N \to \infty$) and stationary ($t \to \infty$) limits.
The justification is provided by the asymptotic stability of the fixed
point along with a few further technical conditions.

\subsection{Fluid limit for JSQ policy}
\label{ssec:jsqfluid}

We now turn to the fluid limit for the ordinary JSQ policy,
which rather surprisingly was not rigorously established until fairly
recently in~\cite{MBLW16-3}, leveraging martingale functional limit
theorems and time-scale separation arguments~\cite{HK94}.

In order to state the fluid limit starting from an arbitrary
fluid-scaled occupancy state, we first introduce some additional notation.
For any fluid state $\qq \in \cS$,
denote by $m(\qq) = \min\{i\geq 0: q_{i + 1} < 1\}$ the minimum queue length
among all servers.
Now if $m(\qq) = 0$, then define $p_0(\qq) = 1$ and $p_i(\qq) = 0$
for all $i = 1, 2, \ldots$.
Otherwise, in case $m(\qq) > 0$, define
\begin{equation}
\label{eq:fluid-gen}
p_i(\qq) =
\begin{cases}
\min\big\{(1 - q_{m(\qq) + 1})/\lambda,1\big\} & \quad\mbox{ for }\quad i=m(\qq)-1, \\
1 - p_{m(\qq) - 1}(\qq) & \quad\mbox{ for }\quad i=m(\qq), \\
0&\quad \mbox{ otherwise.}
\end{cases}
\end{equation}
The fluid-limit result below is paraphrased from \cite{MBLW16-3}.
\paragraph{Fluid limit of JSQ.}
{\em For $\lambda\in (0,1)$, the weak limit of the sequence of processes
$\{\qq^N(t)\}_{t \geq 0}$ is given by a deterministic system $\{\qq(t)\}_{t \geq 0}$
that satisfies the system of differential equations
\begin{equation}
\label{eq:fluid}
\frac{\dif^+ q_i(t)}{\dif t} =
\lambda p_{i-1}(\qq(t)) - (q_i(t) - q_{i+1}(t)),
\quad i = 1, 2, \dots,
\end{equation}
where $\dif^+/\dif t$ denotes the right-derivative.}
The reason we have used derivative in~\eqref{fluid:standard},
and right-derivative in~\eqref{eq:fluid} is that the limiting trajectory
for the JSQ policy may not be differentiable at all time points.
In fact, one of the major technical challenges in proving the fluid limit
for the JSQ policy is that the drift of the process is not continuous,
which leads to non-smooth limiting trajectories,
see~\cite{MBLW16-3} for further details.
The uniqueness of the above weak limit was not established in~\cite{MBLW16-3},
but follows from the recent result in~\cite[Theorem 2.1]{BBD20}.

The fluid-limit trajectory in~\eqref{eq:fluid} can be interpreted as follows.
The coefficient $p_i(\qq)$ represents the instantaneous fraction
of incoming tasks assigned to servers with a queue length of exactly~$i$
in the fluid state $\qq \in \mathcal{S}$.
Note that a strictly positive fraction $1 - q_{m(\qq) + 1}$ of the
servers have a queue length of exactly~$m(\qq)$.
Clearly the fraction of incoming tasks that get assigned to servers
with a queue length of $m(\qq) + 1$ or larger is zero:
$p_i(\qq) = 0$ for all $i = m(\qq) + 1, \dots$.
Also, tasks at servers with a queue length of exactly~$i$ are completed
at (normalized) rate $q_i - q_{i + 1}$, which is zero for all
$i = 0, \dots, m(\qq) - 1$, and hence the fraction of incoming tasks
that get assigned to servers with a queue length of $m(\qq) - 2$ or less
is zero as well: $p_i(\qq) = 0$ for all $i = 0, \dots, m(\qq) - 2$.
This only leaves the fractions $p_{m(\qq) - 1}(\qq)$
and $p_{m(\qq)}(\qq)$ to be determined.
Now observe that the fraction of servers with a queue length of exactly
$m(\qq) - 1$ is zero.
If $m(\qq)=0$, then clearly the incoming tasks will join an empty queue, 
and thus, $p_{m(\qq)} = 1$, and $p_i(\qq) = 0$ for all $i \neq m(\qq)$.
Furthermore, if $m(\qq) \geq 1$, since tasks at servers with a queue
length of exactly $m(\qq)$ are completed at (normalized) rate
$1 - q_{m(\qq) + 1} > 0$, incoming tasks can be assigned to servers
with a queue length of exactly $m(\qq) - 1$ at that rate.
We thus need to distinguish between two cases, depending on whether the
normalized arrival rate $\lambda$ is larger than $1 - q_{m(\qq) + 1}$ or not.
If $\lambda < 1 - q_{m(\qq) + 1}$, then all the incoming tasks can be
assigned to a server with a queue length of exactly $m(\qq) - 1$,
so that $p_{m(\qq) - 1}(\qq) = 1$ and $p_{m(\qq)}(\qq) = 0$.
On the other hand, if $\lambda > 1 - q_{m(\qq) + 1}$, then not all
incoming tasks can be assigned to servers with a queue length of
exactly $m(\qq) - 1$ active tasks, and a positive fraction will be
assigned to servers with a queue length of exactly $m(\qq)$:
$p_{m(\qq) - 1}(\qq) = (1 - q_{m(\qq) + 1}) / \lambda$
and $p_{m(\qq)}(\qq) = 1 - p_{m(\qq) - 1}(\qq)$.

In case $\lambda \in (0,1)$, the unique fixed point
$\qq^\star = (q_1^\star, q_2^\star, \ldots)$ of the dynamical system
in~\eqref{eq:fluid} is given by
\begin{equation}
\label{eq:fpjsq}
q_i^* = \left\{\begin{array}{ll} \lambda, & i = 1, \\
0, & i = 2, 3,\dots. \end{array} \right.
\end{equation}
Note that the fixed point naturally emerges when $d \to \infty$ in the
fixed point expression~\eqref{eq:fixedpoint1} for fixed~$d$.
However, the process-level results in \cite{Mitzenmacher01,VDK96}
for fixed~$d$ cannot be readily used to handle joint scalings of~$d$ and~$N$,
and do not yield the entire fluid-scaled sample path for arbitrary
initial states as given by~\eqref{eq:fluid}.
The fixed point in~\eqref{eq:fpjsq}, in conjunction with an interchange
of limits argument, indicates that in stationarity the fraction
of servers with a queue length of two or larger under the JSQ policy
is negligible as $N \to \infty$.

\subsection{Diffusion limit for JSQ policy}
\label{ssec:diffjsq}

We next describe the diffusion limit for the JSQ policy in the
Halfin-Whitt heavy-traffic regime \eqref{eq:HW}, as derived in~\cite{EG15}. 
The statement below is paraphrased from \cite{EG15}.
Recall the centered and diffusion-scaled processes in~\eqref{eq:diffscale}.

\paragraph{Diffusion limit for JSQ.}
{\em For suitable initial conditions, the sequence of processes
$\big\{\bar{\QQ}^N(t)\big\}_{t \geq 0}$ converges weakly to the limit
$\big\{\bar{\QQ}(t)\big\}_{t \geq 0}$, 
where 
$(\bar{Q}_1(t), \bar{Q}_2(t),\ldots)$
is the unique solution to the following system of SDEs
\begin{equation}
\label{eq:diffusionjsq}
\begin{split}
\dif\bar{Q}_1(t) &= \sqrt{2} \dif W(t) - \beta \dif t - \bar{Q}_1(t) + \bar{Q}_2(t) - \dif U_1(t), \\
\dif\bar{Q}_2(t) &= \dif U_1(t) -  \bar{Q}_2(t), 
\end{split}
\end{equation}
and $\bQ_i(t) = 0$, $i \geq 3$, for $t \geq 0$, where $W$ is standard Brownian motion
and $U_1$ is the unique continuous non-decreasing non-negative process 
satisfying $\int_0^\infty \mathbbm{1}_{[\bar{Q}_1(t) < 0]} \dif U_1(t) = 0$ and $U_1(0) = 0$.}

The diffusion-limit characterization in~\eqref{eq:diffusionjsq}
may be interpreted as follows.
First of all, recall that $- \bar{Q}_1$ corresponds to the number
of vacant servers (normalized by~$\sqrt{N}$), and observe that this number
is governed by the number of arriving tasks on the one hand
(as long as the number of vacant servers is non-zero),
with associated exponential rate $\lambda(N)$, and on the other hand
the number of service completions at servers with exactly one task,
with associated exponential rate $Q_1^N - Q_2^N$.
Noting that $(N - \lambda(N)) / \sqrt{N} \to \beta$,
$\bar{Q}_1^N = - (N - Q_1^N) / \sqrt{N}$ and $\bar{Q}_2^N = Q_2^N / \sqrt{N}$,
we recognize that these dynamics are reflected in the equation for $\dif\bar{Q}_1(t)$,
with $\sqrt{2} \dif W(t)$ an additional diffusion term corresponding
to the variation in the number of arrivals and service completions
around the drift terms and $\dif U_1(t)$ a reflection term
accounting for the fact that the number of vacant servers cannot be negative.
More specifically, the term $\dif U_1(t)$ tracks the number of arriving tasks
assigned to busy servers when there are no vacant servers,
which explains why the derivative can only be positive when $\bar{Q}_1 < 0$.
Now observe that, for suitable initial conditions, since $\beta < 0$,
it is highly unlikely for all servers to have two or more tasks,
and the number of servers with three or more tasks is negligible
on diffusion scale, as reflected in the fact that $\bar{Q}_i = 0$, $i \geq 3$.
Also, the dynamics of the number of servers with two or more tasks are governed
by the assignment of tasks to busy servers captured by the term $\dif U_1(t)$
and the service completions at servers with exactly two tasks, which is equal
to $\bar{Q}_2$ on diffusion scale since the number of servers with three
or more tasks is negligible, explaining the equation for $\dif\bar{Q}_2(t)$.

The above convergence of the scaled occupancy measure was established
in~\cite{EG15} only for any finite time interval.
The tightness of the sequence of diffusion-scaled steady-state
occupancy measures $\{(\bQ_1^N(\infty), \bQ_2^N(\infty))\}_{N\geq 1}$,
the ergodicity of the limiting diffusion process~\eqref{eq:diffusionjsq},
and hence the interchange of limits were open until \cite{Braverman20}
further established that the weak-convergence result extends
to the steady state as well, i.e., $\bar{\QQ}^N(\infty)$ converges weakly
to the random variable $(\bQ_1(\infty), \bQ_2(\infty), 0, 0, \ldots)$
as $N \to \infty$, where $(\bQ_1(\infty), \bQ_2(\infty))$
has the stationary distribution of the process $(\bQ_1, \bQ_2)$.
Thus, the steady state of the diffusion process
in~\eqref{eq:diffusionjsq} is proved to capture the asymptotic behavior
of large-scale systems under the JSQ policy.
 
In~\cite{Braverman20} a Lyapunov function is obtained via a generator
expansion framework using Stein's method, which establishes exponential
ergodicity of $(\bQ_1, \bQ_2)$.
Although this approach gives a good handle on the rate of convergence
to stationarity, it sheds little light on the form of the stationary distribution
of the limiting diffusion process~\eqref{eq:diffusionjsq} itself. 
In two companion papers \cite{BM19a,BM19b} the authors perform
a detailed analysis of the steady state of this diffusion process.
Using a classical regenerative process construction of the diffusion
process in~\eqref{eq:diffusionjsq}, \cite{BM19a} establishes that
$\bQ_1(\infty)$ has a Gaussian tail, and the tail exponent is uniformly
bounded by constants which do not depend on~$\beta$,
whereas $\bQ_2(\infty)$ has an exponentially decaying tail,
and the coefficient in the exponent is linear in~$\beta$.
More precisely, for any $\beta>0$ there exist positive constants
$C_1, C_2, D_1, D_2$ not depending on~$\beta$ and positive constants
$C^l(\beta)$, $C^u(\beta)$, $D^l(\beta)$, $D^u(\beta)$, $C_R(\beta)$,
$D_R(\beta)$ depending only on~$\beta$ such that 
\begin{equation}
\label{eq:statail}
\begin{split}
C^l(\beta) \ee^{- C_1 x^2} \le \P(\bQ_1(\infty) < -x) \le
C^u(\beta) \ee^{- C_2 x^2}, \ \ x \ge C_R(\beta) \\
D^l(\beta) \ee^{- D_1 \beta y} \le \P(\bQ_2(\infty) > y) \le
D^u(\beta) \ee^{- D_2 \beta y}, \ \ y \ge D_R(\beta).
\end{split}
\end{equation}
It was further shown in~\cite{BM19a} that there exists
a positive constant $\mathcal{C^*}$ not depending on~$\beta$
such that almost surely along any sample path:
\[
\begin{split}
- 2 \sqrt{2} &\le \liminf_{t \rightarrow \infty} \frac{\bQ_1(t)}{\sqrt{\log t}} \le -1,\\
\frac{1}{\beta} &\le \limsup_{t \rightarrow \infty} \frac{\bQ_2(t)}{\log t} \le \frac{2}{\mathcal{C^*} \beta}.
\end{split}
\]
Notice that the width of fluctuation of $\bQ_1$ does not depend on the value
of~$\beta$, whereas that of $\bQ_2$ is linear in $\beta^{-1}$. 

Since the $N$-th system is ergodic and its arrival rate is
$N - \beta\sqrt{N}$, it is straightforward to see that
$\E(\bQ_1^N(\infty)) = - \beta$ for all~$N$, and hence,
it can also be derived from the evolution of the limiting diffusion
process that $\E(\bQ_1(\infty)) = - \beta$.
Thus, intuitively, for large enough~$\beta$, the system has mostly
many idle servers, and the number of servers with queue length
at least two diminishes.
But the manner $\bQ_2(\infty)$ scales as $\beta$ becomes large,
is highly non-trivial.
Specifically, it was shown in~\cite{BM19b} that there exists
$\beta_0 \ge 1$ and positive constants $C_1, C_2, D_1, D_2$ such that
for all $\beta \ge \beta_0$,
\begin{eq}
\label{eq:large-beta}
\e^{-C_1\beta^2} \le \mathbb{E}\left(\bQ_2(\infty)\right) \le \e^{-C_2\beta^2},\\
\P\Big(\bQ_2(\infty) \ge \e^{-\e^{D_1\beta^2}}\Big) \le \e^{-D_2\beta^2},
\end{eq}
i.e., the steady-state mean is of order $\e^{- C \beta^2}$, but most
of the steady-state mass concentrates at a much smaller scale
$\e^{- \e^{D_1 \beta^2}}$.
This suggests \emph{intermittency} in the behavior of the $\bQ_2$ process,
namely, $\bQ_2$ is typically of order $\e^{- \e^{D_1 \beta^2}}$,
but during rare events when it achieves higher values, it takes a long time
to decay.
However, for small enough~$\beta$, the behavior is qualitatively different.
Since $\E(\bQ_1(\infty)) = - \beta$, the system is expected to become more
congested as $\beta$ becomes smaller.
As a result, intuitively, $\bQ_2$ should increase in the distributional sense.
In this regime as well, $\bQ_2$ exhibits some striking behavior.
Specifically, it was shown in~\cite{BM19b} that there exist positive
constants $\beta^*$, $M_1$ and $M_2$ such that for all $\beta \le \beta^*$
\begin{eq}
\label{eq:small-beta}
\frac{M_1}{\beta} \le \mathbb{E}(\bQ_2(\infty)) \le \frac{M_2}{\beta}.
\end{eq}

\paragraph{Comparison with M/M/$N$ queue.}
The M/M/$N$ queue in the Halfin-Whitt heavy-traffic regime has been studied
quite extensively (see \cite{GG13a,GG13b,HW81,LK11,LK12,FKL14,LMZ17},
and the references therein).
In this case, the centered and scaled total number of tasks in the system
$(\bar{S}^N(t) - N)/\sqrt{N}$ converges weakly to a diffusion process
$\{\bar{S}(t)\}_{t \geq 0}$ \cite[Theorem 2]{HW81} with
\begin{equation}
\label{eq:diffusionmmn}
\dif\bar{S}(t) =
\sqrt{2} \dif W(t) - \beta \dif t - \dif\bar{S}(t) \ind{\bar{S}(t) \leq 0},
\end{equation}
where $W$ is the standard Brownian motion.
As reflected in~\eqref{eq:diffusionjsq} and~\eqref{eq:diffusionmmn},
the JSQ policy and the M/M/$N$ system share some striking
similarities in terms of the qualitative behavior of the total number
of tasks in the system.
In particular, both the number of idle servers and the number
of waiting tasks are of the order $\Theta(\sqrt{N})$.
This shows that despite the distributed queueing operation
a suitable load balancing policy can deliver a similar combination
of excellent service quality and high resource utilization efficiency
in the QED (Quality-and-Efficiency-Driven) regime
(recall from Section~\ref{asym}) as in a centralized queueing arrangement.
Moreover, the interchange of limits result in~\cite{Braverman20}
implies that for systems under the JSQ policy,
$\bQ_{tot}^N(\infty) := \sum_{i = 1}^{\infty} Q_i^N(\infty)$ converges
weakly to $\bQ_1(\infty) + \bQ_2(\infty)$, which has an exponential
upper tail (large positive deviation) and a Gaussian lower tail
(large negative deviation), see~\eqref{eq:statail}.
This is again reminiscent of the corresponding tail asymptotics
for the M/M/$N$ queue.
Note that since $\bar{S}(\cdot)$ is a simple combination of a Brownian
motion with a negative drift (when all servers are fully occupied)
and an Ornstein Uhlenbeck (OU) process (when there are idle servers),
the steady-state distribution $\bar{S}(\infty)$ can be computed explicitly,
and is indeed a combination of an exponential distribution
and a Gaussian distribution.

There are, however, some clear differences between the diffusion
in~\eqref{eq:diffusionjsq} and~\eqref{eq:diffusionmmn}:
\begin{enumerate}[{\normalfont (i)}]
\item Observe that in case of M/M/$N$ systems, whenever there are
waiting tasks (equivalent to $Q_2$ being positive in our case),
the queue length has a constant negative drift towards zero.
This leads to the exponential upper tail of $\bar{S}(\infty)$,
by comparing with the stationary distribution of a reflected Brownian
motion with constant negative drift.
In the JSQ case, however, the rate of decrease of $Q_2$ is always
proportional to itself, which makes it somewhat counter-intuitive that
its stationary distribution has an exponential tail.
\item In the M/M/$N$ system, the number of idle servers can be non-zero
only when the number of waiting tasks is zero.
Thus, the dynamics of both the number of idle servers and the number
of waiting tasks are completely captured by the one-dimensional process
$\bar{S}^N$ and by the one-dimensional diffusion  $\bar{S}$ in the limit. 
But in the JSQ case, $\bQ_2$ is never zero, and the dynamics
of $(\bQ_1, \bQ_2)$ are truly two-dimensional (although the diffusion
is non-elliptic) with $\bQ_1$ and $\bQ_2$ interacting with each other
in an intricate manner.
\item From~\eqref{eq:diffusionjsq} we see that $\bQ_2$ \emph{never}
hits zero. 
Thus, in steady state, there is no mass at $\bQ_2 = 0$, and the system
always has waiting tasks. 
This is in sharp contrast with the M/M/$N$ case, where the system has
no waiting tasks in steady state with positive probability. 
\item In the M/M/$N$ system, a positive fraction of the tasks incur
a non-zero waiting time as $N \to \infty$, but a non-zero waiting time
is only of length $1 / (\beta \sqrt{N})$ in expectation. 
In contrast, in the JSQ case, it is easy to see that $\bQ_1$ (the limit
of the scaled number of idle servers) spends zero time at the origin,
i.e., in steady state the fraction of arriving tasks that find all servers
busy vanishes in the large-$N$ limit (in fact, this is of order $1/\sqrt{N}$,
see~\cite{Braverman20}).
However, such tasks will have to wait for the duration of a residual
service time, implying that a non-zero waiting time is of the order $\OO(1)$
and does \emph{not} vanish.
\item As $\beta \to 0$, \cite[Proposition 2]{HW81} implies that
$\beta \bar{S}(\infty)$ for the M/M/$N$ queue converges weakly
to a unit-mean exponential distribution. 
In contrast, results in~\cite{BM19b} show that
$\beta (\bQ_1(\infty)+\bQ_2(\infty))$ converges weakly to a Gamma$(2)$
random variable.
This indicates that despite similar order of performance, due to the
distributed operation, in terms of the number of waiting tasks JSQ is
a factor~$2$ worse in expectation than the corresponding centralized system.
\end{enumerate}

\subsection{JSQ($d$) policies in heavy-traffic regime}

Finally, we briefly discuss the behavior of JSQ($d$) policies with a fixed value
of~$d$ in the Halfin-Whitt heavy-traffic regime~\eqref{eq:HW}.
While a complete characterization of the occupancy process
for fixed~$d$ has remained elusive so far, significant partial results
were obtained in~\cite{EG16}.
In order to describe the transient asymptotics, introduce the
following rescaled processes
\begin{equation}
\label{eq:HTtransient}
\bQ_i^N(t) := \frac{N - Q_i^N(t)}{\sqrt{N}}, \quad i = 1, 2,\ldots.
\end{equation}
Note that in contrast to~\eqref{eq:diffscale}, in~\eqref{eq:HTtransient}
\emph{all} components are centered by~$N$. 
Also note that the sign of the first coordinate in~\eqref{eq:HTtransient}
is the opposite of that in~\eqref{eq:diffscale}.
The statement below is paraphrased from~\cite{EG16}.

\paragraph{Process-level limit of JSQ($d$) policy in Halfin-Whitt regime.}
\emph{Assuming that the initial states converge
with respect to the product topology under the above scaling,
\cite[Theorem~2]{EG16} establishes that on any finite time interval,
$\bar{\QQ}^N(\cdot)$ converges weakly to a deterministic system
$\bar{\QQ}(\cdot)$ that satisfies the system of ODEs}
\[
\dif \bQ_i(t) = - d (\bQ_i(t) - \bQ_{i-1}(t)) + \bQ_{i+1}(t) - \bQ_i(t),
\quad i = 1, 2, \ldots
\]
with the convention that $\bQ_0(t) \equiv 0$.
It is noteworthy that the scaled occupancy process loses its diffusive
behavior for fixed~$d$.
It is further shown in~\cite{EG16} that with high probability the
steady-state fraction of queues with length at least
$\log_d(\sqrt{N} / \beta) - \oo(1)$ tasks approaches unity,
which in turn implies that with high probability the steady-state delay is
\emph{at least} $\log_d(\sqrt{N} / \beta) - \OO(1)$ as $N \to \infty$.
The diffusion approximation of the JSQ($d$) policy in the Halfin-Whitt
regime~\eqref{eq:HW}, starting from a different initial state,
has been studied in~\cite{BF17}.

In~\cite{Ying17} a broad framework involving Stein's method was introduced
to analyze the rate of convergence of the stationary distribution
under the JSQ$(2)$ policy, in a heavy-traffic regime,
where $(N - \lambda(N)) / \eta(N) \to \beta > 0$ as $N \to \infty$,
with $\eta(N)$ a positive function diverging to infinity as $N \to \infty$.
Note that the case $\eta(N)= \sqrt{N}$ corresponds to the Halfin-Whitt
heavy-traffic regime~\eqref{eq:HW}.
Using this framework, it was proved that when $\eta(N) = N^\alpha$
with some $4/5<\alpha\leq 1$, 
\begin{equation}
\label{eq:po2ht}
\mathbb{E}\Big(\sum_{i=1}^\infty \Big|q_i^N(\infty) - q_i^{N,\star}\Big|\Big)\leq\frac{1}{N^{2\alpha -1-\xi}},\qquad\mbox{where}\qquad q_i^{N,\star} = \Big(\frac{\lambda(N)}{N}\Big)^{2^i-1},
\end{equation}
and $\xi > 0$ is an arbitrarily small constant.
Equation~\eqref{eq:po2ht} not only shows that asymptotically
the stationary occupancy measure concentrates at $\qq^{N,\star}$,
but also provides the rate of convergence.

\section{Universality of JSQ(d) policies}
\label{univ}

In this section we will further explore the trade-off between delay
performance and communication overhead as a function of the diversity
parameter~$d$, in conjunction with the relative load.
The latter trade-off will be examined in an asymptotic regime
where not only the total task arrival rate $\lambda(N)$ grows with~$N$,
but also the diversity parameter depends on~$N$, and we write $d(N)$
to explicitly reflect this dependence.
We will specifically investigate what growth rate of $d(N)$ is required,
depending on the scaling behavior of $\lambda(N)$,
in order to asymptotically match the optimal performance of the JSQ policy
and achieve a zero mean waiting time in the limit.
The results presented in the remainder of the section are based
on~\cite{MBLW16-3} where also the full proofs are provided,
unless specified otherwise.

\begin{theorem}[Universality of fluid limit for JSQ($d(N)$)]
\label{fluidjsqd}
If $d(N) \to \infty$ as $N \to \infty$, then any fluid limit of the
JSQ$(d(N))$ scheme coincides with that of the ordinary JSQ policy,
and in particular, satisfies the system of differential equations in~\eqref{eq:fluid}.
Consequently, the stationary occupancy states converge to the unique fixed point
as in~\eqref{eq:fpjsq}.
\end{theorem}

\begin{theorem}[Universality of diffusion limit for JSQ($d(N)$)]
\label{diffusionjsqd}
If $d(N) /( \sqrt{N} \log N) \to \infty$ as $N \to \infty$,
then for suitable initial conditions the weak limit of the sequence
of processes $\big\{\bar{\QQ}^{N}(t)\big\}_{t \geq 0}$, under the JSQ($d(N)$) policy,
coincides with that of the ordinary JSQ policy, and in particular,
is given by the system of SDEs in~\eqref{eq:diffusionjsq}.
\end{theorem}

The above universality properties indicate that the JSQ overhead can
be lowered by almost a factor O($N$) and O($\sqrt{N} / \log N$)
while retaining fluid- and diffusion-level optimality, respectively.
In other words, Theorems~\ref{fluidjsqd} or~\ref{diffusionjsqd}
reveal that it is sufficient for $d(N)$ to grow at any rate,
or faster than $\sqrt{N} \log N$, in order to observe similar scaling
benefits as in a pooled system with $N$~parallel single-server queues
on fluid scale and diffusion scale, respectively.
The stated conditions are in fact close to necessary, in the sense
that if $d(N)$ is uniformly bounded or $d(N) / (\sqrt{N} \log N) \to 0$
as $N \to \infty$, then respectively, the fluid-limit and diffusion-limit
paths under the JSQ($d(N)$) scheme \emph{differ} from those
under the ordinary JSQ policy.
In particular, if $d(N)$ is uniformly bounded, the mean steady-state delay
does not vanish as $N \to \infty$.

\begin{remark}
\normalfont
One implication of Theorem~\ref{fluidjsqd} is that in the subcritical
regime any growth rate of $d(N)$ is enough to achieve asymptotically
vanishing steady-state probability of wait. 
This result is complemented by the results in~\cite{LY19,BFL18},
where the steady-state analysis is extended in the heavy-traffic regime
with $N^\alpha (1 - \lambda(N) / N) \to \beta>0$ as $N \to \infty$
with $\alpha \in (0, 1/2)$.
Note that the system approaches heavy traffic as the number of servers~$N$
grows large but that the load is lighter than that in the Halfin-Whitt regime,
which corresponds to $\alpha = 1/2$.
Specifically, it is established in~\cite{LY19} that the steady-state
probability of wait for the JSQ($d(N)$) policy
with $d(N) \geq\frac{1}{\beta} N^\alpha \log N$ vanishes as $N \to \infty$.
The results of~\cite{BFL18} imply that when $\beta = 1$
and $d(N) = \lfloor N^\gamma \rfloor$ with $\alpha, \gamma \in (0, 1]$,
$k = \lceil (1 - \alpha)/ \gamma \rceil$, and $2 \alpha + \gamma (k-1) > 1$,
with probability tending to~$1$ as $N \to \infty$, the proportion
of queues with queue length equal to~$k$ is at least
$1 - 2 N^{- 1 + \alpha + (k-1) \gamma}$ and there are no longer queues.
A crucial distinction between the result stated
in Theorem~\ref{diffusionjsqd} and the results in~\cite{BFL18,LY19}
is that the former analyzes the system on diffusion scale
(and describes its behavior in terms of a limiting diffusion process),
whereas \cite{BFL18,LY19} analyze the system on fluid-scale
(and characterize its behavior in terms of limiting fluid-scaled
occupancy state).
Much less is known when the asymptotic load is higher than the 
Halfin-Whitt regime, that is,
when $N^\alpha (1 - \lambda(N) / N) \to \beta>0$ as $N \to \infty$
with $\alpha \in (1/2, 1)$.
This is also known as the \emph{super-Halfin-Whitt regime}.
In this regime, when the system has a finite buffer capacity,
\cite{LY19b} identifies a broad class of load balancing policies
including   the JSQ policy, idle-one-first (I1F) policy,
and the JSQ($d(N)$) policy with $d(N) \geq N^\alpha\log^2N$, for which,
in steady state, $\E(Q_2^N(\infty))$ is $O\big(N^{\alpha}\log N\big)$
and $\E(Q_3^N(\infty))$ is $O\big(N^{-r(1-\alpha) -1}\big)$,
where $r>0$ can be any constant independent of~$N$.
Further,~\cite{ZBM21} analyzes the process-level and steady-state
limits of the occupancy process under the JSQ policy in the
super-Halfin-Whitt regime and in particular, shows that
$Q_2^N(\infty)/N^{\alpha}$ converges weakly to a Gamma$(2, \beta)$
distribution (sum of two independent Exponential$(\beta)$ distributions). 
Results in~\cite{BBD20} allow for arbitrary growth rate of $d(N)$
in the analysis of JSQ($d(N)$) policy in the heavy-traffic regime. 
In this paper, the authors establish a process-level diffusion limit
of the occupancy process under the JSQ($d(N)$) policy
for certain ranges of $\lambda(N)$ that depend on $d(N)$.
In particular, they include an alternative proof of the universality
result in Theorem~\ref{diffusionjsqd}.
\end{remark}

\subsection{High-level outline of proof approach}
\label{ssec:pf-idea-1}

The proofs of both Theorems~\ref{fluidjsqd} and~\ref{diffusionjsqd}
rely on a stochastic coupling construction to bound the difference
in the queue length processes between the JSQ policy and a scheme
with an arbitrary value of $d(N)$.  
This coupling is then exploited to obtain the fluid and diffusion
limits of the JSQ($d(N)$) policy, along with the associated fixed point,
under the conditions stated in Theorems~\ref{fluidjsqd}
and~\ref{diffusionjsqd}.
Moreover, we will also allow the possibility that the servers have
a finite buffer capacity~$B$.
In that case, whenever a task is assigned to a server that has
$B$~tasks in the queue (including the one currently in service),
that task is lost forever.
For an LBA $\Pi$, we will denote the total number of tasks lost
up to time $t$ by $L^\Pi(t)$.

A direct comparison between the JSQ$(d(N))$ scheme and the ordinary JSQ
policy is not straightforward, which is why the $\CJSQ(n(N))$ class of schemes
is introduced as an intermediate scenario to establish the universality results.
Just like the JSQ$(d(N))$ scheme, the schemes in the class
$\CJSQ(n(N))$ may be thought of as ``sloppy'' versions of the JSQ policy,
in the sense that tasks are not necessarily assigned to a server
with the shortest queue length but to one of the $n(N) + 1$ lowest
ordered servers, as graphically illustrated in Figure~\ref{fig:sfigCJSQ}.
In particular, for $n(N) = 0$, the class only includes the ordinary JSQ policy. 
Note that the JSQ$(d(N))$ scheme is guaranteed to identify the lowest
ordered server, but only among a randomly sampled subset of $d(N)$ servers.
In contrast, a scheme in the $\CJSQ(n(N))$ class only guarantees that
one of the $n(N)+1$ lowest ordered servers is selected, but across the
entire pool of $N$ servers. 
It is worthwhile to note that $\CJSQ(n(N))$ is a class of policies,
and that \emph{any} policy which ensures that tasks are always assigned
to one of the $n(N)+1$ lowest ordered servers,
no matter what the exact mechanism of the policy is, belongs to this class.
The proof of the universality results in Theorems~\ref{fluidjsqd}
and~\ref{diffusionjsqd} has two parts, as further described below. 
The proof strategy is schematically represente
in Figure~\ref{fig:sfigRelation}.

\begin{figure}
\begin{center}
\begin{subfigure}{.4\textwidth}
  \centering
  \includegraphics[scale=.55]{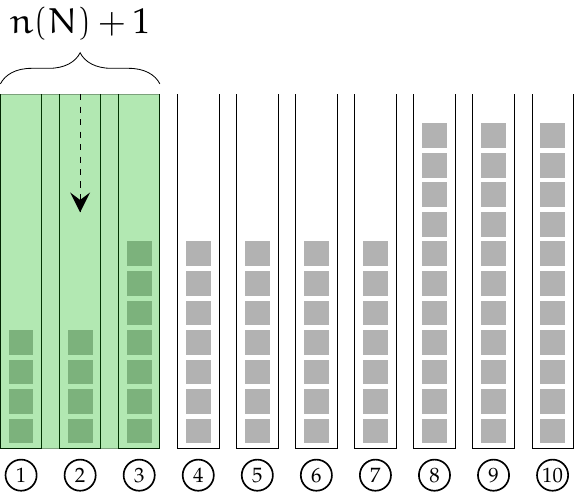}
  \caption{$\CJSQ(n(N))$ scheme\vspace{16pt}}
  \label{fig:sfigCJSQ}
\end{subfigure}
\begin{subfigure}{.6\textwidth}
  \centering
  \includegraphics[scale=1.1]{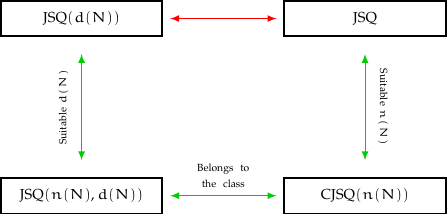}
  \caption{Asymptotic equivalence relations}
  \label{fig:sfigRelation}
\end{subfigure}
\caption{(a) High-level view of the $\CJSQ(n(N))$ class of schemes,
where as in Figure~\ref{figB}, the servers are arranged
in nondecreasing order of their queue lengths, and the arrival
must be assigned through the  green shaded region on the left.
(b) The equivalence structure is depicted for various intermediate
load balancing schemes to facilitate the comparison
between the JSQ$(d(N))$ scheme and the ordinary JSQ policy.
}
\label{fig:strategy}
\end{center}
\end{figure}

\paragraph{Step~1. Performance of schemes in CJSQ($n(N)$) class.}
The first step is to show that for sufficiently small $n(N)$,
\emph{any} scheme from the class $\CJSQ(n(N))$ is still `close'
to the ordinary JSQ policy. 
To achieve this, another type of sloppiness will be introduced.
Let MJSQ$(n(N))$ be a particular scheme that always assigns incoming
tasks to precisely the $(n(N)+1)$-th ordered server. 
Notice that this scheme is effectively the JSQ policy when the system
always maintains $n(N)$ idle servers, or equivalently, uses only
$N - n(N)$ servers, and $\MJSQ(n(N)) \in \CJSQ(n(N))$\footnote{Reviewer:
The sentence `Notice . . . when the system always maintains n(N) idle servers . . . ’ is ambiguous and unclear.}. 
For brevity, we will often suppress $n(N)$ in the notation
where it is clear from the context.
We call any two systems \emph{S-coupled}, if they have synchronized
arrival clocks and departure clocks of the $k$-th longest queue,
for $1 \leq k \leq N$ (`S' in the name of the coupling stands for `Server').
Note that the S-coupling between two systems with identical arrival
and service rates always exists.
Indeed, since tasks have identically and exponentially distributed
service time requirements, synchronizing the departure clocks
of the $k$-th longest queue, for $k = 1, \ldots, N$,
preserves the marginal dynamics of each system.
Consider three S-coupled systems following respectively the JSQ policy,
any scheme from the class $\CJSQ$, and the $\MJSQ$ scheme. 
Recall that $Q_i^\Pi(t)$ is the number of servers with at least
$i$~tasks at time~$t$ and $L^\Pi(t)$ is the total number of lost tasks
up to time~$t$, for the schemes $\Pi=$ JSQ, $\CJSQ$, $\MJSQ$. 
The following proposition provides a stochastic ordering for any scheme
in the class CJSQ with respect to the ordinary JSQ policy and the MJSQ scheme.

\begin{proposition}
\label{prop:stoch-ord}
Fix any $N\geq 1$, $1\leq B\leq \infty$ and $0\leq n(N)\leq N-1$.
Then, in the joint probability space constructed by the S-coupling
of the three systems under respectively \emph{JSQ}, \emph{MJSQ},
and any scheme from the class \emph{CJSQ}, the following ordering
is preserved almost surely throughout the sample path:
for all $1\leq m \leq B$ and $t\geq 0$,
\begin{enumerate}[{\normalfont(i)}] 
\item\label{item:jsq-cjsq} $
\left\{\sum_{i=m}^{B} Q_i^{\JSQ}(t) + L^{\JSQ}(t)\right\}_{t \geq 0} \leq
\left\{\sum_{i=m}^{B} Q_i^{\CJSQ}(t) + L^{\CJSQ}(t)\right\}_{t \geq 0},
$
\item\label{item:cjsq-mjsq} $\left\{\sum_{i=m}^{B} Q_i^{\CJSQ}(t) +
L^{\CJSQ}(t)\right\}_{t \geq 0}\leq
\left\{\sum_{i=m}^{B} Q_i^{\MJSQ}(t) + L^{\MJSQ}(t)\right\}_{t \geq 0},$
\end{enumerate}
provided the inequalities hold at time $t = 0$.
\end{proposition}

\begin{corollary}
\label{cor:bound}
Under the conditions of Proposition~\ref{prop:stoch-ord},
for all $1\leq m \leq B$ and $t\geq 0$,
\begin{enumerate}[{\normalfont(i)}]
\item $Q_m^{\CJSQ}(t) \geq \sum_{i=m}^{B} Q_i^{\JSQ}(t) -
\sum_{i=m+1}^{B} Q_i^{\MJSQ}(t) + L^{\JSQ}(t) - L^{\MJSQ}(t)$,
\item $Q_m^{\CJSQ}(t) \leq \sum_{i=m}^{B} Q_i^{\MJSQ}(t) -
\sum_{i=m+1}^{B} Q_i^{\JSQ}(t) + L^{\MJSQ}(t)-L^{\JSQ}(t),$
\end{enumerate}
provided the inequalities hold at time $t = 0$.
\end{corollary}
\footnote{Reviewer: I would personally make inequalities (i) and (ii) of Prop 4.4
into one big inequality, but this is a matter of personal taste.
It is true that it is awkward to do similarly with inequalities (i) and (ii) of Cor 4.5.}
It can be shown that if $n(N) / N \to 0$ as $N \to \infty$,
then the MJSQ$(n(N))$ scheme has the same fluid limit
along any subsequence as the ordinary JSQ policy, whenever the latter exists.
Corollary~\ref{cor:bound} then implies that as long as $n(N) / N \to 0$,
\emph{any} scheme from the class $\CJSQ(n(N))$ has the same fluid limit
along any subsequence as the ordinary JSQ policy, whenever the latter exists.

\paragraph{Step 2. JSQ($d(N)$) has same limit as a particular scheme
in CJSQ($n(N)$).}
The next step is to prove that for sufficiently large $d(N)$ relative to $n(N)$,
one can construct a scheme belonging to the $\CJSQ(n(N))$ class,
which differs `negligibly' from the JSQ$(d(N))$ scheme. 
Specifically, consider the JSQ$(n(N),d(N))$ scheme with $n(N), d(N) \leq N$,
which is an intermediate blend between the CJSQ$(n(N))$ schemes
and the JSQ$(d(N))$ scheme.
At its first step, just as in the JSQ$(d(N))$ scheme, the JSQ$(d(N),n(N))$
scheme first chooses the shortest of $d(N)$~random candidates
but only sends the arriving task to that server's queue
if it is one of the $n(N)+1$ shortest queues. 
If it is not, then at the second step it picks any of the $n(N)+1$ shortest
queues uniformly at random and then sends the task to that server's queue. 
Note that by construction, JSQ$(d(N),n(N))$ is a scheme in CJSQ$(n(N))$.
Consider two S-coupled systems with a JSQ$(d(N))$ and a JSQ$(n(N), d(N))$ scheme.
Assume that at some specific arrival epoch, the incoming task is dispatched
to the $k$-th ordered server in the system under the JSQ($d(N)$) scheme.
If $k \in \{1, 2, \ldots, n(N) + 1\}$, then the system under the JSQ$(n(N),d(N))$
scheme also assigns the arriving task to the $k$-th ordered server. 
Otherwise, it dispatches the arriving task uniformly at random
amongst the first $(n(N)+1)$ ordered servers.

Next, it is established that if $d(N) \to \infty$,
then for \emph{some} $n(N)$ with $n(N) / N \to 0$, the JSQ$(d(N))$ scheme
and the JSQ$(n(N), d(N))$ scheme have the same fluid limit.
Theorem~\ref{fluidjsqd} then follows by Step~1 and observing that the
JSQ$(n(N),d(N))$ scheme belongs to the class $\CJSQ(n(N))$. \\

The proof of Theorem~\ref{diffusionjsqd} follows the same arguments,
but uses the candidate $n(N) / \sqrt{N}\to 0$ (instead of $n(N) / N \to 0$)
in Step~1, and the candidate $d(N) / (\sqrt{N} \log(N)) \to \infty$
(instead of $d(N) \to \infty$) in Step~2.

\subsection{Extension to batch arrivals}

Consider an extension of the model in which tasks arrive in batches.
We assume that the batches arrive as a Poisson process of rate
$\lambda(N) / \ell(N)$, and have fixed size $\ell(N) > 0$,
so that the effective total task arrival rate remains $\lambda(N)$. 
We will show that for any growing batch size fluid-level optimality
can be achieved with O($1$) communication overhead per task. 
For that, we define the JSQ($d(N)$) scheme adapted for batch arrivals:
When a batch arrives, the dispatcher samples $d(N) \geq \ell(N)$
servers without replacement, and assigns the tasks to the $\ell(N)$
servers with the smallest queue lengths among the sampled servers.

\begin{theorem}[Batch arrivals]
\label{th:batch}
Consider the batch arrival scenario with growing batch size
$\ell(N)\to\infty$ and $\lambda(N)/N\to\lambda<1$ as $N\to\infty$. 
For the $\jsq(d(N))$ scheme with $d(N)\geq \ell(N)/(1-\lambda-\varepsilon)$
for any fixed $\varepsilon>0$, if $q^{N}_1(0)\to q_1(0)\leq \lambda$,
and $q_i^{N}(0)\to 0$ for all $i\geq 2$, then any (subsequential)
weak limit of the sequence of processes $\big\{\qq^{N}(t)\big\}_{t\geq 0}$
coincides with that of the ordinary JSQ policy, and in particular,
is given by the system in~\eqref{eq:fluid}.  
\end{theorem}

Observe that for a fixed $\varepsilon>0$, the communication overhead per task
is on average given by $(1-\lambda-\varepsilon)^{-1}$ which is O($1$).
Thus Theorem~\ref{th:batch} ensures that in case of batch arrivals
with growing batch size, fluid-level optimality can be achieved
with O($1$) communication overhead per task.
The result for the fluid-level optimality in stationarity can also be
obtained indirectly by exploiting the fluid-limit result in~\cite{YSK15}.
Specifically, it can be deduced from the result in~\cite{YSK15} that
for batch arrivals with growing batch size, the JSQ$(d(N))$ scheme
with suitably growing $d(N)$ yields the same fixed point of the fluid
limit as described in~\eqref{eq:fpjsq}.

\section{Blocking and infinite-server dynamics}
\label{bloc}
\footnote{Reviewer: The terminology `infinite-server’ is misleading;
more appropriately, `many-server loss’.}

The basic scenario that we have focused on so far involved
single-server queues.
In this section we turn attention to a system with parallel server pools,
each with a fixed number $B$~servers,
where $B$ can possibly be either finite or infinite.
As before, tasks arrive at a single dispatcher and must immediately
be forwarded to one of the server pools, but also directly start execution
or be discarded otherwise.
As before, under the JSQ($d$) policy, at each task arrival,
the dispatcher selects $d$~random server pools and assigns the task
to the one with the least number of active tasks.
When $B$ is finite, a task that happens to land on a server pool
with $B$~active tasks is lost forever.
In that case, the maximum total rate at which tasks can be processed
in the system is $B N$, which we assume to be higher
than the total arrival rate $\lambda(N)$.
In other words, when $\lambda(N)/N=\lambda \in \R_+$, we assume $\lambda< B$.
The execution times are assumed to be exponentially distributed,
and do not depend on the number of other tasks receiving service simultaneously.
In order to distinguish it from the single-server queueing dynamics
as considered earlier, the current scenario will henceforth be referred
to as the `infinite-server dynamics'.

As it turns out, the JSQ policy has similar stochastic optimality
properties as in the case of single-server queues, and in particular
stochastically minimizes the cumulative number of discarded tasks
\cite{STC93,J89,M87,MS91}.
However, the JSQ policy also suffers from a similar scalability issue
due to the excessive communication overhead in large-scale systems,
which can be mitigated through JSQ($d$) policies.
Results in~\cite{Turner98} and the more recent papers
\cite{KMM17,MKMG15,MMG15,XDLS15} indicate that JSQ($d$) policies
provide similar \emph{power-of-choice} gains for loss probabilities.
It may be shown though that the optimal performance of the JSQ policy
cannot be matched for any fixed value of~$d$.

Motivated by these observations, we explore the trade-off between
performance and communication overhead for infinite-server dynamics.
We will demonstrate that the optimal performance of the JSQ policy
can be asymptotically retained while drastically reducing the
communication burden, mirroring the universality properties described
in Section~\ref{univ} for single-server queues.
The results presented in the remainder of the section are extracted
from~\cite{MBLW16-4} where also the complete proofs are provided,
unless indicated otherwise.

\subsection{Fluid limit for JSQ policy}
\label{ssec:jsqfluid-infinite}

Analogous to the single-server case, we represent the state of the
$N$-th system by the vector $\QQ^N(t) := (Q_1^N(t), Q_2^N(t), \ldots)$
with $Q_i^N(t)$ denoting the number of server pools
with $i$ or more active tasks at time~$t$, and the fluid-scaled occupancy
state is denoted by $\qq^N(t) := (q_1^N(t), q_2^N(t), \ldots)$,
with $q_i^N(t) = Q_i^N(t) / N$ for $i \geq 1$. 
Also, as in Subsection~\ref{ssec:jsqfluid}, for any fluid state
$\qq \in \cS$, denote by $m(\qq) = \min\{i\geq 0: q_{i + 1} < 1\}$ the
minimum number of active tasks among all server pools
with the convention that $q_{B+1} = 0$ if $B<\infty$.
Now if $m(\qq) = 0$, then define $p_0(\qq) = 1$ and $p_i(\qq) = 0$
for all $i = 1, 2, \ldots$. 
Otherwise, in case $m(\qq) > 0$, define
\begin{equation}
\label{eq:fluid-prob-infinite}
p_{i}(\qq) =
\begin{cases}
\min\big\{m(\qq)(1 - q_{m(\qq) + 1})/\lambda,1\big\} & \quad \mbox{ for }
\quad i=m(\qq)-1, \\
1 - p_{ m(\qq) - 1}(\qq) & \quad \mbox{ for } \quad i=m(\qq), \\
0 & \quad \mbox{ otherwise.}
\end{cases}
\end{equation}
{\em Any weak limit of the sequence of processes $\{\qq^N(t)\}_{t \geq 0}$
is given by a deterministic system $\{\qq(t)\}_{t \geq 0}$ satisfying
the following of differential equations
\begin{equation}
\label{eq:fluid-infinite}
\frac{\dif^+ q_i(t)}{\dif t} =
\lambda p_{i-1}(\qq(t)) - i (q_i(t) - q_{i+1}(t)),
\quad i = 1, 2, \dots, B
\end{equation}
where $\dif^+/\dif t$ denotes the right-derivative.}

Equations~\eqref{eq:fluid-prob-infinite} and \eqref{eq:fluid-infinite}
are to be contrasted with Equations~\eqref{eq:fluid-gen} and~\eqref{eq:fluid}.
While the form of the evolution equations~\eqref{eq:fluid-infinite}
of the limiting dynamical system remains similar to~\eqref{eq:fluid},
the rate of decrease of $q_i$ is now $i (q_i - q_{i+1})$,
reflecting the infinite-server dynamics.

Let $K := \lfloor \lambda \rfloor$ and $f := \lambda - K$ denote the
integral and fractional parts of~$\lambda$, respectively.
Assuming $\lambda < B$, the unique fixed point of the dynamical system
in~\eqref{eq:fluid-infinite} is given by
\begin{equation}
\label{eq:fixed-point-infinite}
q_i^\star = \left\{\begin{array}{ll} 1 & i = 1, \dots, K \\
f & i = K + 1 \\
0 & i = K + 2, \dots, B, \end{array} \right.
\end{equation}
and thus $\sum_{i = 1}^{B} q_i^\star = \lambda$.
This is consistent with the results in \cite{MKMG15,MMG15,XDLS15}
for fixed~$d$, where taking $d \to \infty$ yields the same fixed point.
However, the results in \cite{MKMG15,MMG15,XDLS15} cannot be directly
used to handle joint scalings, and do not yield the universality
of the entire fluid-scaled sample path for arbitrary initial states.
The fixed point in~\eqref{eq:fixed-point-infinite}, in conjunction
with an interchange of limits argument, indicates that in stationarity
the fraction of server pools with at least $K + 2$ and at most $K - 1$
active tasks is negligible as $N \to \infty$.

\subsection{Diffusion limit for JSQ policy}
\label{ssec:jsq-diffusion-infinite}

As it turns out, the diffusion-limit results may be qualitatively different,
depending on whether $f = 0$ or $f > 0$,
and we will distinguish between these two cases accordingly.
Observe that for any assignment scheme, in the absence of overflow events,
the total number of active tasks evolves as the number of jobs
in an M/M/$\infty$ system with arrival rate $\lambda(N)$ and unit
service rate, for which the diffusion limit is well-known~\cite{Robert03}.
For the JSQ policy we can establish, for suitable initial conditions,
that the total number of server pools with $K - 2$ or less and $K + 2$
or more tasks is negligible on the diffusion scale.
If $f > 0$, the number of server pools with $K - 1$ tasks is negligible
as well, and the dynamics of the number of server pools with $K$
or $K + 1$ tasks can then be derived from the known diffusion limit
of the total number of tasks mentioned above.
In contrast, if $f = 0$, the number of server pools with $K - 1$ tasks
is not negligible on the diffusion scale, and the limiting behavior is
qualitatively different, but can still be characterized.

\subsubsection{Diffusion-limit results for non-integral
\texorpdfstring{$\boldsymbol{\lambda}$}{lambda}}

We first consider the case $f > 0$, and define $f(N) := \lambda(N) - K N$.
Based on the above observations,
we define the following centered and scaled processes:
\begin{equation}
\label{eq:inf-diff-scale-1}
\begin{split}
\bar{Q}^N_i(t)&=N-Q^N_i(t)\geq 0\quad \mathrm{for}\quad i\leq K-1,\\
\bar{Q}_{K}^N(t)&:=\frac{N-Q_K^N(t)}{\log (N)}\geq 0,\\
\bar{Q}_{K+1}^N(t)&:=\frac{Q^N_{K+1}(t)-f(N)}{\sqrt{N}}\in\R,\\
\bar{Q}^N_i(t)&:=Q^N_i(t)\geq 0\quad \mathrm{for}\quad i\geq K+2.
\end{split}
\end{equation}

\begin{theorem}[{Diffusion limit for JSQ policy; $f>0$}]
\label{th:diffusion-infinite}
Assume $\bar{Q}^N_i(0)$ converges to $\bar{Q}_i(0)$ in $\R$, $i \geq 1$,
and $\lambda(N)/N \to \lambda > 0$ as $N \to \infty$. Then
\begin{enumerate}[{\normalfont(i)}]
\item $\lim_{N\to\infty}\Pro{\sup_{t\in[0,T]}\bar{Q}_{K-1}^N(t)\leq 1}=1$,
and $\big\{\bar{Q}^N_i(t)\big\}_{t\geq 0}$ converges weakly
to $\big\{\bar{Q}_i(t)\big\}_{t\geq 0}$, where $\bar{Q}_i(t)\equiv 0$,
provided  $\lim_{N\to\infty}\Pro{\bar{Q}_{K-1}^N(0)\leq 1}=1$,
and $\bar{Q}_i^N(0)\pto 0$ for $i\leq K-2$.
\item $\big\{\bar{Q}^N_K(t)\big\}_{t\geq 0}$ is
a stochastically bounded sequence of processes.
\item $\big\{\bar{Q}^N_{K+1}(t)\big\}_{t\geq 0}$ converges weakly
to $\big\{\bar{Q}_{K+1}(t)\big\}_{t\geq 0}$,
where $\bar{Q}_{K+1}(t)$ is given by the Ornstein-Uhlenbeck process
satisfying the following stochastic differential equation:
$$d\bar{Q}_{K+1}(t)=-\bar{Q}_{K+1}(t)dt+\sqrt{2\lambda}dW(t),$$
where $W(t)$ is the standard Brownian motion,
provided $\bar{Q}_{K+1}^N(0)$ converges to $\bar{Q}_{K+1}(0)$ in $\mathbb{R}$.
\item For $i\geq K+2$, $\big\{\bar{Q}^N_i(t)\big\}_{t\geq 0}$ converges weakly
to $\big\{\bar{Q}_i(t)\big\}_{t\geq 0}$,
where $\bar{Q}_i(t)\equiv 0$, provided $\bar{Q}_i^N(0)$ converges to 0 in $\R$. \\
\end{enumerate}
\end{theorem}

Theorem~\ref{th:diffusion-infinite} implies that for suitable initial states,
for large~$N$, there will be almost no server pool with $K - 2$ or less tasks
and $K + 2$ or more tasks on any finite time interval. 
Also, the number of server pools having fewer than $K$~tasks is
of order $\log(N)$, and there are $f N + O_P(\sqrt{N})$ server pools
with precisely $K + 1$ active tasks.
Below we present some high-level intuition behind the scaling limits
in Theorem~\ref{th:diffusion-infinite}.

\paragraph{High-level proof idea.}
Observe that $\sum_{i=1}^{K} (N-Q_i^N(\cdot))$ increases by one at rate 
\[
\sum_{i=1}^{K} i (Q_i(t)-Q_{i+1}(t)) =
\sum_{i=1}^{K} (Q_i(t)-Q_{K+1}(t)) \approx K (1 - f) N,
\] 
which is when there is a departure from some server pool with at most
$K$~active tasks, and if positive, decreases by one at constant rate
$\lambda(N) = (K + f) N + o(N)$, which is whenever there is an arrival.
Thus, $\sum_{i=1}^{K} (N - Q_i^N(\cdot))$ roughly behaves as a birth-and-death
process with birth rate $K (1 - f) N$ and death rate $(K + f) N$. 
Since $f > 0$, we have $K + f > K (1 - f)$, and on any finite time interval
the maximum of such a birth-and-death process scales as $\log(N)$.

Similar to the argument above, the process $\sum_{i=1}^{K-1} \bQ_i^N(\cdot)$
increases by one at rate 
\begin{align*}
\sum_{i=1}^{K-1} i (Q^N_i(t)-Q^N_{i+1}(t)) &=
\sum_{i=1}^{K-1} Q^N_i(t) - (K-1) Q_K^N(t) \\
&\leq (K-1) (N-Q_K^N(t)) = O(\log(N)),
\end{align*}
which is when there is a departure from some server pool with at most
$K - 1$ active tasks, and if positive, decreases by one at rate $\lambda(N)$,
which is whenever there is an arrival.
Thus, $\sum_{i=1}^{K-1} \bQ_i^N(\cdot)$ roughly behaves as a birth-and-death
process with birth rate $O(\log(N))$ and death rate $O(N)$.
This leads to the asymptotic result for $\sum_{i=1}^{K-1} \bQ_i^N(\cdot)$,
and in particular for $\bQ_{K-1}^N(\cdot)$.
This completes the proof of Parts~(i) and~(ii)
of Theorem~\ref{th:diffusion-infinite}.

Furthermore, since $\lambda < K + 1$, the number of tasks that are assigned
to server pools with at least $K + 1$ tasks converges to zero
in probability and this completes the proof of Part (iv)
of Theorem~\ref{th:diffusion-infinite}.
 
Finally, all the above combined also means that on any finite time interval
the total number of tasks in the system behaves with high probability
as the total number of jobs in an M/M/$\infty$ system. 
Therefore with the help of the diffusion limit result
for the M/M/$\infty$ system in~\cite[Theorem 6.14]{Robert03},
we conclude the proof of Part (iii) of Theorem~\ref{th:diffusion-infinite}.

\subsubsection{Diffusion-limit results for integral
\texorpdfstring{$\boldsymbol{\lambda}$}{lambda}}

We now turn to the case $f = 0$, and assume that
\begin{equation}
\label{eq:f=0}
\frac{K N - \lambda(N)}{\sqrt{N}} \to \beta \in \R \quad \mbox{ as }
\quad N \to \infty,
\end{equation} 
which can be thought of as an analog of the Halfin-Whitt regime
in~\eqref{eq:HW}.
We now consider the following scaled quantities:
\begin{equation}
\label{eq:inf-diff-scale-2}
\begin{split}
\zeta_1^N(t) := \frac{1}{\sqrt{N}}\sum_{i = 1}^{K} (N - Q_i^N(t)), \qquad
\zeta_2^N(t) := \frac{Q_{K+1}^N(t)}{\sqrt{N}}.
\end{split}
\end{equation}

\begin{theorem}
\label{th: f=0 diffusion}
Assuming the convergence of initial states, the process
$\big\{(\zeta_1^N(t), \zeta_2^N(t))\big\}_{t \geq 0}$ converges weakly
to the process $\big\{(\zeta_1(t), \zeta_2(t))\big\}_{t \geq 0}$
governed by the system of SDEs
\begin{align*}
\dif\zeta_1(t) &= \sqrt{2K} \dif W(t) - (\zeta_1(t) + K \zeta_2(t)) +
\beta \dif t + \dif V_1(t) \\
\dif\zeta_2(t) &= \dif V_1(t) - (K + 1) \zeta_2(t),
\end{align*}
where $W$ is the standard Brownian motion, and $V_1(t)$ is the unique
continuous non-decreasing process satisfying
$\int_0^t \ind{\zeta_1(s) > 0} \dif V_1(s) = 0$ and $V_1(0) =0$.
\end{theorem}

Unlike the $f > 0$ case, the above theorem says that if $f = 0$,
then over any finite time horizon, there will be $O_P(\sqrt{N})$
server pools with fewer than $K$ or more than $K$~active tasks,
and hence most of the server pools have precisely $K$~active tasks.
The proof of Theorem~\ref{th: f=0 diffusion} uses the reflection
argument developed in~\cite{EG15}.
Indeed, the proof follows by observing that the dynamics
of $\big\{(\zeta_1^N(t), \zeta_2^N(t))\big\}_{t \geq 0}$ resembles
the dynamics of the JSQ policy in the Halfin-Whitt regime.

\subsection{Universality of JSQ(d) policies in infinite-server dynamics}
\label{ssec:univ-infinite}

As in Section~\ref{univ}, we now further explore the trade-off
between performance and communication overhead as a function
of the diversity parameter~$d(N)$, in conjunction with the load.
We will specifically investigate what growth rate of $d(N)$ is required,
depending on the scaling behavior of $\lambda(N)$,
in order to asymptotically match the optimal performance of the JSQ policy.

\begin{theorem}[Universality of fluid limit for JSQ($d(N)$) and infinite-server dynamics]
\label{fluidjsqd-infinite}
If $d(N) \to \infty$ as $N \to \infty$, then any (subsequential) fluid limit
of the JSQ$(d(N))$ scheme coincides with that of the ordinary JSQ policy,
and in particular, satisfies the system of differential equations
in~\eqref{eq:fluid-infinite}. 
Consequently, the stationary occupancy states converge to the unique fixed point
as in~\eqref{eq:fixed-point-infinite}.
\end{theorem}

In order to state the universality result on diffusion scale,
define in case $f > 0$,
\begin{equation}
\label{eq:inf-diff-scale-3}
\begin{split}
\bar{Q}_i^{d(N)}(t) &:= \dfrac{N - Q_i^{d(N)}(t)}{\sqrt{N}} \geq 0,
\quad i \leq K, \\  
\bar{Q}_{K+1}^{d(N)}(t) &:= \dfrac{Q_{K+1}^{d(N)}(t) - f(N)}{\sqrt{N}}
\in \R, \\ 
\bar{Q}_i^{d(N)}(t) &:= \frac{Q_i^{d(N)}(t)}{\sqrt{N}}\geq 0,
\quad \text{ for } \quad i \geq K + 2,
\end{split}
\end{equation}
and otherwise, if $f = 0$,
\begin{equation}
\label{eq:inf-diff-scale-4}
\begin{split}
\hQ_{K-1}^{d(N)}(t) &:=
\sum_{i=1}^{K-1} \dfrac{N - Q_i^{d(N)}(t)}{\sqrt{N}} \geq 0, \\
\hQ_K^{d(N)}(t) &:= \dfrac{N - Q_K^{d(N)}(t)}{\sqrt{N}} \geq 0, \\
\hQ_i^{d(N)}(t) &:= \dfrac{Q_i^{d(N)}(t)}{\sqrt{N}} \geq 0, \quad
\text{ for } \quad i \geq K + 1.
\end{split}
\end{equation}
The scaling in Equations~\eqref{eq:inf-diff-scale-3}
and~\eqref{eq:inf-diff-scale-4} should be contrasted
with Equations~\eqref{eq:inf-diff-scale-1} and~\eqref{eq:inf-diff-scale-2},
respectively.

\begin{theorem}[Universality of diffusion limit for JSQ($d(N)$) and infinite-server dynamics]
\label{diffusionjsqd-infinite}
Assume $d(N) / (\sqrt{N} \log N) \to \infty$. 
Under suitable initial conditions
\begin{enumerate}[{\normalfont (i)}]
\item If $f>0$, then $\bQ_i^{d(N)}(\cdot)$ converges to the zero process
for $i\neq K+1$, and $\bQ^{d(N)}_{K+1}(\cdot)$ converges weakly
to the Ornstein-Uhlenbeck process satisfying the SDE
$$d\bar{Q}_{K+1}(t)=-\bar{Q}_{K+1}(t)dt+\sqrt{2\lambda}dW(t),$$
where $W(t)$ is the standard Brownian motion.
\item If $f=0$, then $\hQ_{K-1}^{d(N)}(\cdot)$ converges weakly
to the zero process, and $(\hQ_{K}^{d(N)}(\cdot), \hQ_{K+1}^{d(N)}(\cdot))$ converges weakly
to $(\hQ_{K}(\cdot), \hQ_{K+1}(\cdot))$, described by the unique solution of the system of SDEs
\begin{align*}
\dif\hQ_{K}(t) &= \sqrt{2 K} \dif W(t) - (\hQ_K(t) + K \hQ_{K+1}(t)) +
\beta \dif t + \dif V_1(t) \\
\dif\hQ_{K+1}(t) &= \dif V_1(t) - (K + 1) \hQ_{K+1}(t),
\end{align*}
where $W$ is the standard Brownian motion, and $V_1(t)$ is the unique
continuous non-decreasing process satisfying
$\int_0^t \ind{\hQ_K(s)\geq 0} \dif V_1(s) = 0$ and $V_1(0) = 0$.
\end{enumerate}
\end{theorem}

Having established the asymptotic results for the JSQ policy
in Sections~\ref{ssec:jsqfluid-infinite}
and~\ref{ssec:jsq-diffusion-infinite}, the proofs of the asymptotic
results for the JSQ$(d(N))$ scheme in Theorems~\ref{fluidjsqd-infinite}
and~\ref{diffusionjsqd-infinite} involve establishing a universality
result which shows that the limiting processes for the JSQ$(d(N))$
scheme are `asymptotically equivalent' to those for the ordinary JSQ
policy for suitably large values of~$d(N)$.
The notion of asymptotic equivalence between different schemes is
formalized in the next definition.

\begin{definition}
Let $\Pi_1$ and $\Pi_2$ be two schemes parameterized by the number
of server pools~$N$.
For any positive function $g:\N\to\R_+$,
we say that $\Pi_1$ and $\Pi_2$ are `$g(N)$-alike'
if there exists a common probability space,
such that for any fixed $T\geq 0$, for all $i\geq 1$,
$$\sup_{t\in[0,T]}(g(N))^{-1}|Q_i^{\Pi_1}(t)-Q_i^{\Pi_2}(t)|\pto 0\quad \mathrm{as}\quad N\to\infty.$$
\end{definition}

\noindent
Intuitively speaking, if two schemes are $g(N)$-alike,
then in some sense, the associated system occupancy states are
indistinguishable on $g(N)$-scale. 
For brevity, for two schemes $\Pi_1$ and $\Pi_2$ that are $g(N)$-alike,
we will often say that $\Pi_1$ and $\Pi_2$ have the same process-level
limits on $g(N)$-scale.
The next theorem states a sufficient criterion for the JSQ$(d(N))$
scheme and the ordinary JSQ policy to be $g(N)$-alike, and thus,
provides the key vehicle in establishing the universality result.

\begin{theorem}
\label{th:pwr of d}
Let $g: \N \to \R_+$ be a function diverging to infinity.
Then the JSQ policy and the JSQ$(d(N))$ scheme are $g(N)$-alike,
with $g(N) \leq N$, if 
\begin{align}
\label{eq:fNalike cond1}
\mathrm{(i)}&\quad d(N) \to \infty, \quad \text{for} \quad g(N) = \OO(N), \\
\mathrm{(ii)}&\quad d(N)
\left(\frac{N}{g(N)}\log\left(\frac{N}{g(N)}\right)\right)^{-1} \to \infty,
\quad \text{for} \quad g(N) = \oo(N).
\label{eq:fNalike cond2}
\end{align}
\end{theorem}

Theorem~\ref{th:pwr of d} yields the next two immediate corollaries.
\begin{corollary}
\label{cor:fluid}
If $d(N) \to \infty$ as $N \to \infty$, then the JSQ$(d(N))$ scheme
and the ordinary JSQ policy are $N$-alike.
\end{corollary}

\begin{corollary}
\label{cor-diff}
If $d(N) / (\sqrt{N} \log(N)) \to \infty$ as $N \to \infty$, then the
JSQ$(d(N))$ scheme and the ordinary JSQ policy are $\sqrt{N}$-alike.
\end{corollary}

Observe that Corollaries~\ref{cor:fluid} and~\ref{cor-diff}
together with the asymptotic results for the JSQ policy
in Sections~\ref{ssec:jsqfluid-infinite}
and~\ref{ssec:jsq-diffusion-infinite} imply
Theorems~\ref{fluidjsqd-infinite} and~\ref{diffusionjsqd-infinite}.
The rest of the section will be devoted to the proof
of Theorem~\ref{th:pwr of d}.
The proof crucially relies on a novel coupling construction,
which will be used to (lower and upper) bound the difference
of occupancy states of two arbitrary schemes.

\paragraph{The coupling construction.}
Throughout the description of the coupling, we fix~$N$, and suppress
the superscript~$N$ in the notation. 
Let $Q_i^{\Pi_1}(t)$ and $Q_i^{\Pi_2}(t)$ denote the number of server
pools with at least $i$~active tasks at time~$t$ in two systems
following schemes $\Pi_1$ and $\Pi_2$, respectively.
With a slight abuse of terminology, we occasionally use $\Pi_1$ and $\Pi_2$
to refer to systems following schemes $\Pi_1$ and $\Pi_2$, respectively.
To couple the two systems, we synchronize the arrival epochs 
and maintain a single exponential departure clock with instantaneous rate
at time $t$ given by $M(t) := \max\left\{\sum_{i=1}^{B} Q_i^{\Pi_1}(t),
\sum_{i=1}^{B} Q_i^{\Pi_2}(t)\right\}$.
We couple the arrivals and departures in the various server pools as follows:

\noindent
\textit{Arrival:}
At each arrival epoch, assign the incoming task in each system to one
of the server pools according to the respective schemes.

\noindent
\textit{Departure:} 
Define $$H(t) :=
\sum_{i=1}^{B} \min\left\{Q_i^{\Pi_1}(t), Q_i^{\Pi_2}(t)\right\}$$
and $$p(t):=
\begin{cases}
\dfrac{H(t)}{M(t)}, & \quad \text{if} \quad M(t) > 0, \\
0, & \quad \text{otherwise.}
\end{cases}
$$
At each departure epoch~$t_k$ (say), draw a uniform$[0,1]$ random
variable $U(t_k)$. 
The departures occur in a coupled way based upon the value of $U(t_k)$.
In either of the systems, assign an active task index $(i,j)$,
if it is the $j$-th task (in the order of arrival) of the $i$-th ordered server pool. 
Let $\mathcal{A}_1(t)$ and $\mathcal{A}_2(t)$ denote the set of all task
indices present at time~$t$ in systems $\Pi_1$ and $\Pi_2$, respectively. 
Color the indices (or tasks) in $\mathcal{A}_1 \cap \mathcal{A}_2$,
$\mathcal{A}_1 \setminus \mathcal{A}_2$
and $\mathcal{A}_2 \setminus \mathcal{A}_1$, green, blue and red,
respectively, and note that $|\mathcal{A}_1 \cap \mathcal{A}_2| = H(t)$. 
Define a total order on the set of indices as follows:
$(i_1,j_1)<(i_2,j_2)$ if $i_1<i_2$, or $i_1=i_2$ and $j_1<j_2$.
Now, if $U(t_k)\leq p(t_k-)$, then select one green index uniformly
at random and remove the corresponding tasks from both systems.
Otherwise, if $U(t_k)> p(t_k-)$, then choose one integer~$m$,
uniformly at random from all the integers between~$1$
and $M(t) - H(t) = M(t) (1 - p(t))$, and remove the tasks corresponding
to the $m$-th smallest (according to the order defined above)
red and blue indices in the corresponding systems.
If the number of red (or blue) tasks is less than~$m$,
then do nothing in the corresponding system.

\begin{figure}
\begin{subfigure}{.4\textwidth}
\centering
\includegraphics[scale=.8]{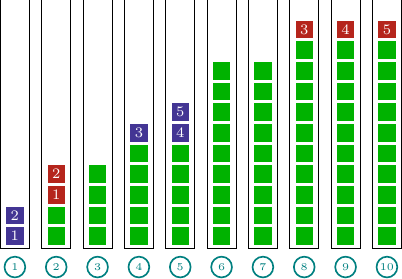}
\caption{T-coupling}
\label{fig:tcoupling}
\end{subfigure}
\begin{subfigure}{.6\textwidth}
\begin{center}
  \begin{tikzpicture}[scale=.6]
  
  \draw(14,1) node[mybox, text width = 2.5cm,text centered]  {%
   {\scriptsize JSQ$(n(N),d(N))$}};
   
  \draw(22,1) node[mybox, text width = 2.5cm,text centered]  {%
   {\scriptsize  $\CJSQ(n(N))$}};
   
   \draw(14,6) node[mybox, text width = 2.5cm,text centered]  {%
   {\scriptsize  JSQ$(d(N))$}};
   
   \draw(22,6) node[mybox, text width = 2.5cm,text centered]  {%
   {\scriptsize  JSQ}};

\draw[doublearr] (16.5,6) to (19.5,6);
\draw[doublearr2] (14,2) to (14,5);
\draw[doublearr2] (16.5,1) to (19.5,1);
\draw[doublearr2] (22,2) to (22,5);

\node  at (18,6.5) {\tiny Theorem~\ref{th:pwr of d}};

\node [rotate=270] at (21.5,3.5) {\tiny Proposition~\ref{prop: modified JSL}};
\node [rotate=270] at (22.5,3.5) {\tiny Suitable $n(N)$};

\node [rotate=90] at (14.5,3.5) {\tiny Proposition~\ref{prop: power of d}};
\node [rotate=90] at (13.5,3.5) {\tiny Suitable $d(N)$};

\draw[text width=1.2cm, align=center] (18,1.75) node {\tiny Belongs to};
\draw[text width=1.1cm, align=center] (18,1.35) node {\tiny the class};
\end{tikzpicture}
\caption{Asymptotic equivalence relations}
\label{fig:sfigrelation-inf}
\end{center}
\end{subfigure}
\caption{(a) Superposition of the occupancy states at some particular time instant,
of  schemes $\Pi_1$ and $\Pi_2$ when the server pools in both systems are arranged
in nondecreasing order of the number of active tasks.
The $\Pi_1$ system is the union of the green and blue tasks,
and the $\Pi_2$ system is the union of the green and red tasks.
(b) The equivalence structure is depicted for various intermediate
load balancing schemes to facilitate the comparison
between the JSQ$(d(N))$ scheme and the ordinary JSQ policy.}
\end{figure}
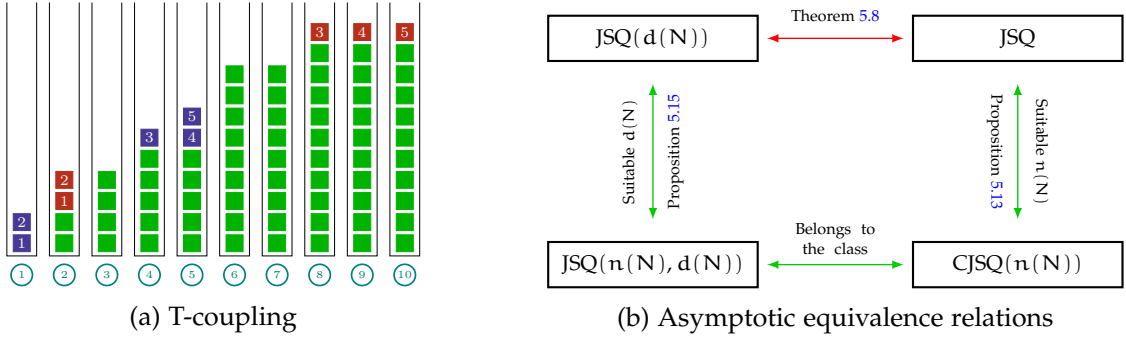

The above coupling has been schematically represented
in Figure~\ref{fig:tcoupling}, and will henceforth be referred to as
T-coupling, where T stands for `task-based'. 
Now we need to show that, under the T-coupling, the two systems,
considered independently, evolve according to their respective marginal statistical laws. 
This can be seen in several steps.
Indeed, the T-coupling basically uniformizes the departure rate
by the maximum number of tasks present in either of the two systems. 
Then informally speaking, the green region signifies the common portion
of tasks, and the red and blue regions represent the separate contributions. 
Without loss of generality, we assume that  $|\mathcal{A}_1| \geq |\mathcal{A}_2|$.
Observe that
\begin{enumerate}[{\normalfont (i)}]
\item The total departure rate from $\Pi_i$ is 
\begin{align*}
& M(t) \left[p(t)+(1-p(t)) \frac{|\mathcal{A}_i\setminus\mathcal{A}_{3-i}|}{M(t)-H(t)}\right]
= |\mathcal{A}_1\cap\mathcal{A}_2|+|\mathcal{A}_i\setminus\mathcal{A}_{3-i}|
=|\mathcal{A}_i|, \quad i = 1, 2.
\end{align*}
\item Since $|\mathcal{A}_1| \geq |\mathcal{A}_2|$, each task in $\Pi_1$
is equally likely to depart.
\item Each task in $\Pi_2$ within $\mathcal{A}_1\cap \mathcal{A}_2$
and each task within $\mathcal{A}_2 \setminus \mathcal{A}_1$
is equally likely to depart, and the probabilities of departures
are proportional to $|\mathcal{A}_1 \cap \mathcal{A}_2|$
and $|\mathcal{A}_2 \setminus \mathcal{A}_1|$, respectively.
\end{enumerate}

The T-coupling can be used to derive several stochastic inequality results
that will play an instrumental role in proving Theorem~\ref{th:pwr of d}.
Recall the CJSQ($n(N)$) class of schemes from Section~\ref{ssec:pf-idea-1}.
From a high-level perspective, the proof follows a somewhat similar structure
as in Section~\ref{ssec:pf-idea-1}.

\paragraph{Step~1. Condition for $g(N)$-alikeness of schemes
in CJSQ($n(N)$) class.}
The next lemma uses T-coupling to compare the occupancy processes
of the JSQ policy with any scheme from the CJSQ($n(N)$) class.

\begin{lemma}
\label{lem:majorization}
Let $Q_i^{\Pi_1}(t)$ and $Q_i^{\Pi_2}(t)$ denote the number of server pools
with at least $i$~tasks in two T-coupled systems under the JSQ policy
and a scheme in the $\CJSQ(n(N))$ class, respectively.
Then, for any $k \in \big\{1, 2, \ldots, B\big\}$,
\begin{equation}
\label{eq:majorization}
\left\{\sum_{i=1}^{k} Q_i^{\Pi_1}(t) - k n(N)\right\}_{t \geq 0} \leq
\left\{\sum_{i=1}^{k} Q_i^{\Pi_2}(t)\right\}_{t \geq 0} \leq
\left\{\sum_{i=1}^{k} Q_i^{\Pi_1}(t)\right\}_{t \geq 0},
\end{equation}
provided the two systems start from the same occupancy states at $t = 0$.
In particular, for all $k \geq 1$, 
\begin{equation}
\label{eq:upper-lower-occ}
\sup_{t\geq 0} \big|Q_k^{\Pi_2}(t) - Q_k^{\Pi_1}(t)\big| \leq k n(N)
\end{equation}
\end{lemma}

\begin{remark}
\label{rem:novelty}
{\normalfont
The stochastic ordering in Lemma~\ref{lem:majorization} is to be contrasted
with the weak majorization results
in \cite{Winston77,towsley,Towsley95,Towsley1992,W78} in the context
of the ordinary JSQ policy in the single-server queueing scenario,
and in~\cite{STC93,J89,M87,MS91} in the scenario of state-dependent
service rates, non-decreasing with the number of active tasks.
In the current infinite-server scenario, the results
in \cite{STC93,J89,M87,MS91} imply that for any non-anticipating
scheme~$\Pi$ taking assignment decisions based on the number of active
tasks only, for all $t \geq 0$,
\begin{align}
\label{eq: towsley}
\sum_{m=1}^{\ell} X_{(m)}^{\rm JSQ}(t) & \leq_{st}
\sum_{m=1}^\ell X_{(m)}^{\Pi}(t), \mbox{ for } \ell = 1, 2, \ldots, N, \\
\left\{L^{\rm JSQ}(t)\right\}_{t \geq 0} &\leq_{st}
\left\{L^{\Pi}(t)\right\}_{t \geq 0},
\end{align}
where $X_{(m)}^\Pi(t)$ is the number of tasks in the $m$-th ordered
server pool at time~$t$ in the system following scheme~$\Pi$
and $L^{\Pi}(t)$ is the total number of overflow events
under policy~$\Pi$ up to time~$t$. 
Observe that $X_{(m)}^\Pi$ can be visualized as the $m$-th largest
(rightmost) vertical bar (or stack) in Figure~\ref{figB}.
Thus~\eqref{eq: towsley} says that the sum of the lengths
of the $\ell$~largest \emph{vertical} stacks in a system following
any scheme~$\Pi$ is stochastically larger than or equal to that
following the ordinary JSQ policy for any $\ell = 1, 2, \ldots, N$.
Mathematically, this ordering can be equivalently written as
\begin{equation}
\label{eq:equiv-ord}
\sum_{i = 1}^{B} \min\big\{\ell, Q_i^{\rm JSQ}(t)\big\} \leq_{st}
\sum_{i = 1}^{B} \min\big\{\ell, Q_i^{\Pi}(t)\big\},
\end{equation}
for all $\ell = 1, \dots, N$.
In contrast, in order to show asymptotic equivalence on various scales,
we need to both upper and lower bound the occupancy states of the
$\CJSQ(n(N))$ schemes in terms of the JSQ policy, and therefore need
a much stronger hold on the departure process.
The T-coupling provides us just that, and has several useful properties
that are crucial for our proof technique.
For example, T-coupling has an important feature that
if two systems are T-coupled, then departures cannot increase the sum
of the absolute differences of the $Q_i$-values, which is not true
for the coupling considered in the above-mentioned literature.
The left stochastic ordering in~\eqref{eq:majorization} also does not
remain valid in those cases.
Furthermore, observe that the right inequality in~\eqref{eq:majorization}
(i.e., $Q_i$'s) implies the stochastic inequality is \emph{reversed}
in~\eqref{eq:equiv-ord}, which is counter-intuitive in view of the
well-established optimality properties of the ordinary JSQ policy.
In the current infinite-server dynamics where there is no queueing,
this can be understood from the intuition that a better LBA has more customers
in service instead of less customers in queue.
The fundamental distinction between the two coupling techniques is
also reflected by the fact that the T-coupling does not allow
for arbitrary nondecreasing state-dependent departure rate functions,
unlike the couplings in \cite{STC93,J89,M87,MS91}.}
\end{remark}

\begin{remark}[Comparison of T-coupling and S-coupling]
\label{rem:contrast}
\normalfont
As briefly mentioned earlier, in the current infinite-server scenario,
the departures of the ordered server pools cannot be coupled,
mainly since the departure rate at the $m^{\rm th}$ ordered server pool,
for some $m = 1, 2, \ldots, N$, depends on its number of active tasks.
It is worthwhile to mention that the T-coupling in the current section
is stronger than the S-coupling used in Section~\ref{univ} in the
single-server queueing scenario.
Observe that due to Lemma~\ref{lem:majorization}, the absolute difference
of the occupancy states of the JSQ policy and any scheme
from the CJSQ class at any time point can be bounded deterministically
(without any terms involving the cumulative number of lost tasks).
It is worth emphasizing that the universality result on some specific scale,
stated in Theorem~\ref{th:pwr of d}, does not depend on the behavior
of the JSQ policy on that scale, whereas in the single-server queueing scenario
it does, mainly because the upper and lower bounds
in Corollary~\ref{cor:bound} involve tail sums of two different policies.
More specifically, in the single-server queueing scenario the fluid
and diffusion limit results of CJSQ($n(N)$) class crucially use those
of the MJSQ($n(N)$) scheme, while in the current scenario it does not --
the results for the MJSQ($n(N)$) scheme comes as a consequence
of those for the CJSQ($n(N)$) class of schemes.
Also, the bounds in Lemma~\ref{lem:majorization} do not depend on~$t$,
and hence, apply in the steady state as well.
Moreover, the S-coupling compares the $k$ \emph{highest} horizontal bars,
whereas the T-coupling in the current section compares the $k$
\emph{lowest} horizontal bars.
As a result, the bounds on the occupancy states established
in Corollary~\ref{cor:bound} involve tail sums of the occupancy states
of the ordinary JSQ policy, which necessitates proving the convergence
of tail sums of the occupancy states of the ordinary JSQ policy. 
In contrast, the bound in the infinite-server scenario involves
only a single component (see Equation~\eqref{eq:upper-lower-occ}),
and thus, proving convergence of each component suffices.
\end{remark}

The goal in the first step is to show that for a suitable choice
of $n(N)$, the schemes in the CJSQ($n(N)$) class are indistinguishable
on suitable scales.
This is formalized in Proposition~\ref{prop: modified JSL} below,
which follows immediately from Lemma~\ref{lem:majorization}.
\begin{proposition}
\label{prop: modified JSL}
For any function $g:\N\to\R_+$ diverging to infinity,
if $n(N)/ g(N)\to 0$ as $N \to \infty$, then the JSQ policy and the
$\CJSQ(n(N))$ schemes are $g(N)$-alike.
\end{proposition}

\paragraph{Step~2. $g(N)$-alikeness of JSQ($d(N)$) and a scheme in CJSQ($n(N)$).}
Next we compare the CJSQ($n(N)$) schemes with the JSQ($d(N)$) scheme. 
The comparison follows a somewhat similar line of argument
as in Section~\ref{ssec:pf-idea-1}, and involves a JSQ$(n(N),d(N))$
scheme which is an intermediate blend between the $\CJSQ(n(N))$
schemes and the JSQ$(d(N))$ scheme.
Specifically, the JSQ$(n(N),d(N))$ scheme selects a candidate server
pool in the exact same way as the JSQ$(d(N))$ scheme.
However, it only assigns the task to that server pool if it belongs
to the $n(N)+1$ lowest ordered ones,
and to a randomly selected server pool among these otherwise.
By construction, the JSQ$(n(N),d(N))$ scheme belongs to the class $\CJSQ(n(N))$.

The next proposition establishes a sufficient criterion on $d(N)$
in order for the JSQ$(d(N))$ scheme and JSQ$(n(N),d(N))$ scheme
to be close in terms of $g(N)$-alikeness.

\begin{proposition}
\label{prop: power of d}
Assume, $n(N) / g(N) \to 0$ as $N \to \infty$ for some function
$g: \N\to \R_+$ diverging to infinity.
The JSQ$(n(N),d(N))$ scheme and the JSQ($d(N)$) scheme are $g(N)$-alike
if the following condition holds:
\begin{equation}
\label{eq:condition-same}
\frac{n(N)}{N}d(N) - \log\frac{N}{g(N)} \to \infty, \quad \text{as}
\quad N \to \infty.
\end{equation}
\end{proposition}

Finally, Proposition~\ref{prop: power of d} in conjunction
with Proposition~\ref{prop: modified JSL} yields Theorem~\ref{th:pwr of d}.
The overall proof strategy as described above, is schematically
represented in Figure~\ref{fig:sfigrelation-inf}.

\section{Load balancing in graph topologies}
\label{networks}

In this section we return to the single-server queueing dynamics,
and extend the universality properties to network scenarios,
where the $N$~servers are assumed to be inter-connected by some
underlying graph topology~$G_N$.
Tasks arrive at the various servers as independent Poisson processes
of rate~$\lambda$, and each incoming task is assigned to whichever
server has the smallest number of tasks amongst the one where it arrives
and its neighbors in~$G_N$. 
Ties are broken arbitrarily.
Thus, in case $G_N$ is a clique, each incoming task is assigned to the
server with the shortest queue across the entire system,
and the behavior is equivalent to that under the JSQ policy.
The stochastic optimality properties of the JSQ policy thus imply that
the queue length process in a clique will be better balanced
and smaller (in a majorization sense) than in an arbitrary graph~$G_N$.

As stated in the introduction, network scenarios are not only of mathematical
interest but also of major relevance from an application perspective.
For example, they emerge in modeling connectivity properties,
geographic restrictions and proximity relations in spatial network settings.
Besides capturing such physical concepts in infrastructure networks,
network scenarios also arise due to `logical relationships',
in particular so-called affinity notions and compatibility constraints
between tasks and servers.
Such features are increasingly common in data centers and cloud networks
due to heterogeneity and data locality issues, see for instance~\cite{WZS20,RM21},
and also relate to the scalability considerations that are important
in load balancing, as further explained below.\\

\noindent
{\bf Sparse graph topologies.}
Besides the prohibitive communication overhead discussed earlier,
a further scalability issue of the JSQ policy arises when executing
a task involves the use of some data.
Storing such data for all possible tasks on all servers will typically
require an excessive amount of storage capacity.
These two burdens can be effectively mitigated in sparser graph
topologies where tasks that arrive at a specific server~$i$ are only
allowed to be forwarded to a subset of the servers ${\mathcal N}_i$.
For the tasks that arrive at server~$i$, queue length information
then only needs to be obtained from servers in ${\mathcal N}_i$,
and it suffices to store replicas of the required data on the servers
in ${\mathcal N}_i$.
The subset ${\mathcal N}_i$ containing the peers of server~$i$
can be naturally viewed as its neighbors in some graph topology~$G_N$.
Here we consider the case of undirected graphs, but most of the analysis
can be extended to directed graphs.

While sparser graph topologies relieve the scalability issues
associated with a clique, the queue length process will be worse
(in the majorization sense) because of the limited connectivity.
Surprisingly, however, even quite sparse graphs can asymptotically
match the optimal performance of a clique, provided they are
suitably random, as we will further describe below.

The above model has been studied in \cite{Gast15,Turner98}, focusing
on certain fixed-degree graphs and in particular ring topologies
for which \cite{Mitzenmacher96} had already presented simulation results.
The results demonstrate that the flexibility to forward tasks
to a few neighbors, or even just one, with possibly shorter queues
significantly improves the performance in terms of the waiting time
and tail distribution of the queue length.
This resembles the \emph{power-of-choice} gains observed for JSQ($d$)
policies in complete graphs.

However, the results in \cite{Gast15,Turner98} also establish that
the performance sensitively depends on the underlying graph topology,
and that selecting from a fixed set of $d - 1$ neighbors typically
does not match the performance of re-sampling $d - 1$ alternate
servers for each incoming task from the entire population,
as in the power-of-$d$ scheme in a complete graph.
Further interesting results for the performance load balancing algorithms 
in a network context, with a focus on tail asymptotics, may be found
in~\cite{FM01,MT00}. \\

\noindent
{\bf Supermarket model on graphs.}
When each arriving task is routed to the shortest of $d \geq 2$
randomly selected neighboring queues, the process-level convergence
over any finite time interval has been established recently in~\cite{BMW17}.
In this work, the authors analyze the evolution of the queue length process
at an arbitrary tagged server as the system size becomes large.
The main ingredient is a careful analysis of local occupancy measures
associated with the neighborhood of each server and to argue that
under suitable conditions their asymptotic behavior is the same for all servers.
Under mild conditions on the graph topology $G_N$ (diverging minimum degree
and the ratio between minimum degree and maximum degree in each connected
component converges to~$1$), for a suitable initial occupancy measure,
\cite[Theorem~2.1]{BMW17} establishes that for any fixed $d \geq 2$,
the global occupancy state process for the JSQ($d$) scheme on $G_N$
has the same weak limit in~\eqref{fluid:standard} as that on a clique,
as the number of vertices $N$ becomes large.
Further, a propagation of chaos property was shown to hold for this system,
in the sense that the  queue lengths at any finite collection
of tagged servers are asymptotically independent, and the queue length
process for each server converges in distribution (in the path space)
to a certain McKean-Vlasov process~\cite[Theorem~2.2]{BMW17}.
Furthermore, when the graph sequence is random, with the $N$-th graph
given as an Erd\H{o}s-R\'enyi random graph (ERRG) on $N$~vertices
with average degree $d(N)$, note that there are two types of randomness
that drive the dynamics of the process: one being the randomness
of the underlying graph and the other being the randomness
of the arrival/departure processes given the graph.
This setup comes under the framework of random processes in random environment.
Here one is typically interested in two types of convergence results: 
(1)~\emph{Annealed convergence}, where one looks at the dynamics
of the sequence of occupancy processes averaged over the randomness
of the underlying graph, and
(2) \emph{Quenched convergence}, where one samples a sequence
of random graphs with increasing $N$ and given that sequence,
considers the dynamics of the sequence of occupancy process.
In~\cite{BMW17} annealed convergence is is established
under the condition $d(N) \to \infty$, and the quenched convergence
is shown under a stronger condition $d(N) / \log N \to \infty$. \\

\noindent
{\bf Asymptotic optimality on graphs.}
We return to the case when each incoming task is assigned
to whichever server has the smallest number of tasks
among the one where it arrives and its neighbors in~$G_N$.
The results presented in the remainder of the section are based
on~\cite{MBL17} where also full proofs are provided, unless indicated otherwise.
As mentioned earlier, the queue length process in a clique
will be better balanced and smaller (in a majorization sense)
than in an arbitrary graph~$G_N$.
Accordingly, a graph $G_N$ is said to be $N$-optimal or $\sqrt{N}$-optimal
when the queue length process on $G_N$ is equivalent
to that on a clique on an $N$-scale or $\sqrt{N}$-scale, respectively.
Roughly speaking, a graph is $N$-optimal if the \emph{fraction} of nodes
with $i$~tasks, for $i = 0, 1, \ldots$, behaves as in a clique as $N \to \infty$.
The fluid-limit results for the JSQ policy discussed
in Section~\ref{ssec:jsqfluid} imply that the latter fraction is
zero in the limit for all $i \geq 2$ in a clique in stationarity, i.e.,
the fraction of servers with two or more tasks vanishes in any graph
that is $N$-optimal, and consequently the mean waiting time vanishes
as well as $N \to \infty$.
Furthermore, the diffusion-limit results of~\cite{EG15} for the JSQ policy
discussed in Section~\ref{ssec:diffjsq} imply that the number of nodes
with zero tasks and that with two tasks both scale as $\sqrt{N}$ as $N \to \infty$.
Again loosely speaking, a graph is $\sqrt{N}$-optimal
if in the heavy-traffic regime the number of nodes with zero tasks
and that with two tasks when scaled by $\sqrt{N}$ both evolve
as in a clique as $N \to \infty$.
Formal definitions of asymptotic optimality on an $N$-scale or $\sqrt{N}$-scale
will be introduced in Definition~\ref{def:opt} below.

As one of the main results, we will demonstrate that, remarkably,
asymptotic optimality can be achieved in quite sparse ERRGs.
We prove that a sequence of ERRGs indexed by the number of vertices~$N$
with $d(N) \to \infty$ as $N \to \infty$, is $N$-optimal.
We further establish that the latter growth condition for the average degree 
is in fact necessary in the sense that any graph sequence that contains
$\Theta(N)$ bounded-degree vertices cannot be $N$-optimal.
This implies that a sequence of ERRGs with finite average degree cannot
be $N$-optimal.
The growth rate condition is more stringent for optimality
on $\sqrt{N}$-scale in the heavy-traffic regime.
Specifically, we prove that a sequence of ERRGs indexed by the number
of vertices $N$ with $d(N) / (\sqrt{N} \log(N)) \to \infty$
as $N \to \infty$, is $\sqrt{N}$-optimal.

The above results demonstrate that the asymptotic optimality of cliques
on an $N$-scale and $\sqrt{N}$-scale can be achieved in far sparser graphs,
where the number of connections is reduced by nearly a factor~$N$
and $\sqrt{N}/\log(N)$, respectively,
provided the topologies are suitably random in the ERRG sense.
This translates into equally significant reductions
in communication overhead and storage capacity,
since both are roughly proportional to the number of connections. \\

\noindent
{\bf Arbitrary graph topologies.}
The key challenge in the analysis of load balancing on arbitrary graph
topologies is that one needs to keep track of the evolution of number of tasks
at each vertex along with their corresponding neighborhood relationship.
This creates a major problem in constructing a tractable Markovian
state descriptor, and renders a direct analysis of such processes
highly intractable, as already alluded to in~\cite{Mitzenmacher96}.
Consequently, even asymptotic results for load balancing processes
on an arbitrary graph have remained scarce so far.
We take a radically different approach and aim to compare the load
balancing process on an arbitrary graph with that on a clique.
Specifically, rather than analyze the behavior for a given class
of graphs or degree value, we explore for what types of topologies
and degree properties the performance is asymptotically similar
to that in a clique. \\

\noindent
{\bf Stochastic coupling for graphs.}
Our proof arguments build on the stochastic coupling constructions
developed in Section~\ref{univ} for JSQ($d$) policies.
Specifically, we view the load balancing process on an arbitrary graph
as a `sloppy' version of that on a clique,
and thus construct several other intermediate sloppy versions.
By constructing novel couplings, we develop a method of comparing
the load balancing process on an arbitrary graph and that on a clique. 
In particular, we bound the difference between the fraction of vertices
with $i$ or more tasks in the two systems for $i = 1, 2, \dots$,
to obtain asymptotic optimality results.
From a high-level viewpoint, conceptually related graph conditions
for asymptotic optimality were examined using quite different techniques
in \cite{TX11,TX17} in a dynamic scheduling framework
(as opposed to the load balancing context). \\

\noindent
{\bf Notation.}
For $k = 1, \ldots, N$, denote by $X_k(G_N, t)$ the queue length
at the $k$-th server at time~$t$ (including the task possibly in service),
and by $X_{(k)}(G_N, t)$ the queue length at the $k$-th ordered server
at time~$t$ when the servers are arranged in  non-decreasing order
of their queue lengths
(ties can be broken in some way that will be evident from the context).
Let $Q_i(G_N, t)$ denote the number of servers with queue length
at least~$i$ at time~$t$ and $q_i(G_N, t) = Q_i(G_N, t)/N$, $i = 1, 2, \ldots$.
It is important to note that $(q_i(G_N, t))_{i \geq 1}$
is itself \emph{not} a Markov process.
Given the graph $G_N$, the queue-length process $(X_{k}(G_N, t))_{k = 1}^{N}$
is Markovian under the model assumptions, and $(q_i(G_N, t)_{i\geq 1})$
is a function of $(X_k(G_N, t))_{k=1}^N$.
Also, in the Halfin-Whitt heavy-traffic regime~\eqref{eq:HW},
define the centered and scaled processes 
\begin{equation}
\label{eq:HWOcc}
\bar{Q}_1(G_N,t) = - \frac{N - Q_1(G_N,t)}{\sqrt{N}}, \qquad
\bar{Q}_i(G_N,t) = \frac{ Q_i(G_N,t)}{\sqrt{N}},
\end{equation}
analogous to~\eqref{eq:diffscale}. \\

\noindent
{\bf Asymptotic optimality.}
In general, the optimality of the clique topology is strict,
but it turns out that near-optimality can be achieved asymptotically
in a broad class of other graph topologies.
Therefore, we now introduce two notions of \emph{asymptotic optimality},
which will be useful to characterize the performance in large-scale systems. 

\begin{definition}[{Asymptotic optimality}]
\label{def:opt}
A graph sequence $\GG = \{G_N\}_{N \geq 1}$ is called `asymptotically
optimal on $N$-scale' or `$N$-optimal', if for any $\lambda < 1$,
the process $(q_1(G_N, \cdot), q_2(G_N, \cdot), \ldots)$ converges weakly,
on any finite time interval, to a process
$(q_1(\cdot), q_2(\cdot),\ldots)$ satisfying~\eqref{eq:fluid}.

Moreover, a graph sequence $\GG = \{G_N\}_{N \geq 1}$ is called
`asymptotically optimal on $\sqrt{N}$-scale' or `$\sqrt{N}$-optimal',
if in the Halfin-Whitt heavy-traffic regime~\eqref{eq:HW},
on any finite time interval, the process
$(\bQ_1(G_N, \cdot), \bQ_2(G_N, \cdot), \ldots)$ as in~\eqref{eq:HWOcc}
converges weakly to the process $(\bQ_1(\cdot), \bQ_2(\cdot), \ldots)$
given by~\eqref{eq:diffusionjsq}.
\end{definition}

Intuitively speaking, if a graph sequence is $N$-optimal
or $\sqrt{N}$-optimal, then in some sense, the associated occupancy
processes are indistinguishable from those of the sequence of cliques
on $N$-scale or $\sqrt{N}$-scale.
In other words, on any finite time interval their occupancy processes
can differ from those in cliques by at most $\oo(N)$ or $\oo(\sqrt{N})$,
respectively. 
We will interchangeably use the terms \emph{fluid scale} and \emph{diffusion scale}
to refer to $N$-scale and $\sqrt{N}$-scale, respectively.
In particular, exploiting interchange of the stationary ($t \to \infty$)
and many-server ($N \to \infty$) limits, we obtain that for any
$N$-optimal graph sequence $\{G_N\}_{N \geq 1}$, as $N \to \infty$
\[
q_1(G_N, \infty) \to \lambda \quad \mbox{ and} \quad
q_i(G_N, \infty) \to 0 \quad \mbox{ for all } i = 2, \dots, B,
\]
implying that the stationary fraction of servers with queue length two
or larger and the mean waiting time vanish.
It is worthwhile to point out that the above interchange of limits
requires the ergodicity of the queue length process for each fixed~$N$,
a certain tightness of the sequence
$\{(q_1(G_N, \infty), q_2(G_N, \infty), \ldots)\}_{N\geq 1}$,
and the global stability of the fluid limits.

\subsection{Asymptotic optimality criteria for deterministic graph sequences}

We now proceed to develop a criterion for asymptotic optimality
of an arbitrary deterministic graph sequence on different scales.
Next this criterion will be leveraged to establish optimality
of a sequence of random graphs.
We start by introducing some useful notation, and two measures
of \emph{well-connectedness}.
Let $G = (V, E)$ be any graph.
For a subset $U \subseteq V$, define $\com(U) := |V\setminus N[U]|$
to be the cardinality of the set of all vertices that are disjoint
from $U$ and its immediate neighbors, where
$N[U] := U\cup \{v \in V:\ \exists\ u \in U \mbox{ with } (u, v) \in E\}$.
For any fixed $\varepsilon > 0$ define
\begin{equation}
\label{def:dis}
\dis_1(G,\varepsilon) := \sup_{U\subseteq V, |U|\geq \varepsilon |V|}\com(U),
\qquad
\dis_2(G,\varepsilon) := \sup_{U\subseteq V, |U|\geq \varepsilon \sqrt{|V|}}\com(U).
\end{equation}

The next theorem provides sufficient conditions for asymptotic optimality
on $N$-scale and $\sqrt{N}$-scale in terms
of the above two well-connectedness measures.

\begin{theorem}
\label{th:det-seq}
For any graph sequence $\GG = \{G_N\}_{N \geq 1}$,
\begin{enumerate}[{\normalfont (i)}]
\item $\GG$ is $N$-optimal if for any $\varepsilon > 0$, 
$\dis_1(G_N, \varepsilon) / N \to 0$ as $N \to \infty$.
\item $\GG$ is $\sqrt{N}$-optimal if for any $\varepsilon > 0$, 
$\dis_2(G_N, \varepsilon) / \sqrt{N} \to 0$ as $N \to \infty$.
\end{enumerate}
\end{theorem}

From a high-level perspective, the conditions in Theorem~\ref{th:det-seq}
(i) and (ii) require that neighborhoods of any $\Theta(N)$
and $\Theta(\sqrt{N})$ vertices contain at least $N - \oo(N)$
and $N - \oo(\sqrt{N})$ vertices, respectively.
As we will see below in Theorems~\ref{th:inhom} and~\ref{th:reg},
the conditions in Theorem~\ref{th:det-seq} impose suitable levels
of connectivity in the graph topology
in order for it to be asymptotically optimal on fluid and diffusion scales,
while significantly reducing the total number of connections.
The next corollary is an immediate consequence of Theorem~\ref{th:det-seq}.

\begin{corollary}
Let $\GG = \{G_N\}_{N \geq 1}$ be any graph sequence.
Then {\rm(i)} if  the minimum degree in $G_N$ equals $N - \oo(N)$,
then $\GG$ is $N$-optimal,
and {\rm(ii)} if  the minimum degree in $G_N$ equals $N - \oo(\sqrt{N})$,
then $\GG$ is $\sqrt{N}$-optimal.
\end{corollary}

The rest of the subsection is devoted to a discussion
of the main proof arguments for Theorem~\ref{th:det-seq},
focusing on the proof of $N$-optimality. 
The proof of $\sqrt{N}$-optimality follows along similar lines.
We establish in Proposition~\ref{prop:cjsq} that if a system is able
to assign each task to a server in the set, denoted by $\cS^N(n(N))$,
of the $n(N)+1$ nodes with shortest queues,
where $n(N)$ is $\oo(N)$, then it is $N$-optimal. 
Since the underlying graph is not a clique however
(otherwise there is nothing to prove), for any $n(N)$ not every arriving task
can be assigned to a server in $\cS^N(n(N))$.
Hence we further prove in Proposition~\ref{prop:stoch-ord-new}
a stochastic comparison property implying that if on any finite time interval
of length~$t$, the number of tasks $\Delta^N(t)$ that are not assigned
to a server in $\cS^N(n(N))$ is $o_P(N)$, then the system is $N$-optimal as well.
The $N$-optimality can then be concluded when $\Delta^N(t)$ is $o_P(N)$,
which we establish in Proposition~\ref{prop:dis-new}
under the condition that $\dis_1(G_N, \varepsilon) / N \to 0$ as $N \to \infty$
as stated in Theorem~\ref{th:det-seq}.

To further explain the idea described in the above proof outline,
it is useful to adopt a slightly different point of view towards
load balancing processes on graphs.
From a high-level viewpoint, a load balancing process can be thought of as follows:
there are $N$~servers, which are assigned incoming tasks by some scheme.
The assignment scheme can arise from some topological structure,
in which case we will call it \emph{topological load balancing},
or it can arise from some other property of the occupancy process,
in which case we will call it \emph{non-topological load balancing}.
As mentioned earlier, the JSQ policy or the clique is optimal
among the set of all non-anticipating schemes,
irrespective of being topological or non-topological.
Also, load balancing on graph topologies other than a clique
can be thought of as a `sloppy' version of that on a clique, when each server
only has access to partial information on the occupancy state.
Below we first introduce a different type of sloppiness in the task
assignment scheme, and show that under a limited amount of sloppiness
optimality is retained on a suitable scale.
Next we will construct a scheme which is a hybrid of topological
and non-topological schemes, whose behavior is simultaneously close to
both the load balancing process on a suitable graph and that on a clique.

\paragraph{A class of sloppy load balancing schemes.}
Fix some function $n: \N \to \N$, and recall the set $\cS^N(n(N))$
as before as well as the class $\CJSQ(n(N))$
from Section~\ref{ssec:pf-idea-1}, where each arriving task is assigned
to one of the servers in $\cS^N(n(N))$. 
It should be emphasized that for any scheme in $\CJSQ(n(N))$,
we are not imposing any restrictions on \emph{how} the incoming task
should be assigned to a server in $\cS^N(n(N))$.
The scheme only needs to ensure that the arriving task is assigned to some
server in $\cS^N(n(N))$ with respect to \emph{some} tie breaking mechanism.
Observe that using Corollary~\ref{cor:bound} and following the arguments
as in the proof of Theorems~\ref{fluidjsqd} and~\ref{diffusionjsqd},
we obtain the next proposition, which provides a sufficient criterion
for asymptotic optimality of any scheme in $\CJSQ(n(N))$.

\begin{proposition}
\label{prop:cjsq}
For $0 \leq n(N) < N$, let $\Pi \in \CJSQ(n(N))$ be any scheme.
{\rm(i)} If $n(N) / N \to 0$ as $N \to \infty$, then $\Pi$ is $N$-optimal,
and
{\rm(ii)} If $n(N) / \sqrt{N} \to 0$ as $N \to \infty$, then $\Pi$ is
$\sqrt{N}$-optimal.
\end{proposition}

\paragraph{A bridge between topological and non-topological load balancing.}
For any graph $G_N$ and $n \leq N$, we first construct a scheme called
$I(G_N, n)$, which is an intermediate blend between the topological
load balancing process on $G_N$ and some kind of non-topological
load balancing on $N$~servers.
The choice of $n = n(N)$ will be clear from the context.

To describe the scheme $I(G_N, n)$, first synchronize the arrival epochs
at server~$v$ in both systems, $v = 1, 2, \ldots, N$.
Further, synchronize the departure epochs at the $k$-th ordered server
with the $k$-th smallest number of tasks in the two systems,
$k = 1, 2, \ldots, N$.
When a task arrives at server~$v$ at time~$t$ say, it is assigned
in the graph $G_N$ to a server $v' \in N[v]$ according to its own
statistical law.
For the assignment under the scheme $I(G_N,n)$, first observe that if
\begin{equation}
\label{eq:criteria}
\min_{u \in N[v]} X_u(G_N,t) \leq \max_{u \in \cS(n)} X_u(G_N,t),
\end{equation}
then there exists \emph{some} tie-breaking mechanism for which
$v' \in N[v]$ belongs to $\cS(n)$ under $G_N$.
Pick such an ordering of the servers, and assume that $v'$ is the
$k$-th ordered server in that ordering, for some $k \leq n+1$.
Under $I(G_N,n)$ assign the arriving task to the $k$-th ordered server
(breaking ties arbitrarily in this case).
Otherwise, if \eqref{eq:criteria} does not hold, then the task is
assigned to one of the $n+1$ servers with minimum queue lengths
under $G_N$ uniformly at random.

Denote by $\Delta^N(I(G_N, n), T)$ the cumulative number of arriving
tasks up to time $T \geq 0$ for which Equation~\eqref{eq:criteria} is
violated under the above coupling.
The next proposition shows that the load balancing process under the
scheme $I(G_N,n)$ is close to that on the graph $G_N$ in terms of the
random variable $\Delta^N(I(G_N, n), T)$.

\begin{proposition}
\label{prop:stoch-ord-new}
The following inequality is preserved almost surely
\begin{equation}
\label{eq:stoch-ord-new}
\sum_{i=1}^{B} |Q_i(G_N, t) - Q_i(I(G_N, n), t)| \leq
2 \Delta^N(I(G_N, n), t), \qquad \forall\ t \geq 0,
\end{equation}
provided the two systems start from the same occupancy state at $t = 0$.
\end{proposition}

In order to conclude optimality on $N$-scale or $\sqrt{N}$-scale,
it remains to be shown that the term $\Delta^N(I(G_N, n), T)$ is
sufficiently small.
The next proposition provides suitable asymptotic bounds for
$\Delta^N(I(G_N, n), T)$ under the conditions on $\dis_1(G_N, \varepsilon)$
and $\dis_2(G_N, \varepsilon)$ stated in Theorem~\ref{th:det-seq}.
For $N$-optimality, the idea is that since for all $\varepsilon>0$,
$\dis_1(G_N, \varepsilon)$ is $o(N)$, one can show that there is a number
$n_{\varepsilon}(N) = o(N)$, such that  $\com(U) =o(N)$ uniformly
over all $U\subseteq V_N$ with $|U|\geq n_{\varepsilon}(N)$. 
Consequently, this can be used to show that on any finite time interval,
`most of the tasks' will be assigned to one of the $n_{\varepsilon}(N)$
servers with smallest queue lengths.
This enables us to couple the system with a scheme from the class $\CJSQ(n_{\varepsilon}(N))$.
The idea is similar when we consider $\sqrt{N}$-optimality.

\begin{proposition}
\label{prop:dis-new}
\begin{enumerate}[{\normalfont (i)}]
\item For any $\varepsilon>0$, there exists $\varepsilon'>0$
and $n_{\varepsilon'}(N)$ with $n_{\varepsilon'}(N)/N\to 0$ as $N\to\infty$,
such that if $\dis_1(G_N,\varepsilon')/N\to 0$ as $N\to\infty$,
then for all $T>0$, $$\Pro{\Delta^N(I(G_N,n_{\varepsilon'}),T)/N>\varepsilon}\to 0.$$ 
\item For any $\varepsilon>0$, there exists $\varepsilon'>0$
and $m_{\varepsilon'}(N)$ with $m_{\varepsilon'}(N)/\sqrt{N}\to 0$ as $N\to\infty$,
such that if $\dis_2(G_N,\varepsilon')/\sqrt{N}\to 0$ as $N\to\infty$, then for all $T>0$, $$\Pro{\Delta^N(I(G_N,m_{\varepsilon'}),T)/\sqrt{N}>\varepsilon}\to 0.$$
\end{enumerate}
\end{proposition}

The proof of Theorem~\ref{th:det-seq} then readily follows
by combining Propositions~\ref{prop:cjsq}-\ref{prop:dis-new}
and observing that the scheme $I(G_N,n)$ belongs to the class $\CJSQ(n)$
by construction.

From the conditions of Theorem~\ref{th:det-seq} it follows that
if for all $\varepsilon > 0$, $\dis_1(G_N, \varepsilon)$
and $\dis_2(G_N, \varepsilon)$ are $\oo(N)$ and $\oo(\sqrt{N})$,
respectively, then the total number of edges in $G_N$ must be
$\omega(N)$ and $\omega(N\sqrt{N})$, respectively.
Theorem~\ref{th:bdd-deg} below states that the \emph{super-linear}
growth rate of the total number of edges is not only sufficient,
but also necessary in the sense that any graph with $\OO(N)$ edges
is asymptotically sub-optimal on $N$-scale.

\begin{theorem}
\label{th:bdd-deg}
Let $\GG = \{G_N\}_{N \geq 1}$ be any graph sequence,
such that there exists a fixed integer $M < \infty$ with 
\begin{equation}
\label{eq:bdd-deg}
\limsup_{N \to \infty} \dfrac{\#\big\{v \in V_N: d_v \leq M\big\}}{N} > 0,
\end{equation}
where $d_v$ is the degree of the vertex~$v$.
Then $\GG$ is sub-optimal on $N$-scale.
\end{theorem}

To prove Theorem~\ref{th:bdd-deg}, we show that starting
from an all-empty state, in finite time, a positive fraction
of servers in $G_N$ will have at least two tasks. 
This establishes that the occupancy processes when scaled by~$N$
cannot agree with those in the sequence of cliques,
and hence $\{G_N\}_{N \geq 1}$ cannot be $N$-optimal.
The idea of the proof can be explained as follows: If a system contains
$\Theta(N)$ bounded-degree vertices, then starting from an all-empty state,
in any finite time interval there will be $\Theta(N)$ servers~$u$ say,
for which all the servers in $N[u]$ have at least one task.
For all such servers an arrival at~$u$ must produce a server
with queue length two.
It follows that the instantaneous rate at which servers of queue
length two are formed is bounded away from zero, and hence $\Theta(N)$
servers of queue length two are produced in finite time. 

\subsection{Asymptotic optimality of random graph sequences}

Next we investigate how the load balancing process behaves
on random graph topologies. 
Specifically, we aim to understand what types of graphs are asymptotically
optimal in the presence of randomness (i.e., in an average-case sense).
Theorem~\ref{th:inhom} below establishes sufficient conditions
for asymptotic optimality of a sequence of inhomogeneous random graphs.
Recall that a graph $G' = (V', E')$ is called a supergraph
of $G = (V, E)$ if $V = V'$ and $E \subseteq E'$.
\footnote{Reviewer: It troubles me that Thm 6.2 and its following comment have
a ‘perturbation of a complete graph’ flavor made explicit in Cor 6.3, whereas Thms 6.8 and 6.9
have more of a ‘sampling proportional to the size of the network’ aspect.
This is due to the randomness, but the mechanism is unclear, and it is not intuitively clear for me for instance
that conditions (i) and (ii) in Theorem 6.8 should be enough to imply conditions (i) and (ii) in Theorem 6.2.
I must be missing something.
Could you briefly comment on this?}

\begin{theorem}
\label{th:inhom}
Let $\GG= \{G_N\}_{N \geq 1}$ be a graph sequence such that for each~$N$,
$G_N = (V_N, E_N)$ is a supergraph of the inhomogeneous random graph $G_N'$
where any two vertices $u, v \in V_N$ share an edge with probability $p_{uv}^N$,
independently of each other.
\begin{enumerate}[{\normalfont (i)}]
\item If $\inf\ \{p^N_{uv}: u, v\in V_N\}$ is $\omega(1/N)$,
then $\GG$ is $N$-optimal.
\item If $\inf\ \{p^N_{uv}: u, v\in V_N\}$ is $\omega(\log(N)/\sqrt{N})$,
then $\GG$ is $\sqrt{N}$-optimal.
\end{enumerate}
\end{theorem}

The proof of Theorem~\ref{th:inhom} relies on Theorem~\ref{th:det-seq}.
Specifically, if $G_N$ satisfies conditions~(i) and~(ii)
in Theorem~\ref{th:inhom}, then the corresponding conditions~(i) and~(ii)
in Theorem~\ref{th:det-seq} hold.

As an immediate corollary of Theorem~\ref{th:inhom} we obtain
an optimality result for the sequence of ERRGs.
Let $\ERRG(N, p(N))$ denote a graph on $N$ vertices, such that any pair
of vertices share an edge with probability $p(N)$.

\begin{corollary}
\label{cor:errg}
Let $\GG = \{G_N\}_{N \geq 1}$ be a graph sequence such that for each~$N$,
$G_N$ is a super-graph of $\ERRG(N, p(N))$, and $d(N) = (N-1) p(N)$.
Then
{\normalfont (i)}
If $d(N) \to \infty$ as $N \to \infty$, then $\GG$ is $N$-optimal.
{\normalfont (ii)}
If $d(N) / (\sqrt{N} \log(N)) \to \infty$ as $N \to \infty$,
then $\GG$ is $\sqrt{N}$-optimal.
\end{corollary}

Theorem~\ref{th:det-seq} can be further leveraged to establish the
optimality of the following sequence of random graphs.
For any $N \geq 1$ and $d(N) \leq N-1$ such that $N d(N)$ is even,
construct the \emph{erased random regular} graph on $N$~vertices as follows:
Initially, attach $d(N)$ \emph{half-edges} to each vertex. 
Call all such half-edges \emph{unpaired}.
At each step, pick one half-edge arbitrarily, and pair it to another
half-edge uniformly at random among all unpaired half-edges
to form an edge, until all the half-edges have been paired.
Thus, note that there can be more than one edge between two vertices
(i.e., multi-edge) or a half-edge of a vertex can be paired
with another half-edge of the same vertex (self-loops).
Such a graph is known as a regular multi-graph.
In fact, it is known~\cite[Proposition 7.7]{remco-book-1}
that the above pairing procedure results in a random graph
that has a uniform distribution over all regular multi-graph with degree $d(N)$.
Now the erased random regular graph is formed by erasing all the
self-loops and collapsing the multiple edges to a single edge,
which thus produces a simple graph.

\begin{theorem}
\label{th:reg}
Let $\GG = \{G_N\}_{N \geq 1}$ be a sequence of erased random regular
graphs with degree $d(N)$.
Then
{\normalfont (i)}
If $d(N) \to \infty$ as $N \to \infty$, then $\GG$ is $N$-optimal.
{\normalfont (ii)}
If $d(N) / (\sqrt{N} \log(N)) \to \infty$ as $N \to \infty$,
then $\GG$ is $\sqrt{N}$-optimal.
\end{theorem}

Note that due to Theorem~\ref{th:bdd-deg}, we can conclude that the
growth rate condition for $N$-optimality in Corollary~\ref{cor:errg}~(i)
and Theorem~\ref{th:reg}~(i) is not only sufficient, but necessary as well.
Thus informally speaking, $N$-optimality is achieved under the minimum
condition required as long as the underlying topology is suitably random.

\section{Token-based load balancing}
\label{token}

While a zero waiting time can be achieved in the limit by sampling only
$d(N) = \oo(N)$ servers as Sections~\ref{univ} and~\ref{networks} showed,
even in network scenarios, the amount of communication overhead
in terms of $d(N)$ must still grow with~$N$.
As mentioned earlier, this can be avoided by introducing memory at the
dispatcher, in particular maintaining a record of only vacant servers,
and assigning tasks to idle servers, if there are any,
or to a uniformly at random selected server otherwise.
This so-called Join-the-Idle-Queue (JIQ) scheme \cite{BB08,LXKGLG11}
can be implemented through a simple token-based mechanism generating
at most one message per task.
Remarkably enough, even with such low communication overhead,
the mean waiting time and the probability of a non-zero waiting time
vanish under the JIQ scheme in both the fluid and diffusion regimes,
as we will discuss in the next two subsections.
It is worth emphasizing though that the JIQ scheme is \emph{not}
optimal in the non-degenerate slow-down regime, which was introduced
in Section~\ref{asym} and will be further discussed
in Section~\ref{nondegenerate}.

\subsection{Fluid-level optimality of JIQ scheme}
\label{ssec:fluidjiq}

We first consider the fluid limit of the JIQ policy.
It is not hard to show that the number of busy servers under the JIQ scheme
is stochastically larger (in the path space) than that for the JSQ($1$) policy
(tasks assigned uniformly at random).
Consequently, the JIQ scheme is stable whenever $\lambda<1$.
Recall that $q_i^N(\infty)$ denotes a random variable denoting the process
$q_i^N(\cdot)$ in steady state.
Under significantly more general conditions (in the presence
of finitely many heterogeneous server pools and for general service
time distributions with decreasing hazard rate) it was proved
in~\cite{Stolyar15} that under the JIQ scheme
\begin{equation}
\label{eq:fpjiq}
q_1^N(\infty) \to \lambda, \qquad q_i^N(\infty) \to 0 \quad
\mbox{ for all } i \geq 2, \qquad \mbox{ as } \quad N \to \infty.
\end{equation}
The above equation in conjunction with the PASTA property yields that
the steady-state probability of a non-zero wait vanishes as $N \to \infty$,
thus exhibiting asymptotic optimality of the JIQ policy on fluid scale.

\paragraph{High-level outline of proof idea.}
Loosely speaking, the proof of~\eqref{eq:fpjiq} consists of three
principal components:
\begin{enumerate}[{\normalfont (i)}]
\item Starting from an all-empty state, the asymptotic rate of increase
of~$q_1$ is given by the arrival rate~$\lambda$. 
Also, the rate of decrease is~$q_1$. 
Thus, on a small time interval $\dif t$, the rate of change of~$q_1$
is given by 
\begin{equation}
\label{eq:fluidjiq}
\frac{\dif q_1(t)}{\dif t} = \lambda - q_1(t).
\end{equation}
Under the above dynamics, the system occupancy states converge
to the unique fixed point of the above ODE, given by the point
$(\lambda, 0, 0, \ldots)$.
\item The occupancy process is monotone, in the sense that
(a) starting from an all-empty state, the occupancy process is
componentwise stochastically nondecreasing in time
(in the sense of stochastic dominance), and 
(b) the occupancy process at any fixed time~$t$, starting
from an arbitrary state, componentwise stochastically dominates
the occupancy process at time~$t$, starting from an all-empty state.
\item Under the JIQ scheme, the system is stable,
and hence the occupancy process is ergodic.
Since $q_1(t)$ is the instantaneous rate of departure from the system,
ergodicity implies that in steady state there can be {\em at most}
$\lambda$ fraction of busy servers (containing at least one task).
In fact, it further establishes that the steady-state fraction
of servers with more than one tasks vanishes asymptotically.
\end{enumerate}

Points~(i) and~(ii) above imply that starting from any state the
system must have at least $\lambda$ fraction of busy servers,
and finally this along with Point (iii) establishes that the
steady-state occupancy process must converge to $(\lambda, 0, 0, \ldots)$. 
 
\subsection{Diffusion-level optimality of JIQ scheme}

We now turn to the diffusion limit of the JIQ scheme established
in~\cite{MBLW16-1}.
Recall the centered and scaled occupancy process as in~\eqref{eq:diffscale}, 
and the Halfin-Whitt heavy-traffic regime in~\eqref{eq:HW}.

\begin{theorem}[Diffusion limit for JIQ]
\label{diffusionjiq}
Assume that $\lambda(N)$ satisfies~\eqref{eq:HW}.
Under suitable initial conditions the weak limit of the sequence
of centered and diffusion-scaled occupancy process in~\eqref{eq:diffscale}
coincides with that of the ordinary JSQ policy, and in particular,
is given by the system of SDEs in~\eqref{eq:diffusionjsq}.
\end{theorem}

The above theorem implies that for suitable states,
on any finite time interval, the occupancy process of a system
under the JIQ policy is indistinguishable from that under the JSQ policy.

\paragraph{High-level outline of proof idea.}
The proof of Theorem~\ref{diffusionjiq} relies on a novel coupling
construction introduced in~\cite{MBLW16-1} as described below in detail.
The idea is to compare the occupancy processes of two systems
following JIQ and JSQ policies, respectively. 
Comparing the JIQ and JSQ policies is facilitated when viewed as follows:
(i) If there is an idle server in the system, both JIQ and JSQ perform
similarly,
(ii)~Also, when there is no idle server and only $\OO(\sqrt{N})$
servers with queue length two or more, JSQ assigns the arriving task
to a server with queue length one. 
In that case, since JIQ assigns at random, the probability that the
task will land on a server with queue length two or more and thus JIQ acts
differently than JSQ is $\OO(1/\sqrt{N})$.
Since on any finite time interval the number of times an arrival finds
all servers busy is at most $\OO(\sqrt{N})$, all the arrivals except
$\OO(1)$ of them are assigned in exactly the same manner in both JIQ
and JSQ, which then leads to the same scaling limit for both policies. \\

The diffusion limit result in Theorem~\ref{diffusionjiq} is in fact true
for an even broader class of load balancing schemes.
As in Section~\ref{ssec:pf-idea-1}, let $B$ denote the buffer capacity
(possibly infinite) of each server, and in case $B < \infty$, if a task is assigned
to a server with $B$~outstanding tasks, it is instantly discarded.
For an LBA $\Pi$, we will denote the total number of tasks lost
up to time $t$ by $L^\Pi(t)$.
Define the class of schemes
$$\Pi^{(N)} := \{\Pi(d_0, d_1, \ldots, d_{B-1}): d_0 = N, 1 \leq d_i \leq N,
1 \leq i \leq B-1, B \geq 2\},$$
where  in the scheme $\Pi(d_0, d_1, \ldots, d_{B-1})$ with buffer capacity~$B$,
the dispatcher assigns an incoming task to the server with the minimum queue length
among $d_k$ (possibly function of~$N$) servers selected uniformly at random
when the minimum queue length across the system is~$k$, $k = 0, 1, \ldots, B - 1$.
The system analyzed in~\cite{EG15} (JSQ with $B = 2$) can be written
as $\Pi(N,N)$, JIQ can be expressed as $\Pi(N, 1, 1, \ldots)$,
and JIQ with a buffer capacity $B = 2$ is $\Pi(N, 1)$.

The crux of the argument in proving diffusion-level optimality
for any scheme in $\Pi^{(N)}$ goes as follows: 
First the occupancy process under the scheme $\Pi(N, d_1, \ldots, d_{B-1})$
is sandwiched between those under $\Pi(N, 1)$ and $\Pi(N, d_1)$.
More specifically, the $\ell_1$-distance between the occupancy processes
under $\Pi(N, d_1, \ldots, d_{B-1})$ and $\Pi(N, 1)$ is bounded
by the number of items lost due to full buffers.
Next, this loss is bounded using the number of servers with queue length~$2$
in $\Pi(N,N)$.
This allows the use of the results in~\cite{EG15}, and yields that
on any finite time interval with high probability an $\OO(1)$ number of items
are lost due to full buffers, which is negligible on $\sqrt{N}$ scale.
Specifically, this shows that for suitable initial states,
the schemes $\Pi(N, 1)$ and $\Pi(N, d_1)$, along with any scheme in the class
$\Pi^{(N)}$ has the same diffusion limit in the Halfin-Whitt heavy-traffic regime.
We conclude this subsection by describing the coupling construction
stating the stochastic inequalities, and a brief proof sketch
for Theorem~\ref{diffusionjiq}.

\paragraph{The coupling construction.}
We now construct a coupling between two systems following any two schemes,
say $\Pi_1 = \Pi(l_0, l_1, \ldots, l_{B-1})$
and $\Pi_2 = \Pi(d_0, d_1, \ldots, d_{B'-1})$ in $\Pi^{(N)}$, respectively,
to establish the desired stochastic ordering results.
Note that $\Pi_1$ and $\Pi_2$ have (possibly different) buffer
capacities~$B$ and~$B'$, respectively.
With slight abuse of notation we will denote by $\Pi_i$ the system
following scheme $\Pi_i$, $i = 1, 2$.

For the arrival process we couple the two systems as follows. First we synchronize the  arrival epochs of the two systems. Now assume that in the systems $\Pi_1$ and $\Pi_2$, the minimum queue lengths are $k$ and $m$, respectively, for some $k\leq B-1$, $m\leq B'-1$. Therefore, when a task arrives, the dispatchers in $\Pi_1$ and $\Pi_2$ have to select $l_k$ and $d_m$ servers, respectively, and then have to send the task to the one having the minimum queue length among the respectively selected servers. Since the servers are being selected uniformly at random we can assume without loss of generality, as in the stack construction, that the servers are arranged in non-decreasing order of their queue lengths and are indexed in increasing order. 
Hence, observe that when a few server indices are selected, the server having the minimum of those indices will be the server with the minimum queue length among these. 
Thus, in this case the dispatchers in $\Pi_1$ and $\Pi_2$ select $l_k$ and $d_m$ random numbers (without replacement) from $\{1,2,\ldots,N\}$ and then send the incoming task to the servers having indices to be the minimum of those selected numbers. 
Now, note that selecting $l_k$ (or $d_m$) random servers is equivalent to selecting a random permutation of $\{1,2,\ldots,N\}$, say $(\sigma_1, \sigma_2,\ldots,\sigma_N)$, and selecting first $l_k$ (or $d_m$) indices.
To couple the assignment decisions of the two systems, at each arrival epoch a \emph{single} random permutation of $\{1,2,\ldots,N\}$ is drawn, denoted by $\boldsymbol{\Sigma}^{(N)}:=(\sigma_1, \sigma_2,\ldots,\sigma_N)$. Define $\sigma_{(i)}:= \min_{j\leq i}\sigma_j$. Then observe that system $\Pi_1$ sends the task to the server with the index $\sigma_{(l_k)}$ and system $\Pi_2$ sends the task to the server with the index $\sigma_{(d_m)}$. Since at each arrival epoch both systems use a common random permutation, they take decisions in a coupled manner.

For the potential departure process, couple the service completion times of the $k^{th}$ queue in both scenarios, $k= 1,2,\ldots,N$. More precisely, for the potential departure process assume that we have a single synchronized exp($N$) clock independent of arrival epochs for both systems. Now when this clock rings, a number $k$ is uniformly selected from $\{1,2,\ldots,N\}$ and a potential departure occurs from the $k^{th}$ queue in both systems. If at a potential departure epoch an empty queue is selected, then we do nothing. 
Since the service time requirements are i.i.d.~exponentially distributed, the memoryless property ensures that the two schemes, considered independently, still evolve according to their appropriate statistical laws under the above coupling.

\begin{proposition}
\label{prop: stoch_ord}
For two schemes $\Pi_1=\Pi(l_0,l_1,\ldots,l_{B-1})$
and $\Pi_2=\Pi(d_0, d_1,\ldots, d_{B'-1})$ with $B\leq B'$
assume $l_0=\ldots=l_{B-2}=d_0=\ldots=d_{B-2}=d$,
$l_{B-1}\leq d_{B-1}$ and either $d=N$ or $d\leq d_{B-1}$.
Then the following holds:
\begin{enumerate}[{\normalfont (i)}]
\item\label{component_ordering} $\{Q^{ \Pi_1}_i(t)\}_{t\geq 0}\leq_{st}\{Q^{ \Pi_2}_i(t)\}_{t\geq 0}$
for $i=1,2,\ldots,B$,
\item\label{upper bound} $\{\sum_{i=1}^B Q^{ \Pi_1}_i(t)+L^{ \Pi_1}(t)\}_{t\geq 0}\geq_{st} \{\sum_{i=1}^{B'} Q^{ \Pi_2}_i(t)+L^{ \Pi_2}(t)\}_{t\geq 0}$,
\item\label{delta_ineq} $\{\Delta(t)\}_{t\geq 0}\geq \{\sum_{i=B+1}^{B'}Q_i^{ \Pi_2}(t)\}_{t\geq 0}$
almost surely under the coupling defined above,
\end{enumerate}
for any fixed $N\in\mathbbm{N}$ where $\Delta(t):=L^{ \Pi_1}(t)-L^{ \Pi_2}(t)$,
provided that at time $t=0$ the above ordering holds.
\end{proposition}

\begin{proof}[Proof of Theorem~\ref{diffusionjiq}]
Let $\Pi = \Pi(N, d_1, \ldots, d_{B-1})$ be a load balancing scheme
in the class $\Pi^{(N)}$.
Denote by $\Pi_1$ the scheme $\Pi(N,d_1)$ with buffer size $B = 2$
and let $\Pi_2$ denote the JIQ policy $\Pi(N,1)$ with buffer size $B = 2$.

Observe that from Proposition~\ref{prop: stoch_ord} we have
under the coupling defined above,
\begin{equation}
\label{eq: bound}
\begin{split}
|Q_i^\Pi(t) - Q_i^{\Pi_2}(t)| & \leq |Q_i^{\Pi}(t) - Q_i^{\Pi_1}(t)| +
|Q_i^{\Pi_1}(t) - Q_i^{\Pi_2}(t)| \\ & \leq
|L^{\Pi_1}(t) - L^{\Pi}(t)| + |L^{\Pi_2}(t) - L^{\Pi_1}(t)| \leq
2 L^{\Pi_2}(t),
\end{split}
\end{equation}
for all $i \geq 1$ and $t \geq 0$ with the understanding that $Q_j(t) = 0$
for all $j > B$, for a scheme with buffer capacity~$B$.
The third inequality above is
due to Proposition~\ref{prop: stoch_ord}\eqref{delta_ineq},
which in particular says that
$\{L^{\Pi_2}(t)\}_{t \geq 0} \geq \{L^{\Pi_1}(t)\}_{t \geq 0} \geq
\{L^{\Pi}(t)\}_{t \geq 0}$ almost surely under the coupling.
Now we have the following lemma.

\begin{lemma}
\label{lem: tight}
For all $t\geq 0$, under the assumptions of Theorem~\ref{diffusionjiq},
$\{L^{\Pi_2}(t)\}_{N \geq 1}$ forms a tight sequence.
\end{lemma}

Since $L^{\Pi_2}(t)$ is non-decreasing in~$t$, the above lemma
in particular implies that
\begin{equation}
\label{eq: conv 0}
\sup_{t\in[0,T]}\frac{L^{\Pi_2}(t)}{\sqrt{N}} \pto 0.
\end{equation}
For any scheme $\Pi \in \Pi^{(N)}$, from~\eqref{eq: bound} we know that
$$\{Q_i^{\Pi_2}(t) - 2 L^{\Pi_2}(t)\}_{t \geq 0} \leq
\{Q_i^{\Pi}(t)\}_{t \geq 0} \leq \{Q_i^{\Pi_2}(t) +
2 L^{\Pi_2}(t)\}_{t \geq 0}.$$
Combining~\eqref{eq: bound} and~\eqref{eq: conv 0} shows that
if the weak limits under the $\sqrt{N}$ scaling exist,
they must be the same for all the schemes in the class $\Pi^{(N)}$. 
Also, as described in Section~\ref{sec:powerofd}, the weak limit
for $\Pi(N, N)$ exists and the common weak limit can be described
by the unique solution of the SDEs in~\eqref{eq:diffusionjsq}.
Hence, the proof of Theorem~\ref{diffusionjiq} is complete.
\end{proof}

\subsection{Multiple dispatchers}
\label{multiple}

\newcommand{\Nb}{\bar{x}_0}

\begin{figure}\centering
\includegraphics[width=.5\linewidth]{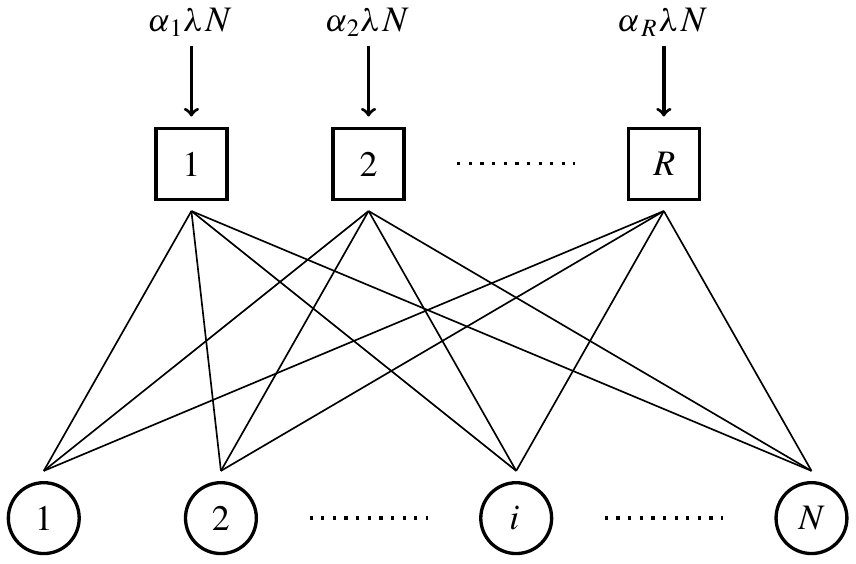}
\caption{Schematic view of the model with $R$ dispatchers and $N$ servers.}
\label{modelfigure}
\end{figure}

So far we have focused on a basic scenario with a single dispatcher,
but it is not uncommon for LBAs to operate across multiple dispatchers.
While the presence of multiple dispatchers does not affect
the queueing dynamics of JSQ($d$) policies, it does matter
for the JIQ scheme which uses memory at the dispatcher.
In order to examine the impact, we consider in this subsection
a scenario with $N$ parallel identical servers as before
and $R \geq 1$ dispatchers, as depicted in Figure~\ref{modelfigure}.
Tasks arrive at dispatcher~$r$ as a Poisson process of rate
$\alpha_r \lambda N$, with $\alpha_r > 0$, $r = 1, \dots, R$,
$\sum_{r = 1}^{R} \alpha_r = 1$, and $\lambda$ denoting the task
arrival rate per server.
For conciseness, we denote $\alpha = (\alpha_1, \dots, \alpha_R)$,
and without loss of generality we assume that the dispatchers are
indexed such that $\alpha_1 \geq \alpha_2 \geq \dots \geq \alpha_R$.

When a server becomes idle, it sends a token to one of the dispatchers
selected uniformly at random, advertising its availability.
When a task arrives at a dispatcher which has tokens available,
one of the tokens is selected, and the task is immediately forwarded
to the corresponding server.

We distinguish two scenarios when a task arrives at a dispatcher
which has no tokens available, referred to as the {\em blocking\/}
and {\em queueing\/} scenario respectively.
In the blocking scenario, the incoming task is blocked and instantly
discarded.
In the queueing scenario, the arriving task is forwarded to one
of the servers selected uniformly at random.
If the selected server happens to be idle, then the outstanding
token at one of the other dispatchers is revoked.

In the queueing scenario we assume $\lambda < 1$, which is not only
necessary but also sufficient for stability.
It is not difficult to show that the joint queue length process is
stochastically majorized by a scheme that assigns each task
to a server chosen uniformly at random.
In the latter case, the system decomposes into $N$~independent
M/M/1 queues, each of which has load $\lambda < 1$ and is stable.

Scenarios with multiple dispatchers have received limited attention
in the literature, and the scant papers that exist
\cite{LXKGLG11,Mitzenmacher16,Stolyar17} almost exclusively assume
that the loads at the various dispatchers are strictly equal, i.e.,
$\alpha_1 = \dots = \alpha_R = 1 / R$.
In these cases the fluid limit, for suitable initial states, is the
same as in Equation~\eqref{eq:fluidjiq} for a single dispatcher,
and in particular the fixed point is the same,
hence, the JIQ scheme continues to achieve asymptotically optimal
delay performance with minimal communication overhead.
The results in~\cite{Stolyar17} in fact show that the JIQ scheme
remains asymptotically optimal even when the servers are heterogeneous,
while it is readily seen that JSQ($d$) policies cannot even provide
maximum stability (i.e.~achieve stability whenever feasible at all)
in that case for any fixed value of~$d$.
As one of the few exceptions, \cite{BBL17} allows the loads at the
various dispatchers to be different.
It is not uncommon for such skewed load patterns to arise for example
when the various dispatchers receive tasks from external sources
making it difficult to perfectly balance the task arrival rates.

\paragraph{Results for blocking scenario.}
For the blocking scenario, denote by $B(R, N, \lambda, \alpha)$
the steady-state blocking probability of an arbitrary task.
It is established in ~\cite{BBL17} that,
\[
B(R, N, \lambda, \alpha) \to \max\{1 - R \alpha_R, 1 - 1 / \lambda\}
\mbox{ as } N \to \infty.
\]
This result shows that in the many-server limit the system performance
in terms of blocking is either determined by the relative load
of the least-loaded dispatcher, or by the aggregate load.
This may be informally explained as follows.
Let $\Nb$ be the expected fraction of busy servers in steady state,
so that each dispatcher receives tokens on average at a rate $\Nb N/R$.
We distinguish two cases, depending on whether a positive fraction
of the tokens reside at the least-loaded dispatcher~$R$ in the limit or not.
If that is the case, then the task arrival rate $\alpha_R \lambda N$
at dispatcher~$R$ must equal the rate $\Nb N / R$
at which it receives tokens, i.e., $\Nb / R = \alpha_R \lambda$.
Otherwise, the task arrival rate $\alpha_R \lambda N$
at dispatcher~$R$ must be no less the rate $\Nb N / R$
at which it receives tokens, i.e., $\Nb / R \leq \alpha_R \lambda$.
Since dispatcher~$R$ is the least-loaded, it then follows that
$\Nb / R \leq \alpha_r \lambda$ for all $r = 1, \dots, R$,
which means that the task arrival rate at all the dispatchers
is higher that the rate at which tokens are received.
Thus the fraction of tokens at each dispatcher is zero in the limit,
i.e., the fraction of idle servers is zero, implying $\Nb = 1$.
Combining the two cases, and observing that $\Nb \leq 1$,
we conclude $\Nb = \min\{R \alpha_R \lambda, 1\}$.
Because of Little's law, $\Nb$ is related to the blocking probability~$B$
as $\Nb = \lambda (1 - B)$.
This yields $1 - B = \min\{R \alpha_R \lambda, 1 / \lambda\}$,
or equivalently, $B = \max\{1 - R \alpha_R, 1 - 1 / \lambda\}$.

The above explanation also reveals that, somewhat counter-intuitively,
it is the least-loaded dispatcher that throttles tokens and leaves
idle servers stranded, thus acting as bottleneck.
Specifically, in the limit dispatcher~$R$ (or the set of least-loaded
dispatchers in case of ties) inevitably ends up with all the available
tokens, if any.
The accumulation of tokens hampers the visibility of idle servers
to the heavier-loaded dispatchers, and leaves idle servers stranded
while tasks queue up at other servers.

\paragraph{Results for queueing scenario.}
For the queueing scenario, denote by $W(R, N, \lambda, \alpha)$
a random variable with the steady-state waiting-time distribution
of an arbitrary task. 
It is shown in~\cite{BBL17} that, for a fixed $\lambda < 1$ and $N \to \infty$,
\[
\mathbb{E}[W(R, N, \lambda, \alpha)] \to
\frac{\lambda_2(R, \lambda, \alpha)}{1 - \lambda_2(R, \lambda, \alpha)},
\]
where 
\[
\lambda_2(R, \lambda, \alpha) =
1 - \frac{1 - \lambda \sum_{i=1}^{r^*} \alpha_i}{1 - \lambda r^* / R}
\]
with 
\[
r^* = \sup\big\{r \big| \alpha_r > \frac{1}{R}
\frac{1 - \lambda\sum_{i = 1}^{r}\alpha_i}{1 - \lambda r/R}\big\}
\]
and the convention that $r^* = 0$ if $\alpha_1 = \hdots = \alpha_R = 1/R$.
In particular,
\[
\lambda_2(2, \lambda, (1 - \alpha_2, \alpha_2)) =
\lambda \frac{1 - 2 \alpha_2}{2 - \lambda},
\]
so that
\[
\mathbb{E}[W(2, N, \lambda, (1 - \alpha_2, \alpha_2))] \rightarrow
\frac{\lambda (1 - 2 \alpha_2)}{2 - 2 \lambda (1 - \alpha_2)}.
\]
Here $\lambda_2$ can be interpreted as the rate at which tasks are
forwarded to randomly selected servers.
Furthermore, dispatchers $1, \hdots, r^*$ receive tokens at a lower
rate than the incoming tasks, and in particular $\lambda_2^* = 0$
if and only if $r^* = 0$.

When the arrival rates at all dispatchers are strictly equal, i.e.,
$\alpha_1 = \dots = \alpha_R = 1 / R$, the above results indicate that
the stationary blocking probability and the mean waiting time
asymptotically vanish as $N \to \infty$, which is in agreement with
the observations in~\cite{Stolyar17} mentioned above.
However, when the arrival rates at the various dispatchers are not
perfectly equal, so that $\alpha_R < 1 / R$, the blocking probability
and mean waiting time are strictly positive in the limit,
even for arbitrarily low overall load and an arbitrarily small degree
of skewness in the arrival rates.
Thus, the ordinary JIQ scheme fails to achieve asymptotically optimal
performance for heterogeneous dispatcher loads.

\paragraph{Enhancements.}
In order to counter the above-described performance degradation
for asymmetric dispatcher loads, \cite{BBL17} proposes two enhancements.

\newtheorem{myenhancement}{Enhancement}

\begin{myenhancement}[Non-uniform token allotment]\label{alg1}
When a server becomes idle, it sends a token to dispatcher~$r$
with probability~$\beta_r$.
\end{myenhancement}

\begin{myenhancement}[Token exchange mechanism]\label{alg2}
Any token is transferred to a uniformly randomly selected dispatcher at rate~$\nu$.
\end{myenhancement}

Note that the token exchange mechanism only creates a constant
communication overhead per task as long as the rate~$\nu$ does not depend
on the number of servers~$N$, and thus preserves the scalability
of the basic JIQ scheme.
The above enhancements can achieve asymptotically optimal performance
for suitable values of the $\beta_r$ parameters and the exchange rate~$\nu$.

\paragraph{Large number of dispatchers.}
In the above set-up we assumed the number of dispatchers to remain fixed
as the number of servers grows large, but a further natural scenario
would be for the number of dispatchers $R(N)$ to scale with the number
of servers as considered in~\cite{Mitzenmacher16}.
He analyzes the case $R(N) = r N$ for some constant~$r$, so that the
relative load of each dispatcher is $\lambda r$. 
The term `I-queue' is used for the queue of (idle) servers that
is known by one of the dispatchers.
A server is added to an I-queue when it becomes idle.
With fluid limits and fixed-point calculations, the analysis
in~\cite{Mitzenmacher16} determines the fraction of I-queues
with $i$~queued servers and the fraction of servers with $i$~tasks
in queue that are in the $j$-th position in one of the I-queues.
The fixed point can be computed numerically.

\paragraph{Anticipation.}
In~\cite{Mitzenmacher16} it is also proposed to have servers issue
their availability tokens to the dispatchers already before they are idle,
e.g. when they have just one task remaining.
This appears beneficial at very high load when there are (on average)
fewer idle servers than dispatchers, and tasks would frequently be assigned
to uniformly at random selected servers otherwise. 
Two variants are introduced.
First, an LCFS-scheme in which the server that is in the I-queue
the least amount of time is chosen for the incoming task.
Second, a server that became idle, may probe $d$~I-queues
after which it chooses the least loaded amongst the~$d$ selected servers.
Both variants lead to small performance improvements.

\subsection{Joint load balancing and auto-scaling}

Besides delay performance and implementation overhead, a further key attribute
in the context of large-scale cloud networks and data centers is energy consumption.
So-called auto-scaling algorithms have emerged as a popular mechanism
for adjusting service capacity in response to varying demand levels
so as to minimize energy consumption while meeting performance targets,
but have mostly been investigated in settings with a centralized queue,
and queue-driven auto-scaling techniques have been widely investigated in the literature
\cite{ALW10,GDHS13,LCBWGWMH12,LLWLA11a,LLWLA11b,LLWA12,LWAT13,PP16,UKIN10,WLT12}.
In systems with a centralized queue it is common to put servers
to `sleep' while the demand is low, since servers in sleep mode
consume much less energy than active servers.
Under Markovian assumptions, the behavior of these mechanisms can be described
in terms of various incarnations of M/M/$N$ queues with setup times.
There are several further recent papers which examine on-demand server
addition/removal in a somewhat different vein \cite{PS16,NS16}. 
Unfortunately, data centers and cloud networks with massive numbers
of servers are too complex to maintain any centralized queue,
as it involves a prohibitively high communication burden to obtain
instantaneous state information.

Motivated by these observations, the authors of~\cite{MDBL17}
propose a joint load balancing and auto-scaling strategy,
which retains the excellent delay performance and low implementation
overhead of the ordinary JIQ scheme, and at the same time minimizes
the energy consumption.
The strategy is referred to as TABS (Token-Based Auto-Balance Scaling)
and operates as follows:
\begin{itemize}
\item When a server becomes idle, it sends a `green' message to the dispatcher, waits for an $\exp(\mu)$ time (standby period), and turns itself off by sending a `red' message to the dispatcher (the corresponding green message is destroyed).
\item When a task arrives, the dispatcher selects a green message at random if there are any, and assigns the task to the corresponding server (the corresponding green message is replaced by a `yellow' message). 
Otherwise, the task is assigned to an arbitrary busy  server, and if at that arrival epoch there is a red message at the dispatcher, then it selects one at random, and the setup procedure of the corresponding server is initiated, replacing its red message by an `orange' message.
Setup procedure takes $\exp(\nu)$ time after which the server becomes active. 
\item Any server which activates due to the latter event, sends a green message to the dispatcher (the corresponding orange message is replaced), waits for an $\exp(\mu)$ time for a possible assignment of a task, and again turns itself off by sending a red message to the dispatcher.
\end{itemize}

\begin{figure}
\begin{center}
\includegraphics[scale=1]{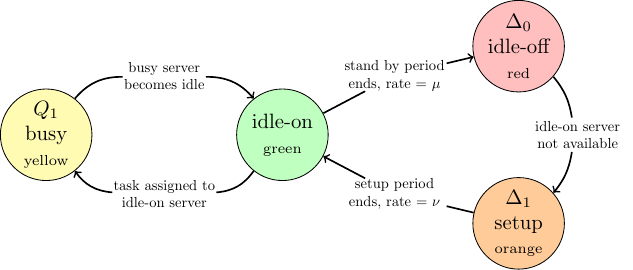}
\end{center}
\caption{Illustration of server on-off decision rules in the TABS scheme,
along with message colors and state variables.}
\label{fig:scheme}
\end{figure}

The TABS scheme gives rise to a distributed operation in which servers
are in one of four states (busy, idle-on, idle-off or standby),
and advertize their state to the dispatcher via exchange of tokens.
Figure~\ref{fig:scheme} illustrates this token-based exchange protocol.
Note that setup procedures are never aborted and continued even when
idle-on servers do become available.
Recently, dynamic scaling and load balancing with variable service
capacity and on-demand agents has been further examined in~\cite{GFPJ17}.

To describe systems under the TABS scheme,
we use $\mathbf{Q}^N(t) := (Q_1^N(t), Q_2^N(t), \dots, Q_B^N(t))$
to denote the system occupancy state at time~$t$ as before, where $B\geq 1$ is a finite buffer capacity.
Also, let $\Delta_0^N(t)$ and $\Delta_1^N(t)$ denote the number
of idle-off servers and servers  in setup mode at time~$t$, respectively. 
The fluid-scaled quantities are denoted by the respective small letters,
\emph{viz.} $q_i^{N}(t) := Q_i^{N}(t) / N$,
$\delta_0^N(t) = \Delta_0^N(t) / N$,
and $\delta_1^N(t) = \Delta_1^N(t) / N$.
For brevity in notation, we will write
$\mathbf{q}^N(t) = (q_1^N(t), \dots, q_B^N(t))$
and $\boldsymbol{\delta}^N(t) = (\delta_0^N(t), \delta_1^N(t))$.
The results presented in the remainder of the section are extracted
from~\cite{MDBL17}, unless indicated otherwise.

\paragraph{Fluid limit.}
Under suitable initial conditions, on any finite time interval,
with probability~$1$, any sequence in $N$ has a further subsequence
along which the sequence of processes $(\qq^N(\cdot), \dd^N(\cdot))$
converges to a deterministic limit $(\qq(\cdot), \dd(\cdot))$
that satisfies the following system of ODEs
\begin{equation}
\label{eq:tabsfluid}
\begin{split}
\frac{\dif^+ q_i(t)}{\dif t} &= \lambda(t) p_{i-1}(\qq(t),\dd(t),\lambda(t))
 - (q_i(t) - q_{i+1}(t)),\ i = 1, \ldots, B, \\
\frac{\dif^+\delta_0(t)}{\dif t} &= u(t) - \frac{\dif^+\xi(t)}{\dif t}, \qquad
\frac{\dif^+\delta_1(t)}{\dif t} = \frac{\dif^+\xi(t)}{\dif t} - \nu \delta_1(t),
\end{split}
\end{equation}
where by convention $q_{B+1}(\cdot) \equiv 0$, and
\begin{align*}
u(t) &= 1 - q_1(t) - \delta_0(t) - \delta_1(t), \quad
\frac{\dif^+\xi(t)}{\dif t} =
\lambda(t) (1 - p_0(\qq(t), \dd(t), \lambda(t))) \ind{\delta_0(t) > 0}.
\end{align*}
For any $(\qq, \dd)$ and $\lambda > 0$,
$(p_i(\qq, \dd, \lambda))_{i \geq 0}$ are given by 
\begin{align*}
p_0(\qq,\dd,\lambda) &= 
\begin{cases}
& 1 \qquad \text{ if } \qquad u = 1 - q_1 - \delta_0 - \delta_1 > 0, \\
& \min\{\lambda^{-1}(\delta_1 \nu + q_1 - q_2), 1\},\quad \text{otherwise,}
\end{cases} \\
\quad p_i(\qq, \dd, \lambda) &=
(1 - p_0(\qq, \dd, \lambda)) (q_i - q_{i+1})q_1^{-1},\ i = 1, \ldots, B.
\end{align*}

We now provide an intuitive explanation of the fluid limit stated above.
The term $u(t)$ corresponds to the asymptotic fraction of idle-on servers
in the system at time~$t$, and $\xi(t)$ represents the
asymptotic cumulative number of server setups (scaled by~$N$)
that have been initiated during $[0,t]$.
The coefficient $p_i(\qq, \dd, \lambda)$ can be interpreted as
the instantaneous fraction of incoming tasks that are assigned to some server
with queue length~$i$, when the fluid-scaled occupancy state is $(\qq, \dd)$
and the scaled instantaneous arrival rate is~$\lambda$.
Observe that as long as $u > 0$, there are idle-on servers,
and hence all the arriving tasks will join idle servers. 
This explains that if $u > 0$, $p_0(\qq, \dd, \lambda) = 1$
and $p_i(\qq, \dd, \lambda)=0$ for $i = 1, \ldots, B-1$.
If $u = 0$, then observe that servers become idle at rate $q_1 - q_2$,
and servers in setup mode turn on at rate $\delta_1 \nu$.
Thus the  idle-on servers are created at a total rate
$\delta_1 \nu + q_1 - q_2$.
If this rate is larger than the arrival rate~$\lambda$,
then almost all the arriving tasks can be assigned to idle servers.
Otherwise, only a fraction $(\delta_1 \nu + q_1 - q_2) / \lambda$
of arriving tasks join idle servers. 
The rest of the tasks are distributed uniformly among busy servers,
so a proportion $(q_i - q_{i+1}) q_1^{-1}$ are assigned to servers
having queue length~$i$.
For any $i = 1, \ldots, B$, $q_i$ increases when there is an arrival
to some server with queue length $i - 1$, which occurs at rate
$\lambda p_{i-1}(\qq, \dd, \lambda)$, and it decreases when there is
a departure from some server with queue length~$i$,
which occurs at rate $q_i - q_{i-1}$. 
Since each idle-on server turns off at rate~$\mu$, the fraction
of servers in the off mode increases at rate $\mu u$.
Observe that if $\delta_0 > 0$, for each task that cannot be assigned
to an idle server, a setup procedure is initiated  at one idle-off server. 
As noted above, $\xi(t)$ captures the (scaled) cumulative number
of setup procedures initiated up to time~$t$.
Therefore the fraction of idle-off servers and the fraction of servers
in setup mode decreases and increases by $\xi(t)$, respectively,
during $[0, t]$.
Finally, since each server in setup mode becomes idle-on at rate~$\nu$,
the fraction of servers in setup mode decreases at rate $\nu \delta_1$.

\paragraph{Fixed point and global stability.}
In case of a constant arrival rate $\lambda(t) \equiv \lambda < 1$,
any fluid sample path in~\eqref{eq:tabsfluid} has a unique fixed point:
\begin{equation}
\label{eq:fixed point}
\delta_0^* = 1 - \lambda, \qquad \delta_1^* = 0, \qquad
q_1^* = \lambda \quad \mbox{ and } \quad q_i^* = 0,
\end{equation}
for $i = 2, \ldots, B$.
Indeed, it can be verified that $p_0(\qq^*, \dd^*, \lambda) = 1$
and $u^* = 0$ for $(\qq^*, \dd^*)$ given by~\eqref{eq:fixed point}
so that the derivatives of~$q_i$, $i = 1, \dots, B$, $\delta_0$,
and~$\delta_1$ become zero, and that these cannot be zero at any other
fluid-scaled occupancy state.
Note that, at the fixed point, a fraction $\lambda$ of the servers
have exactly one task while the remaining fraction have zero tasks,
independently of the values of the parameters~$\mu$ and~$\nu$.

In order to establish the convergence of the sequence of steady states,
we need the global stability of the fluid limit, i.e.,
starting from any fluid-scaled occupancy state, any fluid sample path
described by~\eqref{eq:tabsfluid} converges to the unique fixed
point~\eqref{eq:fixed point} as $t \to \infty$.
More specifically, irrespective of the starting state,
\begin{equation}
\label{eq:globalstab}
(\qq(t), \dd(t)) \to (\qq^*, \dd^*), \quad \mbox{ as } \quad t \to \infty,
\end{equation}
where $(\qq^*, \dd^*)$ is as defined in~\eqref{eq:fixed point}.

\paragraph{Interchange of limits.}
The global stability can be leveraged to show that the steady-state
distribution of the $N$-th system, for large~$N$,
can be well approximated by the fixed point of the fluid limit
in~\eqref{eq:fixed point}. 
Specifically, it justifies the interchange of the many-server
($N \to \infty$) and stationary ($t \to \infty$) limits.
Since the buffer capacity~$B$ at each server is supposed to be finite,
for every~$N$, the Markov process $(\QQ^N(t), \Delta_0^N(t), \Delta_1^N(t))$
is irreducible, has a finite state space, and thus has a unique
steady-state distribution.
Let $\pi^N$ denote the steady-state distribution of the $N$-th system, i.e.,
$$\pi^{N}(\cdot) = \lim_{t \to \infty}
\mathbb{P}\ \big(\qq^{N}(t) = \cdot, \dd^N(t) = \cdot\big).$$ 
The fluid limit result and the global stability thus yield that $\pi^N$
converges weakly to~$\pi$ as $N \to \infty$, where $\pi$ is given
by the Dirac mass concentrated upon $(\qq^*,\dd^*)$ defined
in~\eqref{eq:fixed point}.

\begin{remark}\normalfont
Note that the above interchange of limits result was obtained
under the assumption that the queues have finite buffers,
and analysis of the infinite-buffer scenario was left open.
The key challenge in the latter case stems from the fact that the system
stability under the usual subcritical load assumption is not automatic.
In fact as explained in~\cite{MS19}, when the number of servers $N$ is fixed,
the stability may {\em not} hold even under a subcritical load assumption.
In~\cite{MS19} the stability issue of the TABS scheme has been addressed
and the convergence of the sequence of steady states was shown
for the infinite-buffer scenario.
In particular, it was established that for a fixed choice of parameters
$\lambda < 1$, $\mu > 0$, and $\nu > 0$, the system with $N$~servers
under the TABS scheme is stable for large enough~$N$.
There the authors introduce an induction-based approach that uses both
the conventional fluid limit (in the sense of a large starting state)
and the mean-field fluid limit (when $N \to \infty$) in an intricate fashion
to prove the large-$N$ stability of the system.
\end{remark}

\paragraph{Performance metrics.}
As mentioned earlier, two key performance metrics are the expected
waiting time of tasks $\expect{W^N}$ and energy consumption $\expect{P^N}$
for the $N$-th system in steady state.
In order to quantify the energy consumption, we assume that the energy usage
of a server is $P_{\full}$ when busy or in set-up mode,
$P_{\idle}$ when idle-on, and zero when turned off.
Evidently, for any value of~$N$, at least a fraction $\lambda$
of the servers must be busy in order for the system to be stable,
and hence $\lambda P_{\full}$ is the minimum mean energy usage
per server needed for stability.
We will define $\expect{Z^N} = \expect{P^N} - \lambda P_{\full}$
as the relative energy wastage accordingly.
The interchange of limits result can be leveraged to obtain that
asymptotically the expected waiting time and energy consumption
for the TABS scheme vanish in the limit, for any strictly positive
values of~$\mu$ and~$\nu$.
More specifically, for a constant arrival rate  $\lambda(t) \equiv \lambda < 1$,
for any $\mu > 0$, $\nu > 0$, as $N \to \infty$,
\begin{enumerate}[{\normalfont (a)}]
\item Zero mean waiting time: $\expect{W^N} \to 0$,
\item Zero energy wastage: $\expect{Z^N} \to 0$.
\end{enumerate}
The key implication is that the TABS scheme, while only involving
constant communication overhead per task, provides performance
in a distributed setting that is as good at the fluid level
as can possibly be achieved, even in a centralized queue,
or with unlimited information exchange.

\paragraph{Comparison to ordinary JIQ policy.}
Consider again a constant arrival rate $\lambda(t) \equiv \lambda$. 
It is worthwhile to observe that the component $\qq$ of the fluid limit
as in~\eqref{eq:tabsfluid} coincides with that for the ordinary JIQ policy
where servers always remain on, when the system following the TABS scheme
starts with all the servers being idle-on, and $\lambda + \mu < 1$. 
To see this, observe that the component $\qq$ depends on~$\dd$
only through $(p_{i-1}(\qq, \dd))_{i \geq 1}$. 
Now, $p_0 = 1$, $p_i = 0$, for all $i \geq 1$,
whenever $q_1 + \delta_0 + \delta_1 < 1$,
irrespective of the precise values of $(\qq, \dd)$. 
Moreover, starting from the above initial state, $\delta_1$ can increase
only when $q_1 + \delta_0 = 1$. 
Therefore, the fluid limit of $\qq$ in~\eqref{eq:tabsfluid}
and the ordinary JIQ scheme are identical if the system parameters
$(\lambda, \mu, \nu)$ are such that $q_1(t) + \delta_0(t) < 1$,
for all $t \geq 0$.
Let $y(t) = 1 - q_1(t) - \delta_0(t)$.
The solutions to the differential equations
\[
\frac{\dif q_1(t)}{\dif t} = \lambda - q_1(t), \quad
\frac{\dif y(t)}{\dif t} = q_1(t) - \lambda - \mu y(t),
\]
$y(0) = 1$, $q_1(0) = 0$ are given by 
\[
q_1(t) = \lambda (1 - \ee^{- t}), \quad
y(t) = \frac{\ee^{- (1 + \mu) t}}{\mu-1} \big(\ee^t(\lambda + \mu - 1) -
\lambda \ee^{\mu t}\big).
\]
Notice that if $\lambda + \mu < 1$, then $y(t) > 0$ for all $t \geq 0$
and thus, $q_1(t) + \delta_0(t) < 1$, for all $t \geq 0$.
The fluid-level optimality of the JIQ scheme was described
in Section~\ref{ssec:fluidjiq}. 
This observation thus establishes the optimality of the fluid-limit
trajectory under the TABS scheme for suitable parameter values
in terms of response time performance.
From the energy usage perspective, under the ordinary JIQ policy,
since the asymptotic steady-state fraction of busy servers ($q_1^*$)
and idle-on servers are given by~$\lambda$ and $1 - \lambda$, respectively,
the asymptotic steady-state (scaled) energy usage is given by 
\begin{align*}
\expect{P^{\mathrm{JIQ}}} = \lambda P_{\full} + (1 - \lambda) P_{\idle} =
\lambda P_{\full}(1 + (\lambda^{- 1} - 1) f),
\end{align*}
where $f = P_{\idle}/P_{\full}$ is the relative energy consumption
of an idle server.
As described earlier, the asymptotic steady-state (scaled) energy usage
under the TABS scheme is $\lambda P_{\full}$.
Thus the TABS scheme reduces the asymptotic steady-state energy usage
by $\lambda P_{\full}(\lambda^{- 1} - 1) f = (1 - \lambda) P_{\idle}$,
which amounts to a relative saving
of $(\lambda^{- 1} - 1) f / (1 + (\lambda^{- 1} - 1) f)$.
In summary, the TABS scheme performs as well as the ordinary JIQ policy
in terms of the waiting time and communication overhead while providing
a significant energy saving.

\section{Redundancy policies and alternative scaling}
\label{miscellaneous}

In this section we discuss somewhat related redundancy policies, 
alternative scaling regimes, and some additional performance metrics
of interest.

\subsection{Redundancy-$d$ policies}

So-called redundancy-$d$ policies involve a somewhat similar operation
as JSQ($d$) policies, and also share the primary objective of ensuring
low delays \cite{AGSS13,VGMSRS13}.
In a redundancy-$d$ policy, $d \geq 2$ candidate servers are selected
uniformly at random (with or without replacement) for each arriving task,
just like in a JSQ($d$) policy.
Rather than forwarding the task to the server with the shortest queue
however, replicas are dispatched to all sampled servers.
Note that the initial replication to $d$ servers selected uniformly
at random does not entail any communication burden, but the abortion
of redundant copies at a later stage does involve a significant amount
of information exchange and complexity.

Two common options can be distinguished for abortion of redundant clones.
In the first variant, as soon as the first replica starts service,
the other clones are abandoned.
In this case, a task gets executed by the server which had the smallest
workload at the time of arrival (and which may or may not have had the
shortest queue length) amongst the sampled servers.
This may be interpreted as a power-of-$d$ version of the Join-the-Smallest
Workload (JSW) policy discussed in Section~\ref{ssec:jsw}.
The optimality properties of the JSW policy mentioned in that subsection
suggest that redundancy-$d$ policies should outperform JSQ($d$) policies,
which appears to be supported by simulation experiments.

In the second option the other clones of the task are not aborted
until the first replica has completed service
(which may or may not have been the first replica to start service).
While a task is only handled by one of the servers in the former case, 
it may be processed by several servers in the latter case.
When the service times are exponentially distributed and independent for
the various clones, the aggregate amount of time spent by all the servers
until completion remains exponentially distributed with the same mean.
An exact analysis of the delay distribution in systems with $N = 2$
or $N = 3$ servers is provided in \cite{GZDHHS15,GZDHHS16},
and exact expressions for the mean delay with an arbitrary number
of servers are established in~\cite{GZHS16}.
The limiting delay distribution in the many-server regime (ii)
is derived in \cite{GZVHS16,GHSVZ17}
based on an asymptotic independence assumption among the servers.
In general, the mean aggregate amount of time devoted to a task
and the resulting delay may be larger or smaller for less or more
variable service time distributions, also depending on the number
of replicas per task \cite{PC16,SLR16,WJW14,WJW15}.
In particular, for heavy-tailed service time distributions,
the mean aggregate time spent on a task may be considerably reduced
by virtue of the redundancy.
Indeed, even if the first replica to start service has an extremely long
service time, that is not likely to be case for the other clones as well.
In spite of the extremely long service time of the first replica,
it is therefore unlikely for the aggregate amount of time spent on the
task or its waiting time to be large.
This provides a significant performance benefit to redundancy-$d$
policies over JSQ($d$) policies, and has also motivated a strong interest
in adaptive replication schemes \cite{APS17,Joshi17a, Joshi17b}.

A further closely related model is where $k$ of the replicas need to
complete service, $1 \leq k \leq d$, in order for the task to finish 
which is relevant in the context of storage systems with coding
and MapReduce tasks \cite{JSW15a,JSW17}.
The special case where $k = d = N$ corresponds to a classical
fork-join system.
The authors of~\cite{HBVH19} present a unified approach
for analyzing the stability and performance of a broad class
of workload-dependent task assignment and replication policies
based on considering the so-called cavity process
in a many-server regime with $N \to \infty$.
This class of policies includes both versions of the redundancy-$d$
policy as well as the above-mentioned $k$-out-$d$ system.

\subsection{Conventional heavy traffic}
\label{classical}

In this subsection we briefly discuss a few asymptotic results for LBAs
in the classical heavy-traffic regime as described in Section~\ref{asym}
where the number of servers~$N$ is fixed and the relative load tends
to one in the limit.

The papers \cite{Foschini77,FS78,Reiman84,ZHW95} establish diffusion
limits for the JSQ policy in a sequence of systems with Markovian
characteristics as in our basic model set-up, but where in the $K$-th
system the arrival rate is $K \lambda + \hat\lambda \sqrt{K}$, while
the service rate of the $i$-th server is $K \mu_i + \hat\mu_i \sqrt{K}$,
$i = 1, \dots, N$, with $\lambda = \sum_{i = 1}^{N} \mu_i$,
inducing critical load as $K \to \infty$.
It is proved that for suitable initial conditions the queue lengths
are of the order O($\sqrt{K}$) over any finite time interval
and exhibit a state-space collapse property.
In particular, a properly scaled version of the joint queue length
process lives in a one-dimensional rather than $N$-dimensional space,
reflecting that the various queue lengths evolve in lock-step,
with the relative proportions remaining virtually identical in the limit,
while the aggregate queue length varies.

Atar {\em et al.}~\cite{AKM17} investigate a similar scenario,
and establish diffusion limits for three policies: the JSQ($d$) policy,
the redundancy-$d$ policy (where the redundant clones are abandoned
as soon as the first replica starts service), and a combined policy
called Replicate-to-Shortest-Queues (RSQ)
where $d$~replicas are dispatched to the $d$-shortest queues.
Note that the latter policy requires instantaneous knowledge of all
the queue lengths, and hence involves a similar excessive communication
overhead as the ordinary JSQ policy, besides the substantial
information exchange associated with the abortion of redundant copies.
Conditions are derived for the values of the relative service rates
$\mu_i$, $i = 1, \dots, N$, in conjunction with the diversity
parameter~$d$, in order for the queue lengths under the JSQ($d$)
and redundancy-$d$ policies to be of the order O($\sqrt{K}$)
over any finite time interval and exhibit state-space collapse.
The conditions for the two policies are distinct, but in both cases
they are weaker for larger values of~$d$, as intuitively expected.
While the conditions for the values of~$\mu_i$ depend on~$d$,
whenever they are met, the actual diffusion-scaled queue length
processes do not depend on the exact value of~$d$ in the limit,
showing a certain resemblance with the universality property
as identified in Section~\ref{ssec:powerd} for the conventional
large-capacity and Halfin-Whitt regimes.

The authors of~\cite{ZWTSS17} consider a slightly different model
set-up with a time-slotted operation, and identify a class~$\Pi$ of LBAs
that not only provide throughput-optimality (or maximum stability, i.e.,
keep the queues stable in a suitable sense whenever feasible to do so at all),
but also achieve heavy-traffic delay optimality, in the sense that
the properly scaled aggregate queue length is the same as that
in a centralized queue where all the resources are pooled
as the load tends to one.
As it turns out, the class~$\Pi$ includes JSQ($d$) policies with $d \geq 2$,
but does \emph{not} include the JIQ scheme, which tends to degenerate 
into a random assignment policy when idle servers are rarely available.
The authors further propose a threshold-based policy which has low
implementation complexity like the JIQ scheme, but \emph{does} belong
to the class~$\Pi$, and hence achieves heavy-traffic delay optimality.
A later paper~\cite{ZTS19} establishes both necessary and sufficient
conditions for threshold-based task assignment policies to achieve
heavy-traffic optimality in terms of mean delay.

\subsection{Non-degenerate slowdown}
\label{nondegenerate}

In this subsection we briefly discuss a few of the scarce asymptotic results
for LBAs in the so-called non-degenerate slow-down regime
described in Section~\ref{asym} where $N - \lambda(N) \to \gamma > 0$,
as the number of servers~$N$ grows large.
In a centralized queue the process tracking the evolution of the number
of waiting tasks, suitably accelerated and normalized by~$N$,
converges in this regime to a Brownian motion with drift $- \gamma$
reflected at zero as $N \to \infty$, as demonstrated in~\cite{Atar12}.
In stationarity, the number of waiting tasks, normalized by~$N$,
converges in this regime to an exponentially distributed random variable
with parameter~$\gamma$ as $N \to \infty$.
Hence, the mean number of waiting tasks must be at least of the order $N / \gamma$,
and the waiting time cannot vanish as $N \to \infty$ under any policy.

The authors of~\cite{GW19} characterize the diffusion-scaled queue
length process under the JSQ policy in this asymptotic regime.
They further compare the diffusion limit for the JSQ policy with that
for a centralized queue as described above as well as several LBAs
such as the JIQ scheme and a refined version called Idle-One-First (I1F),
where a task is assigned to a server with exactly one task if no idle
server is available and to a randomly selected server otherwise.

It is proved that the diffusion limit for the JIQ scheme is no longer
asymptotically equivalent to that for the JSQ policy in this asymptotic regime,
and the JIQ scheme fails to achieve asymptotic optimality in that respect,
as opposed to the behavior in the conventional large-capacity
and Halfin-Whitt heavy-traffic regimes discussed in Section~\ref{ssec:jiq}.
In contrast, the I1F scheme does preserve the asymptotic equivalence
with the JSQ policy in terms of the diffusion-scaled queue length process,
and thus retains asymptotic optimality in that sense.

These results provide further indication that the amount and accuracy
of queue length information needed to achieve asymptotic equivalence
with the JSQ policy depend not only on the scale dimension (e.g.~fluid
or diffusion), but also on the load regime.
Put differently, the finer the scale and the higher the load,
the more strictly one can distinguish various LBAs in terms
of the relative performance compared to the JSQ policy.

\subsection{Sparse-feedback regime}
\label{sparse}

As described in Section~\ref{ssec:jiq}, the JIQ scheme involves 
a communication overhead of at most one message per task, and yet
achieves optimal delay performance in the fluid and diffusion regimes.
However, even just one message per task may still be prohibitive,
especially when tasks do not involve big computational tasks,
but small data packets which require little processing.
In such situations the sheer message exchange in providing queue
length information may be disproportionate to the actual amount
of processing required.

Motivated by the above issues, \cite{BBL19} proposes and examines
a novel class of LBAs which also leverage memory at the dispatcher,
but allow the communication overhead to be seamlessly adapted
and reduced below that of the JIQ scheme.
Specifically, in the proposed schemes, the various servers provide
occasional queue status notifications to the dispatcher,
either in a synchronous or asynchronous fashion.
The dispatcher uses these reports to maintain queue estimates,
and forwards incoming tasks to the server with the lowest queue estimate.
The queue estimate for a server is incremented for every task
assigned, and set to the true queue length at update moments,
but never lowered in between updates.
Note that when the update frequency per server is $\delta$,
the number of messages per task is $d = \delta / \lambda$,
with $\lambda < 1$ denoting the arrival rate per server.

The results in~\cite{BBL19} demonstrate that the proposed schemes
markedly outperform JSQ($d$) policies with the same number
of $d \geq 1$ messages per task and they can achieve a vanishing
waiting time in the many-server limit when the update
frequency~$\delta$ exceeds $\lambda / (1 - \lambda)$.
In case servers only report zero queue lengths and suppress updates
for non-zero queues, the update frequency required for a vanishing
waiting time can in fact be lowered to just~$\lambda$,
matching the one message per task involved in the JIQ scheme.

From a scalability viewpoint, the most pertinent regime is $d < 1$
where only very sparse server feedback is required.
It is shown in~\cite{BBL19} that the proposed schemes then outperform
the corresponding sparsified versions of the JIQ scheme where idle
servers only provide notifications to the dispatcher with probability~$d$.
In order to further explore the performance for $d < 1$ in the
many-server limit, \cite{BBL19} investigates fluid limits for the
synchronous case as well as the asynchronous case with exponential
update intervals.
The fixed point of the fluid limit are leveraged to derive the
stationary queue length distribution as function of the update frequency.

Additionally, \cite{BBL19} examines the performance in the ultra-low
feedback regime where the update frequency~$\delta$ goes to zero,
and in particular establishes a somewhat counter-intuitive dichotomy.
In the synchronous case, the behavior of each of the individual
queues approaches that of a single-server queue with
a near-deterministic arrival process and exponential service times,
with the mean waiting time tending to a finite constant.
In contrast, in the asynchronous case, the individual queues
experience saw-tooth behavior with oscillations and waiting times
that grow without bound.

In order to achieve a vanishing waiting time, the dispatcher must assign each incoming task to an idle server with high probability, and thus be able to identify on average at least one idle server for every incoming task.  When the amount of memory at the dispatcher is limited, the dispatcher may in fact have to identify more idle servers on average to ensure that at least one is available with high probability for each incoming task, as also reflected in the results of~\cite{GTZ16,GTZ18,GTZ20}.  These conditions, in conjunction with the fact that the fraction of idle servers in equilibrium is $1 - \lambda$, translate into a minimum required communication overhead for various families of algorithms.  For example, if the dispatcher samples a server at random, it will find that server idle with probability $1 - \lambda$, so in the absence of any memory it will need to sample a number of servers that grows with $N$ for each incoming task, while with unlimited memory, it will need to sample on average $1 / (1 - \lambda)$ servers per incoming task.  Likewise, if servers report their queue status to the dispatcher, then an arbitrary server will report to be idle with probability $1 - \lambda$, so they all need to do that every $\lambda / (1 - \lambda)$ time units on average, i.e.,~$1 / (1 - \lambda)$ times on average per incoming task.  When only idle servers report their status to the dispatcher, as in the JIQ algorithm, they only need to do so at most once per incoming task.  When servers report their status asynchronously rather than all simultaneously, or idle servers only after some delay, the associated memory requirement at the dispatcher can be reduced.

\subsection{Scaling of maximum queue length}
\label{ballsbins}

So far we have focused on the asymptotic behavior of LBAs in terms
of the number of servers with a certain queue length, either on fluid scale
or diffusion scale, in various regimes as $N \to \infty$.
A related but different performance metric is the maximum queue length
$M(N)$ among all servers as $N \to \infty$.
The authors of~\cite{LM06} showed that for fixed $d \geq 2$
the stationary maximum queue length $M(N)$ in a system
under the JSQ($d$) policy is concentrated on at most two adjacent values
which are $\log(\log(N)) / \log(d) + \OO(1)$,
whereas for purely random assignment ($d = 1$),
it scales as $\log(N) / \log(1/\lambda)$
and does not concentrate on a bounded range of values.
This is yet a further manifestation of the power-of-choice effect.

An earlier paper~\cite{LM05} had already shown a similar result
for the maximum bin occupancy under a power-of-$d$ policy
in a balls-and-bins context where arriving items (balls) do not get
served and never depart but simply accumulate in bins,
so that (stationary) queue lengths are not meaningful.
The maximum bin occupancy under purely random assignment, however,
scales as $\log(N)/\log(\log(N))$, and \textit{does} concentrate on two
adjacent values, in contrast with the queueing scenario mentioned above.

In fact, the very notion of randomized load balancing and power-of-$d$
strategies was introduced in a balls-and-bins setting in the seminal
paper~\cite{ABKU94}.
Several further variations and extensions in that context have been considered
in \cite{V99,ACMR95,BCSV00,BCSV06, CMS95,DM93,FPSS05,PR04, P05, G81}.
One of the earliest papers on graph-based load balancing was also
concerned with a balls-and-bins setting~\cite{KP06}.

As alluded to above, there are natural parallels between the balls-and-bins
setup and the queueing scenario that we have focused on so far.
These commonalities are for example reflected in the fact that
power-of-$d$ strategies yield similar dramatic performance improvements
over purely random assignment in both settings.
However, there are also quite fundamental differences between the
balls-and-bins setup and the queueing scenario, besides the obvious
contrasts in the performance metrics.
This distinction is already reflected in the different scaling behavior
under purely random assignment of the maximum queue length in a queueing scenario
and the maximum bin occupancy in a balls-and-bins setting as mentioned above.
A further manifestation of is provided by the fact that a simple
Round-Robin strategy produces a perfectly balanced allocation in
a balls-and-bins setup but is far from optimal in a queueing scenario
as observed in Section~\ref{random}.
In particular, the stationary fraction of servers with two or more tasks
under a Round-Robin strategy remains positive in the limit as $N \to \infty$,
whereas it vanishes under the JSQ policy.
On a related account, since tasks get served and eventually depart
in a queueing scenario, less balanced allocations with a large
portion of vacant servers will generate fewer service completions
and result in a larger total number of tasks.
Thus different schemes yield not only various degrees of balance,
but also variations in the aggregate number of tasks in the system,
which is not the case in a balls-and-bins set-up.

\section{Extensions and future research directions}
\label{sec:ext}

Throughout most of the paper we have focused on the supermarket model
as a canonical setup and adopted several common assumptions in that context:
(i) all servers are identical;
(ii) the service requirements are exponentially distributed;
(iii) no advance knowledge of the service requirements is available;
(iv) in particular, the service discipline at each server is oblivious
to the actual service requirements.
As mentioned earlier, the stochastic optimality of the JSQ policy,
and hence its central role as an ideal performance benchmark,
critically rely on these assumptions.
The latter also broadly applies to the stochastic coupling techniques
and asymptotic universality properties that we have considered
in the previous sections.

In this section we turn to a brief overview of results for scenarios
where some of the above assumptions are relaxed, in particular allowing
for general service requirement distributions and possibly
heterogeneous servers, along with some broader methodological issues.
In Section~\ref{ssec:powerofd-gen} we focus on the behavior
of JSQ($d$) policies in such scenarios, mainly in the large-$N$ limit,
while also briefly commenting on the JIQ policy.
In Section~\ref{anticipating} we discuss strategies which specifically
exploit knowledge of server speeds or service requirements
of arriving tasks in making task assignment decisions, and may not
necessarily use queue length information, mostly in a fixed-$N$ regime.
While non-exponential service requirement distributions and heterogeneous
settings cover a major share of the extensions beyond the supermarket model,
there are also a plethora of further model variations that have been
considered in the literature.
An exhaustive listing is simply out of reach, but some notable examples
within the realm of scaling laws include \cite{LM15,LMU03,WXHB21,WW20}.

\subsection{JSQ(d) policies with general service requirement distributions}
\label{ssec:powerofd-gen}

The authors of \cite{FC91,FC98} use direct probabilistic methods
and fluid limits to obtain stability conditions for finite-size systems
with a renewal arrival process, a FCFS discipline at each server,
various state-dependent routing policies, including JSQ,
and general service requirement distributions, which may depend on the
task type, the server or both.
Using fluid limits as well as Lyapunov functions,
\cite{Bramson98,Bramson11} show that JSQ($d$) policies achieve
stability for any subcritical load in finite-size systems
with a renewal arrival process, identical servers, non-idling local
service disciplines and general service requirement distributions.
In addition, the author derives uniform bounds on the tails of the marginal
queue length distributions, and uses these to prove relative
compactness of these distributions.

The authors of~\cite{BLP10,BLP12} examine mean-field limits
for JSQ($d$) policies with generally distributed service requirements,
leveraging the above-mentioned tail bounds and relative compactness.
They establish that similar \emph{power-of-choice} benefits occur
as originally demonstrated for exponentially distributed service
requirements in~\cite{Mitzenmacher96,VDK96}, provided a certain `ansatz'
holds asserting that finite subsets of queues become independent
in the large-$N$ limit.
The latter `propagation of chaos' property is shown to hold in several
settings, e.g. when the service requirement distribution has
a decreasing hazard rate and the discipline at each server is FCFS
or when the service requirement distribution has a finite second
moment and the load is sufficiently low.
The ansatz also always holds for the power-of-$d$ version of the JSW
rather than JSQ policy, see Theorem 2.1 in~\cite{BLP12}.

It is further shown in \cite{BLP10,BLP12} that the arrival process
at any given server tends to a state-dependent Poisson process,
and that the queue length distribution becomes insensitive with respect
to the service requirement distribution when the service discipline
is either Processor Sharing or LCFS with preemptive resume.
This may be explained from the insensitivity property of queues
with state-dependent Poisson arrivals and symmetric service disciplines.

There are strong plausibility arguments that a similar asymptotic
insensitivity property should hold for the JIQ policy in a queueing
scenario, even if the discipline at each server is not symmetric
but FCFS for example.
So far, however, this has only been rigorously established for service
requirement distributions with decreasing hazard rate in~\cite{Stolyar15}.
This result was in fact proved for systems with heterogeneous server
pools, and was further extended in~\cite{Stolyar17} to systems
with multiple symmetric dispatchers.
As it turns out, general service requirement distributions
with an increasing hazard rate give rise to major technical challenges
due to a lack of certain monotonicity properties.
This has only allowed a proof of the asymptotic zero-wait property
for the JIQ policy for load values strictly below $1/2$ so far~\cite{FS17}.

A fundamental technical issue associated with any general service
requirement distribution is that the joint queue length no longer
provides a suitable state description, and that the state space
required for a Markovian description is no longer countable.
The authors of \cite{AR17,ALR17} introduce a particle representation
for the state of the system and describe the state dynamics
for a JSQ($d$) policy via a sequence of interacting measure-valued processes.
They prove that as $N$ grows large, a suitably scaled sequence
of state processes converges to a hydrodynamic limit
which is characterized as the unique solution of a countable system
of coupled deterministic measure-valued equations, i.e.,
a system of PDE rather than the usual ODE equations.
They also establish a `propagation of chaos' result, meaning that
finite collections of queues are asymptotically independent.

The authors of \cite{MM14,MM16,MKM16} analyzed the performance
and stability of static probabilistic routing strategies and power-of-$d$
policies in the large-$N$ limit in systems with exponential service
requirement distributions, but heterogeneous server pools
and a Processor-Sharing discipline at each server.
They also considered variants of the JSQ($d$) policy which account
for the server speed in the selection criterion as well as hybrid
combinations of the JSQ($d$) policy with static probabilistic routing.
Related results for heterogeneous loss systems rather than queueing
scenarios are presented in \cite{KMM17,MMG15,MKMG15}.
As the results in \cite{MM14,MM16} reflect, ordinary JSQ($d$) policies
may fail to sample the faster servers sufficiently often in such scenarios,
and therefore fail to achieve maximum stability,
let alone asymptotic optimality.
In~\cite{MKM16} a weighted version of JSQ($d$) policies is presented that
does provide maximum stability, without requiring any specific knowledge
of the underlying system parameters and server speeds in particular.

The authors of \cite{VMM17,VMM18,VMM19} examine mean-field limits
for power-of-$d$ policies in many-server loss systems as well
Processor-Sharing queues with phase-type service requirement distributions.
They observe that the fixed point suggests a similar insensitivity
property of the stationary occupancy distribution as mentioned above.
In view of the insensitivity of loss systems with possibly
state-dependent Poisson arrivals, this may be interpreted as an indirect
indication that the arrival process at any given server pool tends
to a state-dependent Poisson arrival process in the large-$N$ limit.
In a somewhat different strand of work, the authors of~\cite{JP16}
investigate the behavior of blocking probabilities in various load
regimes in systems with many single-server finite-buffer queues,
a Processor-Sharing discipline at each server, and an \emph{insensitive}
routing policy.

\subsection{Heterogeneous servers and knowledge of service requirements}
\label{anticipating}

The bulk of the literature has focused on systems with identical servers,
and scenarios with non-identical server speeds have received
relatively limited attention.
A natural extension of the JSQ policy is to assign jobs to the server
with the normalized shortest queue length, or equivalently,
assuming exponentially distributed service requirements, the shortest
expected delay.
While such a Generalized JSQ (GJSQ) or Shortest Expected Delay (SED)
strategy tends to perform well~\cite{BZ89}, it is not strictly optimal
in general~\cite{EVW80}, and the true optimal strategy may in fact
have a highly complicated structure.

The authors of~\cite{SAK16} present approximations for the performance
of GJSQ policies in a fixed-$N$ regime with generally distributed
service requirements and a Processor-Sharing discipline at each server,
extending the analysis in~\cite{GHSW07} for the ordinary JSQ policy
with homogeneous servers.
In~\cite{HM19} necessary and sufficient conditions are established
for JSQ($d$) policies to be optimal in systems with heterogeneous
server speeds in a classical heavy-traffic regime.
The authors of~\cite{GS19} examine fluid limits for a system
with both `fast' and `slow' servers and task assignment policies which
may only have limited knowledge of the speeds of individual servers.
In~\cite{CBL19} and~\cite{YHH18} fluid limits and heavy-traffic limits
are investigated, respectively, for a somewhat related system
where speeds are also heterogeneous but depend on the combination
of the server and the task type due to affinity relations.

In a separate line of work, the authors of~\cite{FMR05} consider
static dispatching policies in a fixed-$N$ regime with heterogeneous
servers and a FCFS or Processor-Sharing discipline at each server.
The assignment decision may depend on the service requirement of the
arriving task, but not on the actual queue lengths or any other
state information.
In case of FCFS the optimal routing policy is shown to have a nested size
interval structure, generalizing the strict size interval structure
of the task assignment strategies in~\cite{HCM99} which are optimal
for homogeneous servers.
In case of Processor Sharing, the knowledge of the service requirements
of arriving tasks is irrelevant, in the absence of any state information.

The authors of~\cite{AAP11} consider static probabilistic routing
policies in a somewhat similar setup of a fixed-$N$ regime
with multiple task types, servers with heterogeneous speeds,
and a Processor-Sharing discipline at each server.
The routing probabilities are selected so as to either minimize the
global weighted holding cost or the expected holding cost for
an individual task, and may depend on the type of the task and its
service requirement, but not on any other state information.

When knowledge of the service requirements of arriving tasks is available,
it is natural to exploit that for the purpose of local scheduling
at the various servers, and for example use size-based disciplines.
The impact of the local scheduling discipline and server heterogeneity
on the performance and degree of efficiency of load balancing strategies
is examined in~\cite{CMW09}.
The authors of~\cite{GSHB19} show that any given dispatching policy
can be augmented with a so-called `guardrails' feature to ensure
minimization of the mean delay in a classical heavy-traffic regime
in systems where the local scheduling is governed by the
Shortest-Remaining Processing Time (SRPT) policy.
An interesting broader issue concerns the relative benefits provided
by exploiting knowledge of service requirements of arriving tasks versus
using information on queue lengths or workloads at the various servers,
which strongly depend on the service requirement distribution~\cite{HSY09}.

\subsection{Open problems and emerging research directions}

If we now return to scalable load balancing as the central theme
of this survey, and consider the above-described extensions in that light,
it is striking how scant the results are if any of the assumptions
(i)-(iv) as stated at the beginning of Section~\ref{asym} are dropped.
On further thought, the paucity of results from a scalability viewpoint
is perhaps not so surprising since it is not even clear what the optimal
achievable (delay) performance is in the absence of these assumptions,
leaving aside any trade-off with communication overhead.

The graph-based load balancing scenario considered in Section~\ref{networks}
moves beyond assumption~(i) of all servers being identical
as it entails that different incoming tasks can only be served by different
subsets of the servers.
Thus, it is not clear what the optimal assignment policy is,
but since the server speeds are still homogeneous, it can be argued
that the JSQ policy provides a bound for the achievable performance.
The results obtained in~\cite{MBL17} as reviewed in Section~\ref{networks}
establish suitable conditions in terms of the graph for that lower bound
to be asymptotically achievable.

Important extensions of these results are presented in \cite{RM21,WZS20}
which allow for more general compatibility constraints between different
task types and different servers represented in terms of a bipartite graph,
and examine conditions in terms of the latter graph for the achievable performance
to be asymptotically equivalent to that in case of full compatibility.
Informally speaking, both papers establish conditions in terms of the
connectivity properties of the bipartite compatibility graph for similar
performance to be achievable as in a fully flexible system.
More specifically, the authors of~\cite{RM21} focus on scenarios with identical
server speeds and uniform loads across the various job types,
and establish process-level limits indicating convergence of the system
occupancy under JSQ policies to that in the supermarket model with full flexibility.
The authors of~\cite{WZS20} allow for heterogeneous server speeds and arbitrary
load distributions, and use drift methods to prove bounds and demonstrate
that speed-aware extensions of the JSQ and JIQ strategies achieve vanishing
waiting times and minimum expected response times.
Interestingly, the results in \cite{RM21,WZS20} also entail a certain notion
of universality as in~\cite{MBL17}, with similar achievable performance
as in a fully flexible system under relatively sparse compatibility relations.
An open question is what the associated communication overhead
is with these policies, and whether that is close in any sense to the minimum
communication overhead required for asymptotically optimal delay performance.

A further extension of the above two models is where the service rates can depend
in an arbitrary way on the pairwise combination of the task and the server.
In that case it is also open what the minimum required overhead is to achieve
asymptotically optimal performance, and it even remains to be established
what the asymptotically optimal performance is.

Both these questions are also largely open for non-exponential service requirement
distributions, even in the absence of any compatibility constraints.
It is evident that for nearly deterministic service requirements,
a zero mean waiting time can be achieved without any communication overhead
at arbitrarily high sub-critical load (using open-loop policies such as Round Robin).
It might thus be natural to expect that for extremely variable service
requirements correspondingly high communication overhead is needed to achieve
a zero mean waiting time.
However, this is countered by the asymptotic insensitivity of the JIQ policy
which has been proven for service requirement distributions
with decreasing hazard rate as mentioned earlier.
Also, the amount of communication overhead can in fact be reduced
by \textit{not} issuing messages when a server is busy at pre-defined time
instants rather than sending messages when a server is idle~\cite{BZB20}.
All in all, it seems largely open exactly how the amount of communication
overhead required for vanishing waiting time depends on the service requirement
distribution in conjunction with the system load.

Finally, throughout we have tacitly assumed that each task involves a single
processing operation that can be handled by a single server.
In reality however, tasks can have a highly complex structure
and consist of several sub-tasks that can be executed by multiple servers
subject to certain precedence constraints,
see for instance~\cite{HarcholBalter21} for references and further background.
The above questions also seem totally open in these scenarios.

{\small
\bibliographystyle{abbrv}
\bibliography{bibliography}
}

\end{document}